\documentclass[reqno]{amsart}

\usepackage{amsfonts,amssymb, amscd, latexsym, graphicx, psfrag, color,float}
\usepackage[normalem]{ulem}

\usepackage{dbnsymb}

\usepackage[all, knot, poly]{xy}

\xyoption{knot}

\usepackage{a4wide}

\usepackage{hyperref}
\usepackage[mathscr]{euscript}
\usepackage{mathtools}
\usepackage{subfigure}
\usepackage{calc}
\usepackage{tikz}

\usepackage{ifthen}
\usetikzlibrary{cd,arrows,matrix}
\usepackage{newtxtext}
\usepackage{newtxmath}
\usepackage{anyfontsize}
\usepackage{bm}

\usepackage[english]{babel}

\usepackage[makeroom]{cancel}

\hypersetup{colorlinks=true,linkcolor=blue}

\usepackage{enumitem}
\usepackage{calligra}

\usepackage{mathrsfs}

\usepackage{bbm}

\numberwithin{equation}{subsection}
\numberwithin{figure}{subsection}

\newtheorem{dummy}{dummy}[section]
\newtheorem{lemma}[dummy]{Lemma}
\newtheorem{theorem}[dummy]{Theorem}

\newtheorem{conjecture}[dummy]{Conjecture}
\newtheorem{corollary}[dummy]{Corollary}
\newtheorem{proposition}[dummy]{Proposition}

\newtheorem{definition/proposition}[dummy]{Definition/Proposition}
\theoremstyle{definition}
\newtheorem{definition}[dummy]{Definition}
\newtheorem{example}[dummy]{Example}
\newtheorem{remark}[dummy]{Remark}

\newtheorem{assumption}[dummy]{Assumption}
\DeclareMathOperator{\diag}{Diag}

\DeclareMathOperator{\im}{Im}

\DeclareMathOperator{\Gr}{Gr}

\def\acts{\curvearrowright}

\newcommand{\mB}{\mathrm{B}}

\newcommand{\mK}{\mathrm{K}}

\newcommand{\mH}{\mathrm{H}}

\newcommand{\mM}{\mathrm{M}}

\newcommand{\ms}{\mathrm{s}}

\newcommand{\cA}{\mathcal{A}}

\newcommand{\cB}{\mathcal{B}}

\newcommand{\cC}{\mathcal{C}}

\newcommand{\cV}{\mathcal{V}}

\newcommand{\cI}{\mathcal{I}}
\newcommand{\cJ}{\mathcal{J}}

\newcommand{\cN}{\mathcal{N}}

\newcommand{\cO}{\mathcal{O}}

\newcommand{\cP}{\mathcal{P}}

\newcommand{\cS}{\mathcal{S}}

\newcommand{\cW}{\mathcal{W}}

\newcommand{\id}{\mathsf{id}}

\newcommand{\fp}{\mathfrak{p}}

\newcommand{\fr}{\mathfrak{r}}

\renewcommand{\emptyset}{\varnothing}

\renewcommand{\hat}{\widehat}
\renewcommand{\tilde}{\widetilde}

\newcommand{\boundary}{\partial}

\newcommand{\field}{\mathbb{K}}

\newcommand{\isomorphic}{\cong}

\newcommand{\mumon}{\mathfrak{\mu mon}}

\newcommand{\Spec}{\mathrm{Spec}}

\newcommand{\modulispace}{\mathcal{M}}

\newcommand{\catquotient}{/^{{\scriptscriptstyle\mathrm{ca}}}}

\newcommand{\goodquotient}{/^{{\scriptscriptstyle\mathrm{go}}}}

\newcommand{\geoquotient}{/^{{\scriptscriptstyle\mathrm{ge}}}}

\newcommand{\pa}{\mathrm{par}}

\newcommand{\permutation}{\mathrm{s}}

\newcommand{\FBD}{\mathrm{FBD}}

\newcommand{\bmdp}{\mathbf{bmdp}}

\newcommand{\BD}{\mathrm{BD}}

\newcommand{\bmd}{\mathbf{bmd}}

\newcommand{\weakequivalent}{\overset{\text{\tiny$\mathfrak{w}$}}{\sim}}

\newcommand{\braidequivalent}{\overset{\text{\tiny$\mathfrak{b}$}}{\sim}}

\newcommand{\notbraidequivalent}{\overset{\text{\tiny$\mathfrak{b}$}}{\nsim}}

\newcommand{\FBr}{\mathrm{FBr}}

\newcommand{\Br}{\mathrm{Br}}

\newcommand{\type}{\bm{\mu}}

\makeatletter
\newcommand{\setword}[2]{%
  \phantomsection
  #1\def\@currentlabel{\unexpanded{#1}}\label{#2}%
}
\makeatother
\newcommand{\Tr}{\mathrm{Tr}}

\newcommand{\et}{\mathrm{\text{\'{e}}t}}

\newcommand{\Supp}{\mathrm{Supp}}

\newcommand{\pt}{\mathrm{pt}}

\newcommand{\weight}{\mathrm{W}}

\newcommand{\Dol}{\mathrm{Dol}}

\newcommand{\NAH}{\mathrm{NAH}}

\newcommand{\lbracket}{{\color{blue}\bm{(}}}
\newcommand{\rbracket}{{\color{blue}\bm{)}}}

\newcommand{\sm}{\mathrm{sm}}
\newcommand{\sing}{\mathrm{sing}}

\newcommand{\ev}{\mathrm{ev}}

\newcommand{\trdeg}{\operatorname{trdeg}}

\newcommand{\Frac}{\operatorname{Frac}}

\begin{document}

\title[Cell decomposition and dual boundary complexes of character varieties]{Cell decomposition and dual boundary complexes of character varieties}

\author[T. Su]{Tao Su}
\email{sutao@bimsa.cn}
\address{Beijing Institute of Mathematical Sciences and Applications}

\dedicatory{Dedicated to the memory of my high school mathematics teacher, Yuwen Zhu (1975--2023)}
\date{}

\begin{abstract}

The weak geometric P=W conjecture of Katzarkov--Noll--Pandit--Simpson asserts that for any smooth Betti moduli space $\modulispace_B$ of complex dimension $d$ associated to a punctured Riemann surface, the dual boundary complex $\mathbb{D}\partial\modulispace_B$ is homotopy equivalent to a $(d-1)$-sphere. Here, $\modulispace_B$ is a generic $GL_n(\mathbb{C})$-character variety of genus $g$, with prescribed semisimple local monodromies at $k$ punctures.

In this article, we prove the weak geometric P=W conjecture for all \emph{very generic} $\modulispace_B$, namely when at least one conjugacy class is regular semisimple. A key step is a strengthening of Mellit's cell decomposition theorem: we decompose $\modulispace_B$ itself (without passing to a vector bundle) into locally closed subvarieties which are stably isomorphic to $(\mathbb{C}^{\times})^{d-2b}\times \mathbb{C}^b$.
A second ingredient is a motivic characterization of the integral cohomology of dual boundary complexes, developed in \cite{Su24}. Following Simpson's strategy, the proof reduces to an inductive computation of dual boundary complexes arising from this cell decomposition.

\end{abstract}

\maketitle

\tableofcontents

\section*{Introduction}

Let $\Sigma$ be a genus $g$ closed Riemann surface with $k$ punctures $\sigma=\{q_1,\cdots,q_k\}$, where $k\geq 1$ and $2g+k\geq 3$, and let $G=GL_n(\mathbb{C})$.
Modulo extra input, the tame nonabelian Hodge correspondence over noncompact curves \cite{Sim90,Kon93} induces a diffeomorphism
\[
\NAH:\modulispace_{\Dol}\simeq \modulispace_B
\]
between two moduli spaces: the Dolbeault moduli space $\modulispace_{\Dol}$ of stable filtered regular (parabolic) $G$-Higgs bundles on $(\Sigma,\sigma)$ with parabolic degree $0$, and the Betti moduli space $\modulispace_B$ of stable filtered $G$-local systems on $\Sigma\setminus\sigma$ with parabolic degree $0$.
For background on nonabelian Hodge theory, see \cite{Cor88,Sim92,Sim94,Sim95,BB04,Moc11,Moc21,HKSZ22,HS22}.

The geometric P=W conjecture of L.~Katzarkov, A.~Noll, P.~Pandit, and C.~Simpson \cite{KNPS15,Sim16} predicts that, under $\NAH$, the ``Hitchin fibration at infinity'' of $\modulispace_{\Dol}$ matches, up to homotopy, a ``fibration at infinity'' intrinsic to the algebraic variety $\modulispace_B$.
More concretely, on the Dolbeault side, the Hitchin fibration $h:\modulispace_{\Dol}\to\mathbb{A}$ induces
\[
\overline{h}:N_{\Dol}^*=\modulispace_{\Dol}\setminus h^{-1}(B_R(0))
\xrightarrow[]{h} \mathbb{A}\setminus B_R(0)\to (\mathbb{A}\setminus B_R(0))/\text{scaling}= S^{d-1},
\qquad R\gg 0,\ d=\dim\modulispace_{\Dol}.
\]
On the Betti side, there is a fibration
\[
\ev: N_B^*\rightarrow \mathbb{D}\partial\modulispace_B,
\]
well defined up to homotopy.
Here, we fix a log compactification $\overline{\modulispace}_B$ of $\modulispace_B$ with simple normal crossing boundary divisor $\partial\modulispace_B$.
Then $N_B^*$ is a punctured tubular neighborhood of $\partial\modulispace_B$ in $\overline{\modulispace}_B$, and $\mathbb{D}\partial\modulispace_B$ is the dual boundary complex of $\modulispace_B$, namely the dual complex of the irreducible components of $\partial\modulispace_B$.
Note that the dual boundary complex is defined for any smooth quasi-projective variety; see Definition \ref{def:dual_boundary_complex}.
For the moment, we suppress the definition of $\ev$; see Remark \ref{rem:fibration_at_infinity_via_dual_boundary_complex} for details.

Then:

\begin{conjecture}[{\cite{KNPS15,Sim16}}, Geometric P=W]\label{conj:geometric_P=W}
There is a homotopy commutative square
\[
\begin{tikzcd}[row sep=1pc]
N_{\Dol}^*\arrow{r}{\simeq}[swap]{\NAH}\arrow{d}{\overline{h}} & N_B^*\arrow{d}{\ev}\\
S^{d-1}\arrow{r}{\simeq} & \mathbb{D}\boundary\modulispace_B.
\end{tikzcd}
\]
\end{conjecture}
As a weak form of the geometric P=W conjecture, one obtains in particular:

\begin{conjecture}[{\cite{KNPS15,Sim16}}, weak geometric P=W conjecture]\label{conj:homotopy_type_conjecture}
The dual boundary complex $\mathbb{D}\boundary\modulispace_B$ is homotopy equivalent to the sphere $S^{d-1}$, where $d=\dim_{\mathbb{C}}\modulispace_B$.
\end{conjecture}

According to a folklore conjecture, all smooth Betti moduli spaces (that is, character varieties) $\modulispace_B$ are expected to be log Calabi--Yau (\textbf{CY}). This has been verified for $G=SL_2(\mathbb{C})$ \cite{Wha20}.
The weak geometric P=W conjecture is therefore closely related to the following conjecture:

\begin{conjecture}[Kontsevich(-Koll\'ar-Xu)]\label{conj:Kontsevich_conjecture}
The dual boundary complex of a log \textbf{CY} variety is (a finite quotient of) a sphere.
\end{conjecture}

So far, the only known general result is due to J.~Koll\'ar and C.~Xu \cite{KX16}: if $X$ is log Calabi--Yau of dimension at most $5$, then $\mathbb{D}\partial X$ is a finite quotient of a sphere.

The geometric P=W conjecture was originally motivated by, and aimed at, a geometric interpretation of the cohomological P=W conjecture of M.~de Cataldo, T.~Hausel, and L.~Migliorini \cite{dCHM12}.
The latter asserts that $\NAH$ exchanges the weight filtration on $H^*(\modulispace_B,\mathbb{Q})$ with the perverse filtration on $H^*(\modulispace_{\Dol},\mathbb{Q})$:
\[
\NAH^*\bigl(W_{2k}H^*(\modulispace_B,\mathbb{Q})=W_{2k+1}H^*(\modulispace_B,\mathbb{Q})\bigr)
=
P_kH^*(\modulispace_{\Dol},\mathbb{Q}).
\]
After the results for rank $2$ \cite{dCHM12} and, separately, for genus $2$ \cite{CMS22}, the cohomological P=W conjecture has now been established independently by three groups \cite{MS22,HMMS22,MSY23} for the major case of twisted $GL_n(\mathbb{C})$-character varieties (in particular, for $k=1$).
See also \cite[Question 4.1.7]{dCM20}, \cite[Conj.B]{Dav23}, \cite{Mau21,FM22}, and \cite[Conj.4.2.7]{MMS22} for extensions of the P=W conjectures to singular or stacky character varieties.

The geometric P=W conjecture recovers the cohomological P=W conjecture for the top weight \cite[Thm.6.2.6]{MMS22}.
In general, however, it is not sufficient to imply the latter \cite[Rmk.6.2.11]{MMS22}.
Nevertheless, the geometric P=W conjecture contains information beyond the cohomological P=W conjecture.
For example, the latter concerns only cohomology with rational coefficients, and hence captures only $H^*(\mathbb{D}\partial\modulispace_B,\mathbb{Q})$ via the identification (see, for example, \cite{Pay13})
\[
\tilde{H}_{i-1}(\mathbb{D}\partial \modulispace_B,\mathbb{Q})
\isomorphic
\Gr_{2d}^WH^{2d-i}(\modulispace_B,\mathbb{Q}),
\qquad d=\dim_{\mathbb{C}}\modulispace_B.
\]
By contrast, the geometric statement also sees $\tilde{H}^*(\mathbb{D}\partial\modulispace_B,\mathbb{Z})$.

As explained above, a refinement of the geometric P=W conjecture is needed in order to recover the cohomological P=W conjecture in all weights.
Let us make a complementary remark.
Assuming the folklore conjecture, one may restrict to (dlt) log \textbf{CY} compactifications.
Then one obtains a \emph{refined dual boundary complex}
$
\mathrm{DMR}(\overline{\modulispace}_B,\partial\modulispace_B),
$
which is well defined up to PL-homeomorphism \cite[Prop.11]{dFKX17}.
By Conjecture \ref{conj:homotopy_type_conjecture}, this is a PL-sphere of dimension $d-1$.
It is expected \cite[\S1.2]{Sim16} that this PL-sphere is closely related to the Kontsevich--Soibelman picture \cite{KS14}: the Kontsevich--Soibelman chambers in the Hitchin base $\mathbb{A}$ of $\modulispace_{\Dol}$ should correspond to the cells in $\mathrm{DMR}(\overline{\modulispace}_B,\partial\modulispace_B)$.

We have not yet said anything about the current state of the geometric P=W conjecture.
The full geometric P=W conjecture is known for the Painlev\'e cases \cite{NS22,Sza19,Sza21}, and for the cases $(g,k)=(1,0)$ and $(k,n)=(0,1)$ \cite[Thm.B]{MMS22}.
Our main interest in this paper is its weak form, namely Conjecture \ref{conj:homotopy_type_conjecture}.
Previously, this was known only in a few cases: $G=SL_2(\mathbb{C})$ \cite{Kom15,Sim16,FF23}; the Painlev\'e cases above; singular character varieties of arbitrary rank with $g=1$ and $k=0$ \cite{MMS22}; and smooth wild character varieties of arbitrary rank with $g=0$ and $k=1$ \cite{Su21}.

\subsection*{Results}
\addtocontents{toc}{\protect\setcounter{tocdepth}{1}}

As a complement to the discussion above, our main result is the following.

\begin{theorem}[{Theorem \ref{thm:weak_geometric_P=W_conjecture_for_very_generic_character_varieties}}]\label{thm:main_result}
Let $(\Sigma,\sigma=\{q_1,\cdots,q_k\})$ be a $k$-punctured genus $g$ Riemann surface, and let $\modulispace_B=\modulispace_{\type}$ be its $G$-character variety of type $\type\in\cP_n^k$.
If $\modulispace_{\type}$ is \emph{very generic}, then the weak geometric P=W conjecture \ref{conj:homotopy_type_conjecture} holds for $\modulispace_{\type}$.
\end{theorem}

\noindent
Here, $\modulispace_{\type}$ is a $GL_n(\mathbb{C})$-character variety on $\Sigma\setminus\sigma$ with fixed semisimple conjugacy class $\cC_i$ around $q_i$;
$\type=(\mu^1,\cdots,\mu^k)$, where $\mu^i=(\mu_1^i\geq\mu_2^i\geq\cdots)\in\cP_n$ records the multiplicities of the eigenvalues of $\cC_i$.
Here, `\emph{generic}' means that $(\cC_1,\cdots,\cC_k)$ is generic in the sense of \cite[Def.2.1.1]{HLRV11} (see Definition \ref{def:generic_monodromy}); and `\emph{very generic}' means, in addition, that $\cC_k$ is regular semisimple (Assumption \ref{ass:very_generic_assumption}).

See Remark \ref{rem:weak_P=W_for_generic_character_varieties} for a discussion of the case where $\modulispace_{\type}$ is only generic.

To prove the main result, we strengthen A.~Mellit's cell decomposition \cite[\S.7]{Mel25}, which applies only to a vector bundle over very generic $\modulispace_{\type}$.
More precisely, our second main result answers Mellit's question in \cite[\S1.4]{Mel25}: we construct a genuine cell decomposition of $\modulispace_{\type}$ itself.

\begin{theorem}[{Theorem \ref{thm:cell_decomposition_of_very_generic_character_varieties}}]\label{thm:main_result_cell_decomposition}
Any very generic $\modulispace_{\type}$ admits a cell decomposition
\[
\modulispace_{\type}=\sqcup_{(\vec{w},p)\in \cW^*} \modulispace_{\type}(\vec{w},p),
\]
where each $\modulispace_{\type}(\vec{w},p)$ is a locally closed affine subvariety such that
\[
\modulispace_{\type}(\vec{w},p) \times \field^{\dim U} \isomorphic (\field^{\times})^{\overline{a}(\vec{w},p)}\times \field^{b(\vec{w},p)}.
\]
Here $U\subset G$ is the subgroup of unipotent upper triangular matrices, and
letting $\overline{b}(\vec{w},p):=b(\vec{w},p) - \dim U$,
\[
\overline{a}(\vec{w},p)+2\overline{b}(\vec{w},p)=d_{\type}=\dim_{\mathbb{C}}\modulispace_{\type}
\]
is constant.
Moreover, there exists a unique $(\vec{w}_{\max},p_{\max})$ such that
$\dim\modulispace_{\type}(\vec{w}_{\max},p_{\max})=d_{\type}$.
Equivalently,
$\overline{a}(\vec{w}_{\max},p_{\max})=d_{\type}$, $\overline{b}(\vec{w}_{\max},p_{\max})=0$.
In particular, $\modulispace_{\type}(\vec{w}_{\max},p_{\max})$ is an open dense algebraic torus:
\[
\modulispace_{\type}(\vec{w}_{\max},p_{\max})\isomorphic (\field^{\times})^{d_{\type}}.
\]
\end{theorem}
\noindent\textbf{Note.} We do not know that whether or not $\modulispace_{\type}(\vec{w},p) \isomorphic (\field^{\times})^{\overline{a}(\vec{w},p)}\times \field^{\overline{b}(\vec{w},p)}$  in general. Such a question is related to the Zariski cancellation problem in characteristic zero.
See Remark \ref{rem:Zariski_cancellation_problem} for further discussion.

\iffalse

$\cA_{\type}(\vec{w},p)$ is stably isomorphic to $\mathbb{A}^{\overline{b}(\vec{w},p)}$.
We expect that, in general, $\cA_{\type}(\vec{w},p)$ is not isomorphic to $\mathbb{A}^{\overline{b}(\vec{w},p)}$.
If so, this would provide counterexamples to the Zariski cancellation problem in characteristic zero for dimensions $b=\overline{b}(\vec{w},p)\geq 3$.
See Remark \ref{rem:Zariski_cancellation_problem} for further discussion.

\fi

As in \cite[\S.7]{Mel25}, Theorem \ref{thm:main_result_cell_decomposition} is proved via a connection with braid varieties.
However, a more careful analysis is required, and for clarity we give a self-contained proof.
In fact, we use a somewhat different language, namely a diagram calculus of matrices.

To deduce Theorem \ref{thm:main_result} from Theorem \ref{thm:main_result_cell_decomposition}, the key idea is to apply a \textbf{remove/reduction lemma} (Lemma \ref{lem:dual_boundary_complex_remove_lemma}):
if $X$ is a connected smooth quasi-projective variety, and $Z\subset X$ is a smooth irreducible closed subvariety with nonempty open complement $U$, such that $\mathbb{D}\partial Z$ is contractible, then there is a homotopy equivalence
\[
\mathbb{D}\partial X\sim\mathbb{D}\partial U.
\]
We can apply this lemma because of a key property \cite[Cor.0.3]{Su24} (see also Lemma \ref{lem:dual_boundary_complex_of_a_stably_affine_space_is_contractible}): if $\cB$ is stably isomorphic to $(\field^{\times})^r\times \field^s$ with $s>0$, then $\mathbb{D}\partial \cB$ is contractible.
This is proved using a motivic characterization of the integral cohomology of dual boundary complexes, recalled in Proposition \ref{prop:dual_boundary_complex_is_motivic}.

Finally, we note that the same strategy can also be applied to wild character varieties \cite{Boa14,BY15}.
This will be pursued elsewhere.

\subsection*{Organization}
\addtocontents{toc}{\protect\setcounter{tocdepth}{1}}

We have already explained the main ideas above; we now sketch the organization of the paper.

In Section \ref{sec:set_up}, we set up the basic notions related to character varieties.
To clarify the computations in Section \ref{sec:cell_decomposition_of_character_varieties}, we introduce a diagram calculus of matrices in Section \ref{subsec:diagram_calculus_of_matrices}.
Morally, this is a language of braid matrix diagrams generalizing the braid matrices associated with positive braids.
In Section \ref{subsec:braid_varieties}, we review braid varieties, complemented by Appendix \ref{sec:cell_decomposition_of_braid_varieties}.

In Section \ref{sec:cell_decomposition_of_character_varieties}, we prove a strong form of the cell decomposition for $\modulispace_{\type}$ (Theorem \ref{thm:cell_decomposition_of_very_generic_character_varieties}).
This involves some routine diagram calculus in Sections \ref{subsec:diagram_calculus_for_punctures}--\ref{subsec:diagram_calculus_for_genera}.
In Sections \ref{subsec:connection_to_braid_varieties}--\ref{subsec:the_cell_decomposition}, we also use results on quotients of varieties from Appendix \ref{sec:quotients_of_varieties}.
In Section \ref{subsec:examples_of_cell_decomposition_of_very_generic_character_varieties}, we illustrate Theorem \ref{thm:cell_decomposition_of_very_generic_character_varieties} with an example.

In Section \ref{sec:dual_boundary_complexes_of_character_varieties}, we study dual boundary complexes of character varieties.
Section \ref{subsec:dual_boundary_complexes} reviews the basics of dual boundary complexes.
In Section \ref{subsec:dual_boundary_complexes_of_very_generic_character_varieties}, we prove the weak geometric P=W conjecture (Theorem \ref{thm:weak_geometric_P=W_conjecture_for_very_generic_character_varieties}).
To conclude the paper, we add several remarks on further directions in Section \ref{subsec:further_directions}.

\addtocontents{toc}{\protect\setcounter{tocdepth}{2}}

\section{Setup}\label{sec:set_up}

Fix an algebraically closed base field $\field$ of characteristic $0$; for simplicity, we take $\field=\mathbb{C}$.
A \emph{$\field$-variety} means a reduced separated scheme of finite type over $\field$.

\noindent\textbf{Convention} \setword{{\color{blue}$1$}}{convention:quotients}:
let $H$ be a linear algebraic group acting on a variety $X$ over $\field$.
\begin{itemize}[wide,labelwidth=!,labelindent=0pt]
\item
A principal $H$-bundle (or fiber bundle) is understood in the étale topology, unless stated otherwise.

\item
If $H$ is reductive and $X$ is affine, then
\[
X//H=\Spec~\cO_X(X)^H
\]
denotes the affine GIT quotient.

\item
$[X/H]$ denotes the quotient stack of $X$ by $H$.

\item
If $H$ acts freely on $X$, and $\pi:X\to Y$ is a principal $H$-bundle over a $\field$-variety, then $\pi$ is also a geometric quotient (see \cite[Def.3.27]{Hos16} and Proposition \ref{prop:associated_fiber_bundles}). In this case, we write
\[
X/H:=Y.
\]
Equivalently, the quotient stack $[X/H]$ is representable by $Y$, so we may also identify
\[
[X/H]=X/H.
\]
\end{itemize}

We refer to Appendix \ref{sec:quotients_of_varieties} for further background on various quotients of varieties.

\subsection{Generic character varieties}\label{subsec:generic_character_varieties}

Recall that $(\Sigma,\sigma=\{q_1,\cdots,q_k\})$ is a $k$-punctured genus $g$ Riemann surface, where $k\geq 1$ and $2g+k\geq 3$.
Let $T\subset G=GL_n(\field)$ be the diagonal maximal torus.
Let $(C_1,\cdots,C_k)\in T^k$ be semisimple elements of type
\[
\type:=(\mu^1,\cdots,\mu^k)\in\cP_n^k.
\]
That is, the multiplicities of the eigenvalues of $C_i$ define a partition
\[
\mu^i=(\mu_1^i,\cdots,\mu_{r_i}^i)\in\cP_n.
\]

Let
\[
\modulispace_{\type}
=
\modulispace_B
=
\modulispace_B(\Sigma,\sigma,G;C_1,\cdots,C_k)
\]
be the character variety of $G$-local systems on $\Sigma\setminus\sigma$ whose local monodromy around $q_i$ is conjugate to $C_i$.
More precisely, define the affine $\field$-variety
\[
M_B
=
M_B(\Sigma,\sigma,G;C_1,\cdots,C_k)
:=
\{(A_j)_{j=1}^{2g},x_1,\cdots,x_k)\in G^{2g+k}:
\prod_{j=1}^g \lbracket A_{2j-1}, A_{2j} \rbracket
\prod_{i=1}^k x_iC_ix_i^{-1}=\id\},
\]
where $\lbracket -,- \rbracket$ denotes the multiplicative commutator.
It carries an action of the reductive group
\[
G_{\pa}:=G\times \prod_{i=1}^k Z(C_i),
\qquad
Z(C_i)=\text{the centralizer of }C_i,
\]
given by
\[
(h_0,h_1,\cdots,h_k)\cdot(A_1,\cdots,A_{2g},x_1,\cdots,x_k)
:=
(h_0A_1h_0^{-1},\cdots,h_0A_{2g}h_0^{-1},h_0x_1h_1^{-1},\cdots,h_0x_kh_k^{-1}).
\]
Then the diagonal $\field^\times$ acts trivially on $M_B$, and
\[
\modulispace_{\type}:=M_B//G_{\pa}
\]
is the affine GIT quotient.

\begin{definition}[{\cite[Def.4.6.1]{Mel25}}]\label{def:generic_monodromy}
$(C_1,\cdots,C_k)\in T^k$ is \emph{generic} if
$\prod_{i=1}^k\det C_i=1$,
and for every integer $1\leq n'<n$, if one chooses any $n'$ eigenvalues
$\alpha_{i,1},\cdots,\alpha_{i,n'}$
of each $C_i$, then
$
\prod_{i=1}^k\prod_{j=1}^{n'} \alpha_{i,j}\neq 1.
$
In this case, $\modulispace_{\type}$ is called a \emph{generic character variety}.
\end{definition}

\begin{lemma}[{\cite[Thm.2.1.5]{HLRV11}, \cite[Thm.5.1.1]{HLRV13}}]\label{lem:generic_character_varieties}
If $(C_1,\cdots,C_k)$ is generic of type $\type$, and
\[
PG_{\pa}\footnote{Here, the notation ``$P$'' stands for ``projective''.}:=G_{\pa}/\field^\times,
\]
then, if nonempty,
$
\modulispace_{\type}=M_B/PG_{\pa}.
$
Moreover, it is a connected smooth affine $\field$-variety of dimension
\begin{equation}
d_{\type}:=n^2(2g-2+k)-\sum_{i,j}(\mu_j^i)^2+2,
\end{equation}
and the quotient map
$
\pi:M_B\to \modulispace_{\type}=M_B/PG_{\pa}
$
is a principal $PG_{\pa}$-bundle.
\end{lemma}

\begin{definition}\label{def:ordered_nicely}
We say that $C_i\in T$ is \emph{ordered nicely} if it is of the form
$
\diag(\lambda_{i,1}I_{\mu_1^i},\cdots,\lambda_{i,r_i}I_{\mu_{r_i}^i}).
$
\end{definition}

Without loss of generality, we may assume that each $C_i\in T$ is ordered nicely, so that $Z(C_i)\subset G$ is the Levi subgroup of block-diagonal matrices of type $\mu^i$.

\subsection{Very generic character varieties}\label{subsec:very_generic_character_varieties}

Let $B\subset G$ be the standard Borel subgroup of upper triangular matrices, with unipotent radical $U\subset B$.
We make the following \emph{very generic} assumption.

\begin{assumption}\label{ass:very_generic_assumption}
$C_k$ is regular semisimple. Equivalently,
\[
Z(C_k)=T
\qquad\text{and}\qquad
\mu^k=(1^n)\in\cP_n.
\]
\end{assumption}

In this case, $\modulispace_{\type}$ is called a \emph{very generic character variety}.
A key feature of this case is the existence of a cell decomposition \cite{Mel25}; an enhanced version will be proved in Theorem \ref{thm:cell_decomposition_of_very_generic_character_varieties}.

The first step is the following.
Taking the diagonal defines the quotient morphism
\[
D:B\twoheadrightarrow B/U\isomorphic T.
\]
Define a closed affine subvariety of $M_B$ by
\begin{equation}
M_B':=M_B\cap (G^{2g+k-1}\times U),
\end{equation}
and a closed subgroup of $G_{\pa}$ by
\begin{equation}
B_{\pa}=B\times \prod_{i=1}^{k-1} Z(C_i)\hookrightarrow G_{\pa},
\qquad
(b,h_1,\cdots,h_{k-1})\mapsto (b,h_1,\cdots,h_{k-1},D(b)).
\end{equation}
Denote
\[
PB_{\pa}:=B_{\pa}/\field^\times.
\]

We have mutually inverse isomorphisms of $G_{\pa}$-varieties
\begin{eqnarray*}
&&PG_{\pa}\times^{PB_{\pa}}M_B' := (PG_{\pa}\times M_B')/PB_{\pa}
\xrightarrow[]{\simeq} M_B:\quad [g_{\pa},m'_B=(A_1,\cdots,x_k)]\mapsto g_{\pa}\cdot m_B',\\
&&M_B\rightarrow PG_{\pa}\times^{PB_{\pa}}M_B':\quad (A_1,\cdots,x_k)\mapsto
[(x_k,\id,\cdots,\id),((x_k^{-1}A_jx_k)_{j=1}^{2g},(x_k^{-1}x_i)_{i=1}^{k})],
\end{eqnarray*}
where $PG_{\pa}\times^{PB_{\pa}}M_B'$ is well defined by Proposition \ref{prop:associated_fiber_bundles}.
Hence, by Proposition \ref{prop:reduction_of_principal_bundles} and Proposition \ref{prop:associated_fiber_bundles}, we obtain a natural isomorphism of $\field$-varieties
\[
M_B'/PB_{\pa}
\isomorphic
(PG_{\pa}\times^{PB_{\pa}}M_B')/PG_{\pa}
\isomorphic
M_B/PG_{\pa}
=
\modulispace_{\type},
\]
and the quotient
\[
\pi':M_B'\to \modulispace_{\type}=M_B'/PB_{\pa}
\]
is a principal $PB_{\pa}$-bundle.
Observe also that
\[
PB_{\pa}/U\isomorphic PT_{\pa}:=(T\times\prod_{i=1}^{k-1}Z(C_i))/\field^\times.
\]

Now make the coordinate change
\begin{equation}
U\to U,\qquad x_k\mapsto u_k:=x_kC_kx_k^{-1}C_k^{-1}.
\end{equation}
Then we may rewrite
\begin{equation}\label{eqn:reduction_of_character_varieties}
M_B'
=
\{(A_1,\cdots,x_{k-1},u_k)\in G^{2g+k-1}\times U:
(\prod_{j=1}^g \lbracket A_{2j-1}, A_{2j} \rbracket)
(\prod_{i=1}^{k-1} x_iC_ix_i^{-1})u_kC_k=\id\}
\end{equation}
with the action of $B_{\pa}$ given by
\begin{equation}\label{eqn:action_on_M_B'}
\hspace{-0.85em}(b,h_1,\cdots,h_{k-1})\cdot (A_1,\cdots,x_{k-1},u_k)
:=
(bA_1b^{-1},\cdots,bx_{k-1}h_{k-1}^{-1},bu_k(b^{C_k})^{-1}),
~~
b^{C_k}:=C_kbC_k^{-1}.
\end{equation}

Next, for any nicely ordered $C\in T$, let
\[
P=BZ(C)\subset G
\]
be the associated parabolic subgroup.
Denote the Weyl groups by
\[
W\isomorphic S_n\subset G,
\qquad
W(C):=W(Z(C)).
\]
Recall the Bruhat decomposition
\[
G=\bigsqcup_{w\in W/W(C)} B\dot{w}P,
\]
where $\dot{w}\in W$ denotes the shortest representative of the coset $w\in W/W(C)$.

In our setting, let
\[
P_i:=BZ(C_i)\subset G.
\]
For each sequence
\begin{equation}
\vec{w}=(\tau_1,\cdots,\tau_{2g},w_1,\cdots,w_{k-1})
\in
W^{2g}\times \prod_{i=1}^{k-1} W/W(C_i),
\end{equation}
we obtain a locally closed affine $B_{\pa}$-subvariety of $M_B'$:
\begin{equation}\label{eqn:Equivariant_Bruhat_cell_of_character_variety}
M_B'(\vec{w})
:=
M_B'\cap
(\prod_{j=1}^{2g} B\tau_jB
\times
\prod_{i=1}^{k-1} B\dot{w}_iP_i
\times U)
=
M_B\cap
(\prod_{j=1}^{2g} B\tau_jB
\times
\prod_{i=1}^{k-1} B\dot{w}_iP_i
\times U).
\end{equation}
Define
\begin{equation}\label{eqn:Bruhat_cell_of_character_variety}
\modulispace_{\type}(\vec{w}):=\pi'(M_B'(\vec{w}))\hookrightarrow \modulispace_{\type}=M_B'/PB_{\pa}.
\end{equation}
By Corollary \ref{cor:base_change_of_principal_bundles}, $\modulispace_{\type}(\vec{w})\subset \modulispace_{\type}$ is a locally closed $\field$-subvariety, and the restricted quotient map
\[
\pi'_{\vec{w}}:=\pi'|_{M_B'(\vec{w})}:M_B'(\vec{w})\to \modulispace_{\type}(\vec{w})=M_B'(\vec{w})/PB_{\pa}
\]
is a principal $PB_{\pa}$-bundle.

To obtain a cell decomposition of $\modulispace_{\type}$, the idea is to decompose $M_B'(\vec{w})$ further via its connection with the so-called \emph{braid varieties}.
We return to this point after some preparations.

\subsection{Diagram calculus of matrices}\label{subsec:diagram_calculus_of_matrices}

To relate character varieties and braid varieties, we introduce a diagram calculus of matrices.
Denote
\begin{eqnarray*}
&&\FBr_n^+:=\langle \sigma_1,\cdots,\sigma_{n-1}\rangle;
\quad (\text{free monoid of $n$-strand positive braid presentations}),\\
&&\Br_n^+:=\FBr_n^+/(\sigma_i\sigma_{i+1}\sigma_i=\sigma_{i+1}\sigma_i\sigma_{i+1},\ \forall i;\ \sigma_i\sigma_j=\sigma_j\sigma_i,\ \forall |i-j|>1);
\quad (\text{$n$-strand positive braids}),\\
&&\permutation:\Br_n^+\to \Br_n^+/(\sigma_i^2=1,\forall i)\isomorphic S_n:
\ \sigma_k\mapsto \ms_k:=\permutation(\sigma_k)=(k\ k+1).
\quad (\text{underlying permutations})
\end{eqnarray*}

\noindent\textbf{Convention} \setword{{\color{blue}$2$}}{convention:braid_diagrams}:
as in Figure \ref{fig:braid_diagram}, any
\[
\beta=\sigma_{i_{\ell}}\cdots\sigma_{i_1}\in\FBr_n^+
\]
is represented by the \emph{braid diagram}
\[
[\sigma_{i_1}|\cdots|\sigma_{i_{\ell}}],
\]
drawn \emph{from left to right}, where $[\beta_1|\beta_2]$ denotes the concatenation of $\beta_1$ with $\beta_2$.
We label the left and right endpoints \emph{from bottom to top} by $1,2,\cdots,n$.
Denote
\[
[n]:=\{1,2,\cdots,n\}.
\]

\begin{figure}[!htbp]
\vspace{-0.3cm}
\begin{center}
\begin{tikzpicture}[baseline=-.5ex,scale=0.5]
\begin{scope}

\draw[thin] (0,0) to [out=0,in=180] (2,1) to [out=0,in=180] (4,2);
\draw[white,fill=white] (1,0.5) circle(0.1);
\draw[white,fill=white] (3,1.5) circle(0.1);
\draw[thin] (0,1) to [out=0,in=180] (2,0)--(4,0);
\draw[thin] (0,2)--(2,2) to [out=0,in=180] (4,1);

\draw (0,0) node[left] {\text{\tiny$1$}};
\draw (0,1) node[left] {\text{\tiny$2$}};
\draw (0,2) node[left] {\text{\tiny$3$}};

\draw (4,0) node[right] {\text{\tiny$1$}};
\draw (4,1) node[right] {\text{\tiny$2$}};
\draw (4,2) node[right] {\text{\tiny$3$}};

\draw (2,-0.5) node[below] {\text{\tiny$[\sigma_1|\sigma_2]=\sigma_2\sigma_1$}};

\end{scope}
\end{tikzpicture}
\end{center}
\vspace{-0.5cm}
\caption{Braid diagram of a positive braid for $n=3$.}
\label{fig:braid_diagram}
\end{figure}

\begin{definition}\label{def:elementary_braid_matrix_diagrams}
Let $e_{i,j}\in M_{n\times n}(\field)$ be the matrix with entries
$
(e_{i,j})_{a,b}=\delta_{a,i}\delta_{b,j}.
$
\begin{enumerate}[wide,labelwidth=!,labelindent=0pt,itemindent=!]
\item
For any $\epsilon\in \field^{\times}$, define
\[
\mK_k(\epsilon):=\sum_{i\neq k}e_{i,i}+\epsilon e_{k,k}\in G,
\qquad 1\leq k\leq n,
\]
and let $[\mK_k(\epsilon)]$ be the diagram in Figure \ref{fig:elementary_braid_matrix_diagrams} (left).

\item
For any $\epsilon\in \field$, define
\[
\mH_{i,j}(\epsilon):=I_n+\epsilon e_{i,j}\in G,
\qquad 1\leq i<j\leq n,
\]
and let $[\mH_{i,j}(\epsilon)]$ be the diagram in Figure \ref{fig:elementary_braid_matrix_diagrams} (middle).
Denote
$
\mH_k(\epsilon):=\mH_{k,k+1}(\epsilon).
$

\item
Define
\[
[\ms_k]:=\sigma_k
\]
to be the diagram in Figure \ref{fig:elementary_braid_matrix_diagrams} (right), for $1\leq k\leq n-1$.
\end{enumerate}
Each of $[\mK_k(\epsilon)]$, $[\mH_{i,j}(\epsilon)]$, and $[\ms_k]$ is called an \emph{elementary braid matrix diagram} of rank $n$ over $\field$.
\end{definition}

\begin{figure}[!htbp]
\vspace{-0.5cm}
\begin{center}
\begin{tikzcd}[row sep=0.5pc,column sep=1pc,ampersand replacement=\&]
\begin{tikzpicture}[baseline=-.5ex,scale=0.3]
\begin{scope}

\draw[thin] (0,0)--(4,0);
\draw[thin] (0,2.5)--(4,2.5);
\draw[thin] (0,5)--(4,5);
\draw (2,1.25) node[right] {\text{\tiny$\vdots$}};
\draw (2,3.75) node[right] {\text{\tiny$\vdots$}};

\draw (0,0) node[left] {\text{\tiny$1$}};
\draw (0,2.5) node[left] {\text{\tiny$k$}};
\draw (0,5) node[left] {\text{\tiny$n$}};

\draw (4,0) node[right] {\text{\tiny$1$}};
\draw (4,2.5) node[right] {\text{\tiny$k$}};
\draw (4,5) node[right] {\text{\tiny$n$}};

\draw (1.5,3.3) node[below] {$\ast$};
\draw (1.5,2.5) node[above] {$\epsilon$};

\draw (2,-0.5) node[below] {\text{\tiny$[\mK_k(\epsilon)]$}};

\end{scope}
\end{tikzpicture}
\&
\begin{tikzpicture}[baseline=-.5ex,scale=0.3]
\begin{scope}

\draw[thin] (0,0)--(4,0);
\draw[thin] (0,1.5)--(4,1.5);
\draw[thin] (0,3.5)--(4,3.5);
\draw[thin] (0,5)--(4,5);
\draw (2,0.65) node[right] {\text{\tiny$\vdots$}};
\draw (2,2.45) node[right] {\text{\tiny$\vdots$}};
\draw (2,4.2) node[right] {\text{\tiny$\vdots$}};

\draw (0,0) node[left] {\text{\tiny$1$}};
\draw (0,1.5) node[left] {\text{\tiny$i$}};
\draw (0,3.5) node[left] {\text{\tiny$j$}};
\draw (0,5) node[left] {\text{\tiny$n$}};

\draw (4,0) node[right] {\text{\tiny$1$}};
\draw (4,1.5) node[right] {\text{\tiny$i$}};
\draw (4,3.5) node[right] {\text{\tiny$j$}};
\draw (4,5) node[right] {\text{\tiny$n$}};

\draw[line width=0.5mm] (1.5,1.5)--(1.5,3.5);
\draw (1.5,2.5) node[left] {$\epsilon$};

\draw (2,-0.5) node[below] {\text{\tiny$[\mH_{i,j}(\epsilon)]$}};

\end{scope}
\end{tikzpicture}
\&
\begin{tikzpicture}[baseline=-.5ex,scale=0.3]
\begin{scope}

\draw[thin] (0,0)--(4,0);
\draw[thin] (0,2)--(1,2) to [out=0,in=180] (3,3)--(4,3);
\draw[white,fill=white] (2,2.5) circle (0.2);
\draw[thin] (0,3)--(1,3) to[out=0,in=180] (3,2)--(4,2);
\draw[thin] (0,5)--(4,5);
\draw (1.6,1) node[right] {\text{\tiny$\vdots$}};
\draw (1.6,4) node[right] {\text{\tiny$\vdots$}};

\draw (0,0) node[left] {\text{\tiny$1$}};
\draw (0,2) node[left] {\text{\tiny$k$}};
\draw (0,3) node[left] {\text{\tiny$k+1$}};
\draw (0,5) node[left] {\text{\tiny$n$}};

\draw (4,0) node[right] {\text{\tiny$1$}};
\draw (4,2) node[right] {\text{\tiny$k$}};
\draw (4,3) node[right] {\text{\tiny$k+1$}};
\draw (4,5) node[right] {\text{\tiny$n$}};

\draw (2,-0.5) node[below] {\text{\tiny$[\ms_k]=\sigma_k$}};

\end{scope}
\end{tikzpicture}

\end{tikzcd}
\end{center}
\vspace{-0.5cm}
\caption{Elementary braid matrix diagrams representing elementary matrices:
$[\mK_k(\epsilon)]$ (\emph{scaling}),
$[\mH_{i,j}(\epsilon)]$ for $i<j$ (\emph{handleslide}),
and $[\ms_k]$ (\emph{transposition}).}
\label{fig:elementary_braid_matrix_diagrams}
\end{figure}

\begin{definition}[Braid matrix diagram presentations]~\label{def:braid_matrix_diagram_presentations}
\begin{enumerate}[wide,labelwidth=!,labelindent=0pt]
\item
The monoid of \emph{braid matrix diagram presentations} ($\bmdp$) of rank $n$ over $\field$ is
\[
\FBD_n
:=
\langle \sigma_1,\cdots,\sigma_{n-1},[\mK_k(\epsilon)],[\mH_{i,j}(\epsilon')]\rangle
/([\mK_k(1)]=\id_n=[\mH_{i,j}(0)]),
\qquad
\Gamma_2\circ\Gamma_1:=[\Gamma_1|\Gamma_2].
\]
Thus, a $\bmdp$ may be viewed as a finite concatenation of elementary braid matrix diagrams.

\item
Define a monoid morphism
\[
g_{-}:\FBD_n\rightarrow G,
\qquad
\sigma_k\mapsto \ms_k,\quad
[\mK_k(\epsilon)]\mapsto \mK_k(\epsilon),\quad
[\mH_{i,j}(\epsilon')]\mapsto \mH_{i,j}(\epsilon').
\]
We call $g_{\Gamma}$ the \emph{underlying matrix} of $\Gamma\in\FBD_n$.

\item
Two $\bmdp$'s $\Gamma_1,\Gamma_2\in\FBD_n$ are \emph{weakly equivalent} if
$
g_{\Gamma_1}=g_{\Gamma_2},
$
denoted by
\[
\Gamma_1\weakequivalent \Gamma_2.
\]
\end{enumerate}
\end{definition}

\begin{lemma}[\emph{Elementary moves} of $\bmdp$'s]\label{lem:elementary_moves}
Denote $i':=\ms_k(i)$. Then
\begingroup
\setlength{\tabcolsep}{2pt}

\begin{tabular}{ll}
\minipage{0.42\linewidth}
\[
\hspace{-1em}\mK_k(\epsilon_1)\text{\tiny$\circ$}\mK_{\ell}(\epsilon_2)=\left\{\begin{array}{ll}
\mK_k(\epsilon_1\epsilon_2) & \text{if }k=\ell,\\
\mK_{\ell}(\epsilon_2)\text{\tiny$\circ$}\mK_k(\epsilon_1) & \text{if }k\neq\ell.
\end{array}\right.
\]
\endminipage
&
\minipage{0.53\linewidth}
\[
\mH_{i,j}(\epsilon_1)\text{\tiny$\circ$}\mK_k(\epsilon_2)=\left\{\begin{array}{ll}
\mK_i(\epsilon_2)\text{\tiny$\circ$}\mH_{i,j}(\epsilon_2^{-1}\epsilon_1) & \text{if }k=i,\\
\mK_j(\epsilon_2)\text{\tiny$\circ$}\mH_{i,j}(\epsilon_1\epsilon_2) & \text{if }k=j,\\
\mK_k(\epsilon_2)\text{\tiny$\circ$}\mH_{i,j}(\epsilon_1) & \text{if }k\neq i,j.
\end{array}\right.
\]
\endminipage
\end{tabular}
\endgroup

\begingroup
\setlength{\tabcolsep}{2pt}
\begin{tabular}{ll}
\minipage{0.25\textwidth}
\[
\left.\begin{array}{l}
\hspace{-2em}\ms_k\text{\tiny$\circ$}\mK_i(\epsilon)=\mK_{i'}(\epsilon)\text{\tiny$\circ$}\ms_k,\\
\hspace{-2em}\ms_k\text{\tiny$\circ$}\mH_{i,j}(\epsilon)=\mH_{i',j'}(\epsilon)\text{\tiny$\circ$}\ms_k.
\end{array}\right.
\]
\endminipage\hfill
&
\minipage{0.7\textwidth}
\[
\mH_{i,j}(\epsilon_1)\text{\tiny$\circ$}\mH_{k,\ell}(\epsilon_2)=\left\{\begin{array}{ll}
\mH_{i,j}(\epsilon_1+\epsilon_2) & \text{if }(k,\ell)=(i,j),\\
\mH_{j,\ell}(\epsilon_2)\text{\tiny$\circ$}\mH_{i,\ell}(\epsilon_1\epsilon_2)\text{\tiny$\circ$}\mH_{i,j}(\epsilon_1) & \text{if }k=j,\\
\mH_{k,i}(\epsilon_2)\text{\tiny$\circ$}\mH_{k,j}(-\epsilon_2\epsilon_1)\text{\tiny$\circ$}\mH_{i,j}(\epsilon_1) & \text{if }\ell=i,\\
\mH_{k,\ell}(\epsilon_2)\text{\tiny$\circ$}\mH_{i,j}(\epsilon_1) & \text{otherwise.}
\end{array}\right.
\]
\endminipage\hfill
\end{tabular}
\endgroup
\\
Each identity is either trivial (commutative), or represented by an elementary move in Figure \ref{fig:elementary_moves}.
\end{lemma}

\begin{proof}
Let $e_1,\cdots,e_n$ be the standard basis of $\field^n$.
A direct computation shows that both sides of each identity have the same effect on every $e_i$.
See also Figure \ref{fig:elementary_moves} for an illustration.
\end{proof}

\begin{figure}[!htbp]
\vspace{-0.5cm}
\begin{center}
\begin{tikzcd}[row sep=0.5pc,column sep=0.5pc,ampersand replacement=\&]

\begin{tikzpicture}[baseline=-.5ex,scale=0.5]
\begin{scope}

\draw[thin] (0,0)--(1.5,0);
\draw[thin] (0,1)--(1.5,1);
\draw[thin] (0,2)--(1.5,2);

\draw (0.5,1.5) node[below] {$\ast$};
\draw (0.5,1) node[below] {\text{\tiny$\epsilon_2$}};

\draw (1,1.5) node[below] {$\ast$};
\draw (1,1) node[above] {\text{\tiny$\epsilon_1$}};

\end{scope}
\end{tikzpicture}

\arrow[r,leftrightarrow,line width=0.1mm,shift left=4,"{(1)}"]\&

\begin{tikzpicture}[baseline=-.5ex,scale=0.5]
\begin{scope}

\draw[thin] (0,0)--(1.5,0);
\draw[thin] (0,1)--(1.5,1);
\draw[thin] (0,2)--(1.5,2);

\draw (0.75,1.5) node[below] {$\ast$};
\draw (0.75,1) node[below] {\text{\tiny$\epsilon_1\epsilon_2$}};

\end{scope}
\end{tikzpicture}

\&

\begin{tikzpicture}[baseline=-.5ex,scale=0.5]
\begin{scope}

\draw[thin] (0,0)--(1.5,0);
\draw[thin] (0,2)--(1.5,2);

\draw (0.5,0.5) node[below] {$\ast$};
\draw (0.5,0) node[below] {\text{\tiny$\epsilon_2$}};

\draw[line width=0.5mm] (1,0)--(1,2);
\draw (0.8,1.2) node[right] {\text{\tiny$\epsilon_1$}};

\end{scope}
\end{tikzpicture}

\arrow[r,leftrightarrow,line width=0.1mm,shift left=4,"{(2)}"]\&

\begin{tikzpicture}[baseline=-.5ex,scale=0.5]
\begin{scope}

\draw[thin] (0,0)--(1.5,0);
\draw[thin] (0,2)--(1.5,2);

\draw (1,0.5) node[below] {$\ast$};
\draw (1,0) node[below] {\text{\tiny$\epsilon_2$}};

\draw[line width=0.5mm] (0.5,0)--(0.5,2);
\draw (0.3,1.2) node[right] {\text{\tiny$\epsilon_2^{-1}\epsilon_1$}};

\end{scope}
\end{tikzpicture}

\&

\begin{tikzpicture}[baseline=-.5ex,scale=0.5]
\begin{scope}

\draw[thin] (0,0)--(1.5,0);
\draw[thin] (0,2)--(1.5,2);

\draw (0.5,2.5) node[below] {$\ast$};
\draw (0.7,2) node[above] {\text{\tiny$\epsilon_2$}};

\draw[line width=0.5mm] (1,0)--(1,2);
\draw (0.8,0.8) node[right] {\text{\tiny$\epsilon_1$}};

\end{scope}
\end{tikzpicture}

\arrow[r,leftrightarrow,line width=0.1mm,shift left=4,"{(3)}"]\&

\begin{tikzpicture}[baseline=-.5ex,scale=0.5]
\begin{scope}

\draw[thin] (0,0)--(1.5,0);
\draw[thin] (0,2)--(1.5,2);

\draw (1,2.5) node[below] {$\ast$};
\draw (1,2) node[above] {\text{\tiny$\epsilon_2$}};

\draw[line width=0.5mm] (0.5,0)--(0.5,2);
\draw (0.3,0.8) node[right] {\text{\tiny$\epsilon_1\epsilon_2$}};

\end{scope}
\end{tikzpicture}

\&

\begin{tikzpicture}[baseline=-.5ex,scale=0.5]
\begin{scope}

\draw[thin] (0,0.5)--(0.5,0.5) to [out=0,in=180] (2.5,1.5);
\draw[white,fill=white] (1.5,1) circle (0.2);
\draw[thin] (0,1.5)--(0.5,1.5) to[out=0,in=180] (2.5,0.5);

\draw (0.5,1) node[below] {$\ast$};
\draw (0.5,0.5) node[below] {\text{\tiny$\epsilon$}};

\end{scope}
\end{tikzpicture}

\arrow[r,leftrightarrow,line width=0.1mm,shift left=4,"{(4)}"]\&

\begin{tikzpicture}[baseline=-.5ex,scale=0.5]
\begin{scope}

\draw[thin] (0,0.5) to [out=0,in=180] (2,1.5)--(2.5,1.5);
\draw[white,fill=white] (1,1) circle (0.2);
\draw[thin] (0,1.5) to [out=0,in=180] (2,0.5)--(2.5,0.5);

\draw (2,2) node[below] {$\ast$};
\draw (2,1.5) node[above] {\text{\tiny$\epsilon$}};

\end{scope}
\end{tikzpicture}

\end{tikzcd}
\end{center}

\begin{center}
\begin{tikzcd}[row sep=0.5pc,column sep=0.5pc,ampersand replacement=\&]

\begin{tikzpicture}[baseline=-.5ex,scale=0.5]
\begin{scope}

\draw[thin] (0,0)--(1.5,0);
\draw[thin] (0,2)--(1.5,2);

\draw[line width=0.5mm] (0.5,0)--(0.5,2);
\draw (0.8,1) node[left] {\text{\tiny$\epsilon_2$}};

\draw[line width=0.5mm] (1,0)--(1,2);
\draw (0.7,1) node[right] {\text{\tiny$\epsilon_1$}};

\end{scope}
\end{tikzpicture}

\arrow[r,leftrightarrow,line width=0.1mm,shift left=4,"{(5)}"]\&

\begin{tikzpicture}[baseline=-.5ex,scale=0.5]
\begin{scope}

\draw[thin] (0,0)--(1.5,0);
\draw[thin] (0,2)--(1.5,2);

\draw[line width=0.5mm] (0.5,0)--(0.5,2);
\draw (0.2,1) node[right] {\text{\tiny$\epsilon_1+\epsilon_2$}};

\end{scope}
\end{tikzpicture}

\&

\begin{tikzpicture}[baseline=-.5ex,scale=0.5]
\begin{scope}

\draw[thin] (0,0)--(1.5,0);
\draw[thin] (0,1)--(1.5,1);
\draw[thin] (0,2)--(1.5,2);

\draw[line width=0.5mm] (1,0)--(1,1);
\draw (0.7,0.5) node[right] {\text{\tiny$\epsilon_1$}};

\draw[line width=0.5mm] (0.5,1)--(0.5,2);
\draw (0.8,1.5) node[left] {\text{\tiny$\epsilon_2$}};

\end{scope}
\end{tikzpicture}

\arrow[r,leftrightarrow,line width=0.1mm,shift left=4,"{(6)}"]\&

\begin{tikzpicture}[baseline=-.5ex,scale=0.5]
\begin{scope}

\draw[thin] (0,0)--(2,0);
\draw[thin] (0,1)--(2,1);
\draw[thin] (0,2)--(2,2);

\draw[line width=0.5mm] (0.5,0)--(0.5,1);
\draw (0.8,0.5) node[left] {\text{\tiny$\epsilon_1$}};

\draw[line width=0.5mm] (1.5,1)--(1.5,2);
\draw (1.2,1.5) node[right] {\text{\tiny$\epsilon_2$}};

\draw[line width=0.5mm] (1,0)--(1,2);
\draw (0.7,0.5) node[right] {\text{\tiny$\epsilon_1\epsilon_2$}};

\end{scope}
\end{tikzpicture}

\&

\begin{tikzpicture}[baseline=-.5ex,scale=0.5]
\begin{scope}

\draw[thin] (0,0)--(1.5,0);
\draw[thin] (0,1)--(1.5,1);
\draw[thin] (0,2)--(1.5,2);

\draw[line width=0.5mm] (1,1)--(1,2);
\draw (0.7,1.5) node[right] {\text{\tiny$\epsilon_1$}};

\draw[line width=0.5mm] (0.5,0)--(0.5,1);
\draw (0.8,0.5) node[left] {\text{\tiny$\epsilon_2$}};

\end{scope}
\end{tikzpicture}

\arrow[r,leftrightarrow,line width=0.1mm,shift left=4,"{(7)}"]\&

\begin{tikzpicture}[baseline=-.5ex,scale=0.5]
\begin{scope}

\draw[thin] (0,0)--(2,0);
\draw[thin] (0,1)--(2,1);
\draw[thin] (0,2)--(2,2);

\draw[line width=0.5mm] (0.5,1)--(0.5,2);
\draw (0.8,1.5) node[left] {\text{\tiny$\epsilon_1$}};

\draw[line width=0.5mm] (1.5,0)--(1.5,1);
\draw (1.2,0.5) node[right] {\text{\tiny$\epsilon_2$}};

\draw[line width=0.5mm] (1,0)--(1,2);
\draw (0.7,1.5) node[right] {\text{\tiny$-\epsilon_2\epsilon_1$}};

\end{scope}
\end{tikzpicture}

\&

\begin{tikzpicture}[baseline=-.5ex,scale=0.5]
\begin{scope}

\draw[thin] (0.3,0)--(0.5,0) to [out=0,in=180] (2.5,1);
\draw[white,fill=white] (1.5,0.5) circle (0.2);
\draw[thin] (0.3,1)--(0.5,1) to [out=0,in=180] (2.5,0);
\draw[thin] (0.3,2)--(2.5,2);

\draw[line width=0.5mm] (0.5,1)--(0.5,2);
\draw (0.2,1.5) node[right] {\text{\tiny$\epsilon$}};

\end{scope}
\end{tikzpicture}

\arrow[r,leftrightarrow,line width=0.1mm,shift left=4,"{(8)}"]\&

\begin{tikzpicture}[baseline=-.5ex,scale=0.5]
\begin{scope}

\draw[thin] (0,0) to [out=0,in=180] (2,1)--(2.2,1);
\draw[white,fill=white] (1,0.5) circle (0.2);
\draw[thin] (0,1) to [out=0,in=180] (2,0)--(2.2,0);
\draw[thin] (0,2)--(2.2,2);

\draw[line width=0.5mm] (2,0)--(2,2);
\draw (1.7,1.5) node[right] {\text{\tiny$\epsilon$}};

\end{scope}
\end{tikzpicture}

\end{tikzcd}
\end{center}
\vspace{-0.5cm}
\caption{\emph{Elementary moves} for $\bmdp$'s; the trivial ones are omitted.
Each move is a weak equivalence in $\FBD_n$ representing one of the composition identities in Lemma \ref{lem:elementary_moves}, and conversely.
Composition is read from left to right: $\Gamma_2\circ\Gamma_1=[\Gamma_1|\Gamma_2]$.}
\label{fig:elementary_moves}
\end{figure}

\begin{definition}[Braid matrix diagrams]~\label{def:braid_matrix_diagrams}
\begin{enumerate}[wide,labelwidth=!,labelindent=0pt]
\item
Let $\underline{\FBD}_n$ be the quotient of $\FBD_n$ by elementary moves.
Then there is a monoid morphism
\[
\beta_{-}:\underline{\FBD}_n\rightarrow \FBr_n^+,
\qquad
\sigma_k\mapsto \sigma_k,\quad
[\mK_k(\epsilon)]\mapsto \id_n,\quad
[\mH_{i,j}(\epsilon')]\mapsto \id_n,
\]
called the \emph{underlying shape}.

\item
The monoid $\BD_n$ of rank-$n$ \emph{braid matrix diagrams} ($\bmd$) is the quotient of $\underline{\FBD}_n$ by the usual braid relations/moves:
\[
[\sigma_i|\sigma_j]\sim[\sigma_j|\sigma_i],\quad |i-j|>1;\qquad
[\sigma_i|\sigma_{i+1}|\sigma_i]\sim[\sigma_{i+1}|\sigma_i|\sigma_{i+1}], \quad 1\leq i<n-1.
\]
Thus
\[
\BD_n=\FBD_n/\braidequivalent,
\]
where $\braidequivalent$\footnote{The notation $\braidequivalent$ stands for \emph{braid equivalence}. Clearly, $\braidequivalent$ implies $\weakequivalent$. On the other hand, $\sigma_k^2\weakequivalent\id_n$, but $\sigma_k^2\notbraidequivalent\id_n$.}
is generated by elementary moves together with braid moves.

\item
The maps $\beta_{-}:\underline{\FBD}_n\to\FBr_n^+$ and $g_{-}:\FBD_n\to G$ induce monoid morphisms
\[
\beta_{-}:\BD_n\rightarrow \Br_n^+,
\qquad
g_{-}:\BD_n\rightarrow G.
\]
\end{enumerate}
\end{definition}

\begin{remark}\label{rem:positive_braids_vs_braid_matrix_diagrams}
There are natural monoid morphisms
\[
i:\FBr_n^+\rightarrow \underline{\FBD}_n:\sigma_k\mapsto \sigma_k,
\qquad\rightsquigarrow\qquad
i:\Br_n^+\rightarrow \BD_n:\sigma_k\mapsto \sigma_k.
\]
Since $\beta_{-}\circ i=\id$, the map $i$ induces embeddings
\[
\FBr_n^+\hookrightarrow \underline{\FBD}_n,
\qquad
\Br_n^+\hookrightarrow \BD_n.
\]
This explains the terminology.
Altogether, we obtain a commutative diagram of monoids:
\[
\begin{tikzcd}[row sep=1pc,column sep=1.5pc]
&\FBr_n^+\arrow[d,hookrightarrow,swap,"{i}"]\arrow[dl,hookrightarrow]\arrow[r,twoheadrightarrow] & \Br_n^+\arrow[d,hookrightarrow,swap,"{i}"]\arrow[r,twoheadrightarrow,"{\permutation}"] & S_n\arrow[d,hookrightarrow]\\
\FBD_n\arrow[r,twoheadrightarrow] & \underline{\FBD}_n\arrow[d,twoheadrightarrow,"{\beta_{-}}"]\arrow[r,twoheadrightarrow] & \BD_n\arrow[d,twoheadrightarrow,"{\beta_{-}}"]\arrow[r,twoheadrightarrow,"{g_{-}}"] & G\\
&\FBr_n^+\arrow[r,twoheadrightarrow] & \Br_n^+ &
\end{tikzcd}
\]
\end{remark}

Next, the map $\ms:\Br_n^+\to S_n$ admits a canonical section as a map of sets, characterized by
\begin{equation}\label{eqn:positive_braid_lifting_a_permutation}
[-]:S_n\to \Br_n^+\subset\BD_n:\ w\mapsto [w],
\qquad
\ms([w])=w,\quad \ell([w])=\ell(w).
\end{equation}
As above, $[\ms_k]=\sigma_k$.
Up to choice, we may assume $[w]\in\FBr_n^+$.
By definition,
\begin{equation}\label{eqn:composition_of_positive_braids_lifting_permutations}
[w_1]\circ[w_2]=[w_1w_2]\in\Br_n^+
\quad\Longleftrightarrow\quad
\ell(w_1w_2)=\ell(w_1)+\ell(w_2).
\end{equation}
Also, for $B\subset G$, there is a canonical monoid morphism extending Definition \ref{def:elementary_braid_matrix_diagrams}:
\begin{equation}\label{eqn:braid_matrix_diagrams_lifting_Borel_elements}
[-]:B\to \underline{\FBD}_n:\ b\mapsto [b],
\qquad
g_{[b]}=b.
\end{equation}

\begin{proposition}\label{prop:from_matrices_to_braid_matrix_diagrams}
The map $g_{-}:\BD_n\to G$ admits a canonical section as a map of sets,
\[
[-]:G=\bigsqcup_{w\in W}BwB\to \BD_n:
\qquad
x=b_1wb_2\mapsto [x]:=[b_1]\circ[w]\circ[b_2].
\]
Up to a choice of $[w]\in\FBr_n^+$, we may assume $[x]\in\underline{\FBD}_n$.
Moreover, for any other $x'=b_1'w'b_2'\in Bw'B$,
\[
[x]\circ[x']=[xx']\in\BD_n
\Longleftrightarrow
\ell(ww')=\ell(w)+\ell(w')
\Longleftrightarrow
\mathrm{inv}(ww')=w'^{-1}(\mathrm{inv}(w))\sqcup\mathrm{inv}(w').
\]
In this case, we obtain unique decompositions for $xx'\in Bww'B=Bww'U_{ww'}^-$:
\begin{equation}
Bww'U_{ww'}^-
=
BwU_w^-w'U_{w'}^-
=
U_{w^{-1}}^-wBw'U_{w'}^-
=
U_{w^{-1}}^-wU_{w'^{-1}}^-w'B
=
U_{w'^{-1}w^{-1}}^-ww'B.
\end{equation}
\end{proposition}

\noindent\textbf{Convention} \setword{{\color{blue}$3$}}{convention:notation_for_braid_matrix_diagrams}:
often, for $[A]\in\BD_n$, we use the same notation for a chosen lift of $[A]$ to $\FBD_n$ or $\underline{\FBD}_n$.

Before giving the proof, we recall some linear algebra.
Denote
\[
\cI:=\{(i,j):1\leq i<j\leq n\}.
\]
Then
\[
U=I_n+\sum_{(i,j)\in \cI}\field e_{i,j}\subset B\subset G.
\]
A subset $\cJ\subset \cI$ is called \emph{multiplicative} if
\[
(i,j),(j,k)\in \cJ \Longrightarrow (i,k)\in \cJ.
\]
For each $m\geq 1$, set $\cJ_1:=\cJ$ and, more generally,
\[
\cJ_m:=\{(j_0,\cdots,j_m):(j_k,j_{k+1})\in\cJ,\ \forall\,0\leq k\leq m-1\}.
\]

\begin{lemma}\label{lem:closed_subgroups_of_U_via_handleslides}
Let $\cJ\subset \cI$ be multiplicative, and denote
\[
U_{\cJ}:=I_n+\sum_{(i,j)\in \cJ}\field e_{i,j}\subset U.
\]
Then:
\begin{enumerate}
\item
$U_{\cJ}\subset U$ is a closed subgroup.

\item
Any fixed total order $\preceq$ on $\cJ$ induces an isomorphism of $\field$-varieties
\[
\phi_{\preceq}:\mathbb{A}^{|\cJ|}\xrightarrow{\ \simeq\ } U_{\cJ},
\qquad
(\epsilon_{i,j})_{(i,j)\in \cJ}\mapsto \prod_{(i,j)\in \cJ} \mH_{i,j}(\epsilon_{i,j}).
\]
\end{enumerate}
\end{lemma}

\begin{proof}
We prove (2); the proof of (1) is similar.
Write
$
\cJ=\{(i_1,j_1)\prec\cdots\prec(i_N,j_N)\}.
$
Then
\[
\prod_{(i,j)\in\cJ} \mH_{i,j}(\epsilon_{i,j})
=
\prod_{\ell=1}^N (I_n+\epsilon_{i_{\ell},j_{\ell}}e_{i_{\ell},j_{\ell}})
=
I_n+\sum_{m\geq 1}\sum_{(k_0,k_1)\prec\cdots\prec(k_{m-1},k_m)\text{ in }\cJ}
\epsilon_{k_0,k_1}\cdots\epsilon_{k_{m-1},k_m}e_{k_0,k_m}.
\]
Thus the equation
$
I_n+ \sum_{(i,j)\in\cJ} a_{ij}e_{i,j}
=
\prod_{(i,j)\in\cJ} \mH_{i,j}(\epsilon_{i,j})
$
becomes
\[
a_{i,j}
=
\epsilon_{i,j}
+
\sum_{m\geq 2}\sum_{(i,k_1)\prec\cdots\prec(k_{m-1},j)\text{ in }\cJ}
\epsilon_{i,k_1}\cdots\epsilon_{k_{m-1},j},
\qquad
\forall\,(i,j)\in\cJ.
\]
Hence the $\epsilon_{i,j}$ are uniquely determined inductively, in increasing order of $|j-i|$.
\end{proof}

For any permutation $w\in W=S_n$, denote
\[
\mathrm{inv}(w):=\{(i,j)\in\cI:i<j,\ w(i)>w(j)\},
\qquad
\mathrm{noinv}(w):=\{(i,j)\in\cI:i<j,\ w(i)<w(j)\}.
\]
Then
\[
\ell(w)=|\mathrm{inv}(w)|,
\]
and both $\mathrm{inv}(w)$ and $\mathrm{noinv}(w)$ are multiplicative subsets of $\cI$.
Hence
\[
U_w^+:=U_{\mathrm{noinv}(w)}
=I_n+\sum_{(i,j)\in\mathrm{noinv}(w)} \field e_{i,j},
\qquad
U_w^-:=U_{\mathrm{inv}(w)}
=I_n+\sum_{(i,j)\in\mathrm{inv}(w)} \field e_{i,j},
\]
are closed subgroups of $U$, to which Lemma \ref{lem:closed_subgroups_of_U_via_handleslides} applies.
Equivalently,
\begin{equation}
U_w^+=U\cap w^{-1}Uw=w^{-1}U_{w^{-1}}^+w,
\qquad
U_w^-=U\cap w^{-1}U^-w,
\end{equation}
where $U^-\subset G$ is the opposite unipotent subgroup.
By Lemma \ref{lem:closed_subgroups_of_U_via_handleslides} (2), we obtain the decompositions
\begin{equation}\label{eqn:decompose_U}
U=U_w^+U_w^-=U_w^-U_w^+,
\qquad
BwB=U_{w^{-1}}^-wB=BwU_w^-.
\end{equation}

\begin{proof}[Proof of Proposition \ref{prop:from_matrices_to_braid_matrix_diagrams}]
Clearly, the map $[-]$ is well defined and satisfies $g_{-}\circ[-]=\id$.
We now prove the stated equivalences.

\noindent
\emph{``LHS $\Rightarrow$ Middle''.}
Suppose $xx'\in B\tilde{w}B$ for some $\tilde{w}\in S_n$, so that $\beta_{[xx']}=[\tilde{w}]$.
If $[x]\circ[x']=[xx']$, then
\[
[\tilde{w}]
=
\beta_{[xx']}
=
\beta_{[x]}\circ\beta_{[x']}
=
[w]\circ[w'],
\qquad
\tilde{w}
=
\ms(\beta_{[xx']})
=
\ms(\beta_{[x]})\circ\ms(\beta_{[x']})
=
ww'.
\]
Hence
\[
\ell(ww')
=
\ell([ww'])
=
\ell([w]\circ[w'])
=
\ell([w])+\ell([w'])
=
\ell(w)+\ell(w'),
\]
as required.

\noindent
\emph{``Middle $\Rightarrow$ RHS''.}
Denote
\begin{eqnarray*}
&&(a',b'):=(w'(a),w'(b)),\qquad (a'',b''):=(w(a'),w(b')),
\qquad 1\leq a<b\leq n;\\
&&I_{+-}:=\{(a,b)\in \cI: a'<b',\ a''>b''\},\qquad
I_{-+}:=\{(a,b)\in \cI: a'>b',\ a''>b''\}.
\end{eqnarray*}
Observe that
$
\mathrm{inv}(ww')=I_{+-}\sqcup I_{-+},
$
and that $w'$ induces bijections
\begin{eqnarray*}
&&w':I_{+-}\xrightarrow[]{\simeq}\mathrm{noinv}(w'^{-1})\cap\mathrm{inv}(w),\\
&&R\circ w':I_{-+}\xrightarrow[]{\simeq}\mathrm{inv}(w'^{-1})\cap\mathrm{noinv}(w),
\qquad
R(i,j):=(j,i).
\end{eqnarray*}
Hence
\[
w'\sqcup R\circ w':
\mathrm{inv}(ww')=I_{+-}\sqcup I_{-+}
\xrightarrow[]{\simeq}
\bigl(\mathrm{noinv}(w'^{-1})\cap\mathrm{inv}(w)\bigr)
\sqcup
\bigl(\mathrm{inv}(w'^{-1})\cap\mathrm{noinv}(w)\bigr).
\]
Therefore,
\[
\ell(ww')
=
|\mathrm{noinv}(w'^{-1})\cap\mathrm{inv}(w)|
+
|\mathrm{inv}(w'^{-1})\cap\mathrm{noinv}(w)|
\leq
|\mathrm{inv}(w)|+|\mathrm{inv}(w'^{-1})|
=
\ell(w)+\ell(w'),
\]
with equality if and only if
$
\mathrm{inv}(w)\cap\mathrm{inv}(w'^{-1})=\emptyset.
$
In that case,
\[
I_{+-}=w'^{-1}(\mathrm{inv}(w)),
\qquad
I_{-+}=(R\circ w')^{-1}(\mathrm{inv}(w'^{-1}))=\mathrm{inv}(w'),
\]
so
$
\mathrm{inv}(ww')=w'^{-1}(\mathrm{inv}(w))\sqcup\mathrm{inv}(w'),
$
as claimed.

\noindent
\emph{``RHS $\Rightarrow$ Middle''.}
This is immediate, so $\mathrm{RHS}\Leftrightarrow\mathrm{Middle}$.

\noindent
\emph{``Middle $\Rightarrow$ LHS''.}
By the preceding discussion,
$
\mathrm{noinv}(w)\cup\mathrm{noinv}(w'^{-1})=\cI.
$
By Lemma \ref{lem:closed_subgroups_of_U_via_handleslides}, we obtain a surjection
\[
m:U_w^+\times U_{w'^{-1}}^+\rightarrow U,
\qquad
(u_1,u_2)\mapsto u_1u_2.
\]
By \eqref{eqn:decompose_U}, we may assume
$b_2\in U_w^-$, and $b_1'\in U_{w'^{-1}}^-$.
By the above surjection, we may write
\[
b_2b_1'=u_1u_2,
\qquad
u_1\in U_w^+,\quad u_2\in U_{w'^{-1}}^+.
\]
Hence
$
xx'
=
(b_1wu_1w^{-1})\,ww'\,(w'^{-1}u_2w'b_2')
\in Bww'B,
$
and
\begin{eqnarray*}
[xx']
&=&
[b_1wu_1w^{-1}]\circ[ww']\circ[w'^{-1}u_2w'b_2']
\in\BD_n
\qquad (\text{by definition})\\
&=&
[b_1]\circ[wu_1w^{-1}]\circ[w]\circ[w']\circ[w'^{-1}u_2w']\circ[b_2']
\qquad (\text{by \eqref{eqn:composition_of_positive_braids_lifting_permutations}, \eqref{eqn:braid_matrix_diagrams_lifting_Borel_elements}})\\
&=&
[b_1]\circ([w]\circ[u_1])\circ([u_2]\circ[w'])\circ[b_2']
\qquad (\text{by elementary moves, cf.\ Figure \ref{fig:elementary_moves} (8)})\\
&=&
[b_1]\circ[w]\circ[b_2]\circ[b_1']\circ[w']\circ[b_2']
=
[x]\circ[x'].
\qquad (\text{by \eqref{eqn:braid_matrix_diagrams_lifting_Borel_elements}})
\end{eqnarray*}

It remains to prove the decompositions.
By Lemma \ref{lem:closed_subgroups_of_U_via_handleslides}, the multiplication map
\[
m:
U_{w'^{-1}(\mathrm{inv}(w))}\times U_{\mathrm{inv}(w')}
=
w'^{-1}U_w^-w'\times U_{w'}^-
\longrightarrow
U_{\mathrm{inv}(ww')}=U_{ww'}^-,
\qquad
(u_1,u_2)\mapsto u_1u_2
\]
is an isomorphism.
The same holds for
$
m:U_{w'}^-\times w'^{-1}U_w^-w'\rightarrow U_{ww'}^-.
$
Thus
$
U_{ww'}^-=(w'^{-1}U_w^-w')U_{w'}^-=U_{w'}^-(w'^{-1}U_w^-w').
$
Applying the same argument to $(ww')^{-1}=w'^{-1}w^{-1}$, we obtain the unique decompositions
\[
Bww'B
=
Bww'U_{ww'}^-
=
BwU_w^-w'U_{w'}^-
=
U_{w^{-1}}^-wBw'U_{w'}^-
=
U_{w^{-1}}^-wU_{w'^{-1}}^-w'B
=
U_{w'^{-1}w^{-1}}^-ww'B.
\qedhere
\]
\end{proof}

Note that for any $b_1,b_2\in B$, one has
\[
[b_1]=[b_2]\in\underline{\FBD}_n
\quad\Longleftrightarrow\quad
b_1=b_2.
\]
Moreover, any diagram in $\underline{\FBD}_n$ with underlying shape
$\beta=\sigma_{i_{\ell}}\cdots \sigma_{i_1}\in \FBr_n^+$ (Definition~\ref{def:braid_matrix_diagrams}) is of the form
\[
[b_0|\sigma_{i_1}|b_1|\cdots|\sigma_{i_{\ell}}|b_{\ell}]
=
[b_{\ell}]\circ \sigma_{i_{\ell}}\circ \cdots \circ [b_1]\circ \sigma_{i_1}\circ [b_0],
\qquad
b_j\in B.
\]

\begin{proposition}[Equality in $\underline{\FBD}_n$]\label{prop:equality_in_underline_FBD_n}
Fix $\beta = \sigma_{i_{\ell}}\cdots \sigma_{i_1}\in \FBr_n^+$. 
For any $b_j,b_j'\in B$, $0\leq j\leq \ell$, we have
\[
[b_{\ell}]\circ \sigma_{i_{\ell}}\circ \cdots \circ [b_1]\circ \sigma_{i_1}\circ [b_0]
=
[b_{\ell}']\circ \sigma_{i_{\ell}}\circ \cdots \circ [b_1']\circ \sigma_{i_1}\circ [b_0']
\in \underline{\FBD}_n
\]
if and only if there exist
$c_j\in U_{\ms_{i_j}}^+T$, $1\leq j\leq \ell$,
such that
\[
c_{j+1}^{-1}b_j(\ms_{i_j}c_j\ms_{i_j}^{-1})=b_j',
\qquad
\forall\,0\leq j\leq \ell,
\]
with the convention
$
c_0=\ms_{i_0}=c_{\ell+1}=1.
$
\end{proposition}

\begin{proof}
The elementary moves (1), (2), (3), (5), (6), and (7), together with the trivial moves in Figure \ref{fig:elementary_moves}, correspond to writing the same matrix $b\in B$ as different ordered products of scalings and handleslides.
Modulo these moves, the only remaining elementary moves are (4) and (8), namely the ones involving a crossing.
Such moves near different crossings clearly commute with one another.
Moreover, near a single crossing $\sigma_k$, a finite composition of such moves corresponds to the identity
\[
[\tilde{b}]\circ \sigma_k\circ [b]
=
[\tilde{b}(\ms_{i_k}c\ms_{i_k}^{-1})]\circ \sigma_k\circ [c^{-1}b],
\]
for some $b,\tilde{b}\in B$ and $c\in U_{\ms_{i_k}}^+T$.
The result follows.
\end{proof}

\subsection{Braid calculus for character varieties}

We now apply these results to character varieties. In particular, we reinterpret $M_B'(\vec{w})$ (see (\ref{eqn:Equivariant_Bruhat_cell_of_character_variety})) via braid matrix diagrams.

\medskip

First, we set up some notations. We have seen (unique) decompositions:
\begin{eqnarray}\label{eqn:decomposition_for_U}
\hspace{-2em}&&U=U_w^+U_w^-: u=L_w^+(u) L_w^-(u),\quad B=TU=(TU_w^+)U_w^-: b=D(b)b_R=L_w^+(b)L_w^-(b);\\
\hspace{-2em}&&U=U_w^-U_w^+: u=R_w^-(u)R_w^+(u),\quad B=UT=(U_w^-)(U_w^+T): b=b_LD(b)=R_w^{-}(b)R_w^+(b).\nonumber
\end{eqnarray}
Similarly, for each parabolic $P_i\subset G$, let $N_i$ be its unipotent radical. Then we have decompositions
\begin{equation}
P_i=N_iZ(C_i)=Z(C_i)N_i:\quad p_i=L_i(p_i)D_i(p_i)=D_i(p_i)R_i(p_i).
\end{equation}

\medskip

\noindent\textbf{Claim}. As the shortest representative of $w_i\in W/W(C_i)$, $\dot{w}_i$ satisfies
\[
U_{\dot{w}_i}^-\subset N_i,\qquad Z(C_i)\cap U\subset U_{\dot{w}_i}^+.
\]

\begin{proof}[Proof of Claim]
Let $\cJ_i:=\{(k,\ell)\in \cI:(k~\ell)\in W(C_i)\}$. Then $\cJ_i\subset \cI$ is the minimal multiplicative subset containing all $(k,k+1)$ with $s_k=(k~k+1)\in W(C_i)$. Moreover,
$Z(C_i)\cap U = U_{\cJ_i}$, and $N_i = U_{\cJ_i^c}$.
Recall that
$U_{\dot{w}_i}^+ = U_{\mathrm{noinv}(\dot{w}_i)}$,
$U_{\dot{w}_i}^- = U_{\mathrm{inv}(\dot{w}_i)}$, and
$\mathrm{inv}(\dot{w}_i)=\mathrm{noinv}(\dot{w}_i)^c$.
Thus it suffices to show $\cJ_i\subset \mathrm{noinv}(\dot{w}_i)$.

For any $s_k=(k~k+1)\in W(C_i)$, we have $\ell(\dot{w}_is_k)=\ell(\dot{w}_i)+1$. By Proposition \ref{prop:from_matrices_to_braid_matrix_diagrams},
$
\mathrm{inv}(s_k)=\{(k,k+1)\}\subset \mathrm{inv}(\dot{w}_is_k),
$
hence $\dot{w}_i(k+1)>\dot{w}_i(k)$, i.e.\ $(k,k+1)\in \mathrm{noinv}(\dot{w}_i)$. Since $\mathrm{noinv}(\dot{w}_i)$ is multiplicative, we conclude $\cJ_i\subset \mathrm{noinv}(\dot{w}_i)$.
\end{proof}

\medskip

By the claim and Lemma \ref{lem:closed_subgroups_of_U_via_handleslides}, we obtain decompositions
\begin{equation}\label{eqn:decomposition_for_N_i}
N_i=U_{\dot{w}_i}^-(U_{\dot{w}_i}^+\cap N_i) = (U_{\dot{w}_i}^+\cap N_i)U_{\dot{w}_i}^-,
\quad
U_{\dot{w}_i}^+ = (U_{\dot{w}_i}^+\cap N_i)(Z(C_i)\cap U) = (Z(C_i)\cap U)(U_{\dot{w}_i}^+\cap N_i).
\end{equation}

\medskip

Next, we reinterpret $M_B'(\vec{w})$. We have a canonical isomorphism
\begin{equation}\label{eqn:decomposition_for_Bw_iP_i}
U_{\dot{w}_i^{-1}}^-\times N_i\times Z(C_i)\isomorphic U_{\dot{w}_i^{-1}}^-\dot{w}_iP_i = B\dot{w}_iP_i,
\quad
(\nu_i,n_i,z_i)\mapsto x_i=\nu_i\dot{w}_in_iz_i.
\end{equation}
Then
\[
x_iC_ix_i^{-1}
=
\nu_i\dot{w}_in_iC_in_i^{-1}\dot{w}_i^{-1}\nu_i^{-1}
=
\nu_i\dot{w}_in_i'C_i\dot{w}_i^{-1}\nu_i^{-1},
\]
where we use the isomorphism
\[
N_i\xrightarrow[]{\isomorphic} N_i,\quad n_i\mapsto n_i'=n_iC_in_i^{-1}C_i^{-1}.
\]

\medskip

Recall that for $b\in B$, $[b]\in\underline{\FBD}_n$ (see (\ref{eqn:braid_matrix_diagrams_lifting_Borel_elements})), $\beta_{[b]}=\id_n\in\Br_n^+$ and $g_{[b]}=b$. 
Define\footnote{$[x_iC_ix_i^{-1}]'$ is not a priori equal to $[x_iC_ix_i^{-1}]$, since the map $[]$ is not a morphism of monoid. See Proposition \ref{prop:from_matrices_to_braid_matrix_diagrams}.}
\begin{equation}\label{eqn:braid_matrix_diagram_for_a_puncture}
[x_iC_ix_i^{-1}]':=[\nu_i]\circ[\dot{w}_i]\circ[n_i']\circ[C_i]\circ[\dot{w}_i^{-1}]\circ[\nu_i^{-1}]=[\nu_i^{-1}|\dot{w}_i^{-1}|C_i|n_i'|\dot{w}_i|\nu_i]\in\underline{\FBD}_n.
\end{equation}

Recall that $A_j\in B\tau_j B$. Using Proposition \ref{prop:from_matrices_to_braid_matrix_diagrams}, define $[\mM_{\vec{w}}]=[\mM_{\vec{w}}((A_j)_j,(x_i)_i,u_k)]$ by
\begin{equation}
[\mM_{\vec{w}}]:=\prod_{j=1}^g ([A_{2j-1}]\circ[A_{2j}]\circ [A_{2j-1}^{-1}]\circ[A_{2j}^{-1}])\circ \prod_{i=1}^{k-1} ([x_iC_ix_i^{-1}]') \circ [u_kC_k]\in \underline{\FBD}_n.
\end{equation}
Then by (\ref{eqn:Equivariant_Bruhat_cell_of_character_variety}) and (\ref{eqn:reduction_of_character_varieties}), the defining equation for $M_B'(\vec{w})$ reads
\begin{equation}\label{eqn:Bruhat_cell_of_character_variety_via_braid_matrix_diagrams}
[\mM_{\vec{w}}]\weakequivalent \id_n \in\underline{\FBD}_n.
\end{equation}
Moreover, $[\mM_{\vec{w}}]$ has fixed shape
\begin{equation}\label{eqn:shape}
\beta(\vec{w})
=
\prod_{j=1}^g ([\tau_{2j-1}]\circ[\tau_{2j}]\circ[\tau_{2j-1}^{-1}]\circ[\tau_{2j}])
\circ
\prod_{i=1}^{k-1} ([\dot{w}_i]\circ[\dot{w}_i^{-1}])
\in\FBr_n^+.
\end{equation}

\medskip

The next step is to canonicalize $[\mM_{\vec{w}}]$ by diagram calculus. Starting from
\[
[\mM_{\vec{w}}]
=
[C_k|u_k|[x_{k-1}C_{k-1}x_{k-1}^{-1}]'|\cdots|A_2^{-1}|A_1^{-1}|A_2|A_1],
\]
we successively push every handleslide or scaling (i.e.\ factors from $U$ or $T$) to the right as far as possible, using elementary moves in Figure \ref{fig:elementary_moves}.
Later on, we will see that this canonical form naturally leads to braid varieties.

\subsection{Braid varieties}\label{subsec:braid_varieties}

As mentioned above, we now review some basics on braid varieties.

\begin{definition}\label{def:braid_matrices}
The \emph{braid matrix} (resp.\ \emph{$\bmdp$}) with coefficient $\epsilon\in\field$ associated to $\sigma_k$ is\footnote{The matrices $\mB_k(\epsilon)$ already appear in Mellit’s article \cite{Mel25}, where they are denoted by $f_k(\epsilon)$.}
\begin{equation}
\mB_k(\epsilon):=\ms_k\mH_k(\epsilon)\in G,
\qquad
[\mB_k(\epsilon)]:=\sigma_k\circ[\mH_k(\epsilon)]\in\FBD_n
\quad
(\text{Figure \ref{fig:braid_matrix_diagram_with_coefficient_for_a_single_crossing}}).
\end{equation}
For $\beta=\sigma_{i_{\ell}}\cdots\sigma_{i_1}\in\FBr_n^+$ and
$\vec{\epsilon}=(\epsilon_{\ell},\cdots,\epsilon_1)\in\mathbb{A}^{\ell}$, define
\begin{equation}
\mB_{\beta}(\vec{\epsilon}):=\mB_{i_{\ell}}(\epsilon_{\ell})\cdots\mB_{i_1}(\epsilon_1)\in G,
\qquad
[\mB_{\beta}(\vec{\epsilon})]':=[\mB_{i_{\ell}}(\epsilon_{\ell})]\circ\cdots\circ[\mB_{i_1}(\epsilon_1)]\in\FBD_n.
\end{equation}
\end{definition}

\begin{figure}[!htbp]
\vspace{-0.3cm}
\begin{center}
\begin{tikzpicture}[baseline=-.5ex,scale=0.3]
\begin{scope}
\draw[thin] (0,0)--(4,0);
\draw[thin] (0,2)--(1,2) to [out=0,in=180] (3,3)--(4,3);
\draw[white,fill=white] (2,2.5) circle (0.2);
\draw[thin] (0,3)--(1,3) to[out=0,in=180] (3,2)--(4,2);
\draw[thin] (0,5)--(4,5);

\draw[line width=0.5mm] (1,2)--(1,3);
\draw (1.2,2.5) node[left] {\text{\tiny$\epsilon$}};

\draw (3.2,1.3) node[left] {\text{\tiny$\vdots$}};
\draw (3.2,4.3) node[left] {\text{\tiny$\vdots$}};

\draw (0,0) node[left] {\text{\tiny$1$}};
\draw (0,2) node[left] {\text{\tiny$k$}};
\draw (0,3) node[left] {\text{\tiny$k+1$}};
\draw (0,5) node[left] {\text{\tiny$n$}};

\draw (4,0) node[right] {\text{\tiny$1$}};
\draw (4,2) node[right] {\text{\tiny$k$}};
\draw (4,3) node[right] {\text{\tiny$k+1$}};
\draw (4,5) node[right] {\text{\tiny$n$}};

\draw (2,-0.5) node[below] {\text{\tiny$[\mB_k(\epsilon)]$}};
\end{scope}
\end{tikzpicture}
\end{center}
\vspace{-0.5cm}
\caption{The braid matrix diagram with coefficient $\epsilon$ associated to $\sigma_k$.}
\label{fig:braid_matrix_diagram_with_coefficient_for_a_single_crossing}
\end{figure}

As a first application of the diagram calculus of matrices, we obtain the following.

\begin{lemma}[Braid relations for braid matrices]\label{lem:braid_relations_for_braid_matrices}
We have
\begin{eqnarray*}
(a)&&[\mB_i(\epsilon_1)]\circ[\mB_j(\epsilon_2)]
=
[\mB_j(\epsilon_2)]\circ[\mB_i(\epsilon_1)]
\in\BD_n,
\qquad |i-j|>1,\ \epsilon_{\bullet}\in\field;\\
(b)&&[\mB_i(\epsilon_1)]\circ[\mB_{i+1}(\epsilon_2)]\circ[\mB_i(\epsilon_3)]
=
[\mB_{i+1}(\epsilon_3)]\circ[\mB_i(\epsilon_2-\epsilon_3\epsilon_1)]\circ[\mB_{i+1}(\epsilon_1)]
\in\BD_n,
\qquad \epsilon_{\bullet}\in\field.
\end{eqnarray*}
\end{lemma}

\begin{proof}
Part (a) is clear. For part (b), see Figure \ref{fig:braid_relation_for_braid_matrices}. More precisely:
\begin{enumerate}[label=(\roman*),wide,labelwidth=!,labelindent=0pt,itemindent=!]
\item
is a composition of elementary moves $(8)$ in Figure \ref{fig:elementary_moves} and a trivial move switching $\epsilon_1,\epsilon_2$;

\item
is a composition of elementary move $(7)$ in Figure \ref{fig:elementary_moves}, a trivial move switching $-\epsilon_3\epsilon_1,\epsilon_3$, and elementary move $(5)$;

\item
is a braid move;

\item
is a composition of elementary moves $(8)$ in Figure \ref{fig:elementary_moves}.
\end{enumerate}
\end{proof}

\begin{figure}[!htbp]
\vspace{-0.5cm}
\begin{center}
\begin{tikzcd}[row sep=0.5pc,column sep=1pc,ampersand replacement=\&]
%%% diagram 1
\begin{tikzpicture}[baseline=-.5ex,scale=0.5]
\begin{scope}
\draw[thin] (0,0)--(0.5,0) to [out=0,in=180] (2.5,1) to [out=0,in=180] (4.5,2)--(6.5,2);

\draw[white,fill=white] (1.5,0.5) circle(0.1);
\draw[white,fill=white] (3.5,1.5) circle(0.1);

\draw[thin] (0,1)--(0.5,1) to [out=0,in=180] (2.5,0)--(4.5,0) to [out=0,in=180] (6.5,1);

\draw[white,fill=white] (5.5,0.5) circle(0.1);

\draw[thin] (0,2)--(2.5,2) to [out=0,in=180] (4.5,1) to [out=0,in=180] (6.5,0);

\draw (0.3,0) node[left] {\text{\tiny$i$}};
\draw (0.3,1) node[left] {\text{\tiny$i+1$}};
\draw (0.3,2) node[left] {\text{\tiny$i+2$}};

\draw[line width=0.5mm] (0.5,0)--(0.5,1);
\draw (0.5,0) node[below] {\text{\tiny$\epsilon_3$}};

\draw[line width=0.5mm] (2.5,1)--(2.5,2);
\draw (2.5,2) node[above] {\text{\tiny$\epsilon_2$}};

\draw[line width=0.5mm] (4.5,0)--(4.5,1);
\draw (4.5,0) node[below] {\text{\tiny$\epsilon_1$}};
\end{scope}
\end{tikzpicture}

\arrow[r,leftrightarrow,line width=0.1mm,shift left=5,"{\braidequivalent}","{\text{\tiny$(i)$}}"']\&

%%% diagram 2
\begin{tikzpicture}[baseline=-.5ex,scale=0.5]
\begin{scope}
\draw[thin] (-1,0)--(0.5,0) to [out=0,in=180] (2.5,1) to [out=0,in=180] (4.5,2)--(6.5,2);

\draw[white,fill=white] (1.5,0.5) circle(0.1);
\draw[white,fill=white] (3.5,1.5) circle(0.1);

\draw[thin] (-1,1)--(0.5,1) to [out=0,in=180] (2.5,0)--(4.5,0) to [out=0,in=180] (6.5,1);

\draw[white,fill=white] (5.5,0.5) circle(0.1);

\draw[thin] (-1,2)--(2.5,2) to [out=0,in=180] (4.5,1) to [out=0,in=180] (6.5,0);

\draw (-0.7,0) node[left] {\text{\tiny$i$}};
\draw (-0.7,1) node[left] {\text{\tiny$i+1$}};
\draw (-0.7,2) node[left] {\text{\tiny$i+2$}};

\draw[line width=0.5mm] (-0.5,0)--(-0.5,1);
\draw (-0.5,0) node[below] {\text{\tiny$\epsilon_3$}};

\draw[line width=0.5mm] (0,1)--(0,2);
\draw (0,2) node[above] {\text{\tiny$\epsilon_1$}};

\draw[line width=0.5mm] (0.5,0)--(0.5,2);
\draw (0.5,0) node[below] {\text{\tiny$\epsilon_2$}};
\end{scope}
\end{tikzpicture}

\arrow[r,leftrightarrow,line width=0.1mm,shift left=5,"{\braidequivalent}","{\text{\tiny$(ii)$}}"']\&

%%% diagram 3
\begin{tikzpicture}[baseline=-.5ex,scale=0.5]
\begin{scope}
\draw[thin] (-1,0)--(0.5,0) to [out=0,in=180] (2.5,1) to [out=0,in=180] (4.5,2)--(6.5,2);

\draw[white,fill=white] (1.5,0.5) circle(0.1);
\draw[white,fill=white] (3.5,1.5) circle(0.1);

\draw[thin] (-1,1)--(0.5,1) to [out=0,in=180] (2.5,0)--(4.5,0) to [out=0,in=180] (6.5,1);

\draw[white,fill=white] (5.5,0.5) circle(0.1);

\draw[thin] (-1,2)--(2.5,2) to [out=0,in=180] (4.5,1) to [out=0,in=180] (6.5,0);

\draw (-0.7,0) node[left] {\text{\tiny$i$}};
\draw (-0.7,1) node[left] {\text{\tiny$i+1$}};
\draw (-0.7,2) node[left] {\text{\tiny$i+2$}};

\draw[line width=0.5mm] (-0.5,1)--(-0.5,2);
\draw (-0.5,2) node[above] {\text{\tiny$\epsilon_1$}};

\draw[line width=0.5mm] (0,0)--(0,2);
\draw (1,2) node[above] {\text{\tiny$\epsilon_2-\epsilon_3\epsilon_1$}};

\draw[line width=0.5mm] (0.5,0)--(0.5,1);
\draw (0.5,0) node[below] {\text{\tiny$\epsilon_3$}};
\end{scope}
\end{tikzpicture}
\end{tikzcd}
\end{center}

\begin{center}
\begin{tikzcd}[row sep=0.5pc,column sep=1pc,ampersand replacement=\&]

\arrow[r,leftrightarrow,line width=0.1mm,shift left=5,"{\braidequivalent}","{\text{\tiny$(iii)$}}"']\&

%%% diagram 4
\begin{tikzpicture}[baseline=-.5ex,scale=0.5]
\begin{scope}
\draw[thin] (-1,0)--(2.5,0) to [out=0,in=180] (4.5,1) to [out=0,in=180] (6.5,2);

\draw[white,fill=white] (3.5,0.5) circle(0.1);
\draw[white,fill=white] (5.5,1.5) circle(0.1);

\draw[thin] (-1,1)--(0.5,1) to [out=0,in=180] (2.5,2)--(4.5,2) to [out=0,in=180] (6.5,1);

\draw[white,fill=white] (1.5,1.5) circle(0.1);

\draw[thin] (-1,2)--(0.5,2) to [out=0,in=180] (2.5,1) to [out=0,in=180] (4.5,0)--(6.5,0);

\draw (-0.7,0) node[left] {\text{\tiny$i$}};
\draw (-0.7,1) node[left] {\text{\tiny$i+1$}};
\draw (-0.7,2) node[left] {\text{\tiny$i+2$}};

\draw[line width=0.5mm] (-0.5,1)--(-0.5,2);
\draw (-0.5,2) node[above] {\text{\tiny$\epsilon_1$}};

\draw[line width=0.5mm] (0,0)--(0,2);
\draw (1,2) node[above] {\text{\tiny$\epsilon_2-\epsilon_3\epsilon_1$}};

\draw[line width=0.5mm] (0.5,0)--(0.5,1);
\draw (0.5,0) node[below] {\text{\tiny$\epsilon_3$}};
\end{scope}
\end{tikzpicture}

\arrow[r,leftrightarrow,line width=0.1mm,shift left=5,"{\braidequivalent}","{\text{\tiny$(iv)$}}"']\&

%%% diagram 5
\begin{tikzpicture}[baseline=-.5ex,scale=0.5]
\begin{scope}
\draw[thin] (0,0)--(2.5,0) to [out=0,in=180] (4.5,1) to [out=0,in=180] (6.5,2);

\draw[white,fill=white] (3.5,0.5) circle(0.1);
\draw[white,fill=white] (5.5,1.5) circle(0.1);

\draw[thin] (0,1)--(0.5,1) to [out=0,in=180] (2.5,2)--(4.5,2) to [out=0,in=180] (6.5,1);

\draw[white,fill=white] (1.5,1.5) circle(0.1);

\draw[thin] (0,2)--(0.5,2) to [out=0,in=180] (2.5,1) to [out=0,in=180] (4.5,0)--(6.5,0);

\draw (0.3,0) node[left] {\text{\tiny$i$}};
\draw (0.3,1) node[left] {\text{\tiny$i+1$}};
\draw (0.3,2) node[left] {\text{\tiny$i+2$}};

\draw[line width=0.5mm] (4.5,1)--(4.5,2);
\draw (4.5,2) node[above] {\text{\tiny$\epsilon_3$}};

\draw[line width=0.5mm] (2.5,0)--(2.5,1);
\draw (2.5,0) node[below] {\text{\tiny$\epsilon_2-\epsilon_3\epsilon_1$}};

\draw[line width=0.5mm] (0.5,1)--(0.5,2);
\draw (0.5,2) node[above] {\text{\tiny$\epsilon_1$}};
\end{scope}
\end{tikzpicture}
\end{tikzcd}
\end{center}
\vspace{-0.5cm}
\caption{Braid relation for braid matrix diagrams with coefficients.}
\label{fig:braid_relation_for_braid_matrices}
\end{figure}

\begin{definition}[Braid varieties]\label{def:braid_varieties}
Let $\beta=\sigma_{i_{\ell}}\cdots\sigma_{i_1}\in\FBr_n^+$. For each $1\leq j\leq \ell$, denote
\begin{equation}\label{eqn:canonical_morphism_for_braid_matrix_with_coefficients}
f_j:\mathbb{A}^j\rightarrow G,\qquad
(\epsilon_j,\cdots,\epsilon_1)\mapsto \mB_{i_j}(\epsilon_j)\cdots\mB_{i_1}(\epsilon_1).
\end{equation}

\begin{enumerate}[wide,labelwidth=!,labelindent=0pt]
\item
The \emph{(resp.\ restricted) braid variety} associated to $\beta$ is the closed subvariety of $\mathbb{A}^{\ell}$ defined by
\begin{equation}
X(\beta):=f_{\ell}^{-1}(B),
\quad
\text{resp.\ }
X(\beta,C):=\mumon^{-1}(C),~~C\in T,
\quad
\mumon:=D\circ f_{\ell}:X(\beta)\rightarrow B\rightarrow T.
\end{equation}
By Lemma \ref{lem:braid_relations_for_braid_matrices}, $X(\beta)$ (resp.\ $X(\beta,C)$) depends only on $\beta\in\Br_n^+$, up to a canonical isomorphism.

\item
An element $b\in B$ \emph{acts} on $X(\beta)\ni\vec{\epsilon}$ as follows.
We define
\[
\hat{\vec{\epsilon}}=(\hat{\epsilon}_{\ell},\cdots,\hat{\epsilon}_1):=b\cdot\vec{\epsilon}
\]
to be the unique tuple such that
\begin{equation}
[\mB_{i_{\ell}}(\epsilon_{\ell})]\circ\cdots\circ[\mB_{i_1}(\epsilon_1)]\circ[b^{-1}]
=
[\tilde{b}_{\ell}]\circ[\mB_{i_{\ell}}(\hat{\epsilon}_{\ell})]\circ\cdots\circ[\mB_{i_1}(\hat{\epsilon}_1)]
\in\underline{\FBD}_n,
\qquad
\tilde{b}_{\ell}\in B.
\end{equation}
The uniqueness follows from Proposition \ref{prop:equality_in_underline_FBD_n}. The existence is obtained by the following diagrammatic procedure.

In the diagrammatic notation of $\underline{\FBD}_n$, start from
$
[b^{-1}|\mB_{i_1}(\epsilon_1)|\cdots|\mB_{i_{\ell}}(\epsilon_{\ell})]
$
and push $[b^{-1}]$ to the right by elementary moves in Figure \ref{fig:elementary_moves} as far as possible. The resulting diagram has the form
$
[\mB_{i_1}(\hat{\epsilon}_1)|\cdots|\mB_{i_{\ell}}(\hat{\epsilon}_{\ell})|\tilde{b}_{\ell}].
$

More concretely, set $\tilde{b}_0:=b^{-1}$. Inductively, suppose that for some $0\leq k<\ell$, we have constructed $\hat{\epsilon}_1,\dots,\hat{\epsilon}_k\in\field$ and $\tilde{b}_k\in B$ such that
\begin{equation}\label{eqn:inductive_step_for_B_action_on_braid_variety}
[\tilde{b}_0|\mB_{i_1}(\epsilon_1)|\cdots|\mB_{i_k}(\epsilon_k)]
=
[\mB_{i_1}(\hat{\epsilon}_1)|\cdots|\mB_{i_k}(\hat{\epsilon}_k)|\tilde{b}_k].
\end{equation}
Using the decomposition
$
B\ms_{i_{k+1}}B=B\ms_{i_{k+1}}U^-_{\ms_{i_{k+1}}},
$
there exists a unique pair $\hat{\epsilon}_{k+1}\in\field$ and $\tilde{b}_{k+1}\in B$ such that
\begin{equation}\label{eqn:local_step_for_B_action_on_braid_variety}
[\tilde{b}_k|\mB_{i_{k+1}}(\epsilon_{k+1})]
=
[\mB_{i_{k+1}}(\hat{\epsilon}_{k+1})|\tilde{b}_{k+1}].
\end{equation}
This completes the induction. See Example \ref{ex:B-action_on_braid_variety} for an illustration.
\end{enumerate}
\end{definition}

\begin{example}[A sample computation for the $B$-action]\label{ex:B-action_on_braid_variety}
Let $n=3$ and $\beta=\sigma_2\sigma_1\in\Br_3^+$. For any $b\in B$, we may write $b^{-1}$ uniquely as
\[
b^{-1}
=
{\tiny\left(\begin{array}{ccc}
1 & \lambda_{12} & \lambda_{13}\\
0 & 1 & \lambda_{23}\\
0 & 0 & 1
\end{array}\right)
\left(\begin{array}{ccc}
\lambda_1 & 0 & 0\\
0 & \lambda_2 & 0\\
0 & 0 & \lambda_3
\end{array}\right)
=
\left(\begin{array}{ccc}
1 & 0 & 0\\
0 & 1 & \lambda_{23}\\
0 & 0 & 1
\end{array}\right)
\left(\begin{array}{ccc}
1 & 0 & \lambda_{13}\\
0 & 1 & 0\\
0 & 0 & 1
\end{array}\right)
\left(\begin{array}{ccc}
1 & \lambda_{12} & 0\\
0 & 1 & 0\\
0 & 0 & 1
\end{array}\right)
\left(\begin{array}{ccc}
\lambda_1 & 0 & 0\\
0 & \lambda_2 & 0\\
0 & 0 & \lambda_3
\end{array}\right)}
\]
with $\lambda_1,\lambda_2,\lambda_3\in\field^\times$ and $\lambda_{12},\lambda_{13},\lambda_{23}\in\field$. Equivalently,
\[
[b^{-1}]
=
[\mK_1(\lambda_1)|\mK_2(\lambda_2)|\mK_3(\lambda_3)|
\mH_{1,2}(\lambda_{12})|\mH_{1,3}(\lambda_{13})|\mH_{2,3}(\lambda_{23})]
\in\underline{\FBD}_3.
\]
By Lemma \ref{lem:elementary_moves}, $[\mH_{1,3}(\lambda_{13})]$ commutes with $[\mH_{1,2}(\lambda_{12})]$ and $[\mH_{2,3}(\lambda_{23})]$; these trivial moves are omitted in Figure \ref{fig:B-action_on_braid_variety}.

Let $\vec{\epsilon}=(\epsilon_2,\epsilon_1)\in \mathbb{A}^2$. The action of $b$ on $\vec{\epsilon}$ is illustrated in Figure \ref{fig:B-action_on_braid_variety}. Write $\tilde{b}_0:=b^{-1}$. We start with
\[
[\tilde{b}_0|\mB_1(\epsilon_1)|\mB_2(\epsilon_2)].
\]
\begin{enumerate}[wide,labelwidth=!,labelindent=0pt,itemindent=!,label=(\roman*)]
\item
Apply elementary moves $(6)$, trivial moves, $(5)$, $(2)$, and $(3)$; this corresponds to the decomposition $B=U_{\ms_1}^+TU_{\ms_1}^-$. We obtain the second diagram.

\item
Apply elementary moves $(8)$ and $(4)$. We obtain the third diagram, namely
$
[\mB_1(\hat{\epsilon}_1)|\tilde{b}_1|\mB_2(\epsilon_2)].
$
This completes the first inductive step.

\item
Apply a trivial move and elementary moves $(5)$, $(2)$, and $(3)$; this corresponds to the decomposition $B=U_{\ms_2}^+TU_{\ms_2}^-$. We obtain the fourth diagram.

\item
Apply elementary moves $(8)$ and $(4)$. We obtain the last diagram, namely
$
[\mB_1(\hat{\epsilon}_1)|\mB_2(\hat{\epsilon}_2)|\tilde{b}_2].
$
This completes the second inductive step.
\end{enumerate}

Thus we obtain an identity in $\underline{\FBD}_3$:
\[
[b^{-1}|\mB_1(\epsilon_1)|\mB_2(\epsilon_2)]
=
[\mB_1(\hat{\epsilon}_1)|\mB_2(\hat{\epsilon}_2)|\tilde{b}_2].
\]
Hence
\[
b\cdot(\epsilon_2,\epsilon_1)
=
(\hat{\epsilon}_2,\hat{\epsilon}_1)
=
\bigl(\lambda_3\lambda_1^{-1}(\epsilon_2+\lambda_{13}+\epsilon_1\lambda_{23}),
\lambda_2\lambda_1^{-1}(\epsilon_1+\lambda_{12})\bigr).
\]
This illustrates concretely how the tuple $b\cdot\vec{\epsilon}$ is obtained by pushing $b^{-1}$ to the right through the braid word via elementary moves.
\end{example}

\begin{figure}[!htbp]
\vspace{-0.5cm}
\begin{center}
\begin{tikzcd}[row sep=0.5pc,column sep=1pc,ampersand replacement=\&]
%%% diagram 1
\begin{tikzpicture}[baseline=-.5ex,scale=0.5]
\begin{scope}

\draw[thin] (-1,0)--(2.5,0) to [out=0,in=180] (4.5,1) to [out=0,in=180] (6.5,2);

\draw[white,fill=white] (3.5,0.5) circle(0.1);
\draw[white,fill=white] (5.5,1.5) circle(0.1);

\draw[thin] (-1,1)--(2.5,1) to [out=0,in=180] (4.5,0)--(6.5,0);

\draw[thin] (-1,2)--(4.5,2) to [out=0,in=180] (6.5,1);

\draw (-0.7,0) node[left] {\text{\tiny$1$}};
\draw (-0.7,1) node[left] {\text{\tiny$2$}};
\draw (-0.7,2) node[left] {\text{\tiny$3$}};

\draw[line width=0.5mm] (0.5,0)--(0.5,1);
\node at (0.5,-0.4) {\text{\tiny$\lambda_{12}$}};

\draw[line width=0.5mm] (1,0)--(1,2);
\node at (1,2.2) {\text{\tiny$\lambda_{13}$}};

\draw[line width=0.5mm] (1.5,1)--(1.5,2);
\node at (1.5,0.6) {\text{\tiny$\lambda_{23}$}};

\node at (-0.5,-0.1) {\text{\tiny$\ast$}};
\node at (-0.5,0.3) {\text{\tiny$\lambda_1$}};

\node at (-0.5,0.9) {\text{\tiny$\ast$}};
\node at (-0.5,1.3) {\text{\tiny$\lambda_2$}};

\node at (-0.5,1.9) {\text{\tiny$\ast$}};
\node at (-0.5,2.3) {\text{\tiny$\lambda_3$}};

\draw[thick,green] (-0.9,-0.7) rectangle (1.9,2.7);
\node[green] at (0.5,-1.5) {\text{\tiny$[\tilde{b}_0]$}};

\draw[line width=0.5mm] (2.5,0)--(2.5,1);
\draw (2.5,0) node[below] {\text{\tiny$\epsilon_1$}};

\draw[thick,green] (2.2,-0.7) rectangle (3.8,2.7);
\node[green] at (3,-1.5) {\text{\tiny$[\mB_1(\epsilon_1)]$}};

\draw[line width=0.5mm] (4.5,1)--(4.5,2);
\draw (4.5,2) node[above] {\text{\tiny$\epsilon_2$}};

\draw[thick,green] (4.2,-0.7) rectangle (6,2.7);
\node[green] at (5.1,-1.5) {\text{\tiny$[\mB_2(\epsilon_2)]$}};
\end{scope}
\end{tikzpicture}

\arrow[r,equal,shift left=5,"{\text{\tiny$(i)$}}"]\&

%%% diagram 2
\begin{tikzpicture}[baseline=-.5ex,scale=0.5]
\begin{scope}

\draw[thin] (-1,0)--(2.5,0) to [out=0,in=180] (4.5,1) to [out=0,in=180] (6.5,2);

\draw[white,fill=white] (3.5,0.5) circle(0.1);
\draw[white,fill=white] (5.5,1.5) circle(0.1);

\draw[thin] (-1,1)--(2.5,1) to [out=0,in=180] (4.5,0)--(6.5,0);

\draw[thin] (-1,2)--(4.5,2) to [out=0,in=180] (6.5,1);

\draw (-0.7,0) node[left] {\text{\tiny$1$}};
\draw (-0.7,1) node[left] {\text{\tiny$2$}};
\draw (-0.7,2) node[left] {\text{\tiny$3$}};

\draw[line width=0.5mm] (-0.5,0)--(-0.5,1);
\node at (0,-0.7) {\text{\tiny$\hat{\epsilon}_1 = \lambda_2\lambda_1^{-1}(\epsilon_1+\lambda_{12})$}};

\draw[line width=0.5mm] (1.5,0)--(1.5,2);
\node at (2,2.2) {\text{\tiny$\lambda_{13} + \epsilon_1\lambda_{23}$}};

\draw[line width=0.5mm] (2.5,1)--(2.5,2);
\node at (2.5,0.6) {\text{\tiny$\lambda_{23}$}};

\node at (0.5,-0.1) {\text{\tiny$\ast$}};
\node at (0.5,0.3) {\text{\tiny$\lambda_1$}};

\node at (0.5,0.9) {\text{\tiny$\ast$}};
\node at (0.5,1.3) {\text{\tiny$\lambda_2$}};

\node at (0.5,1.9) {\text{\tiny$\ast$}};
\node at (0.5,2.3) {\text{\tiny$\lambda_3$}};

\draw[line width=0.5mm] (4.5,1)--(4.5,2);
\draw (4.5,2) node[above] {\text{\tiny$\epsilon_2$}};

\draw[thick,green] (4.2,-0.7) rectangle (6,2.7);
\node[green] at (5.1,-1.5) {\text{\tiny$[\mB_2(\epsilon_2)]$}};
\end{scope}
\end{tikzpicture}

\arrow[r,equal,shift left=5,"{\text{\tiny$(ii)$}}"]\&

%%% diagram 3
\begin{tikzpicture}[baseline=-.5ex,scale=0.5]
\begin{scope}

\draw[thin] (-1,0)--(-0.5,0) to [out=0,in=180] (1.5,1)--(4.5,1) to [out=0,in=180] (6.5,2);

\draw[white,fill=white] (0.5,0.5) circle(0.1);
\draw[white,fill=white] (5.5,1.5) circle(0.1);

\draw[thin] (-1,1)--(-0.5,1) to [out=0,in=180] (1.5,0)--(6.5,0);

\draw[thin] (-1,2)--(4.5,2) to [out=0,in=180] (6.5,1);

\draw (-0.7,0) node[left] {\text{\tiny$1$}};
\draw (-0.7,1) node[left] {\text{\tiny$2$}};
\draw (-0.7,2) node[left] {\text{\tiny$3$}};

\draw[line width=0.5mm] (-0.5,0)--(-0.5,1);
\node at (-0.5,-0.5) {\text{\tiny$\hat{\epsilon}_1$}};

\draw[thick,green] (-0.9,-0.7) rectangle (0.9,2.7);
\node[green] at (0,-1.5) {\text{\tiny$[\mB_1(\hat{\epsilon}_1)]$}};

\draw[line width=0.5mm] (2.5,1)--(2.5,2);
\node at (3,2.2) {\text{\tiny$\lambda_{13} + \epsilon_1\lambda_{23}$}};

\draw[line width=0.5mm] (3,0)--(3,2);
\node at (3,-0.4) {\text{\tiny$\lambda_{23}$}};

\node at (1.5,-0.1) {\text{\tiny$\ast$}};
\node at (1.5,0.3) {\text{\tiny$\lambda_2$}};

\node at (1.5,0.9) {\text{\tiny$\ast$}};
\node at (1.5,1.3) {\text{\tiny$\lambda_1$}};

\node at (1.5,1.9) {\text{\tiny$\ast$}};
\node at (1.5,2.3) {\text{\tiny$\lambda_3$}};

\draw[thick,green] (1.2,-0.7) rectangle (4.2,2.7);
\node[green] at (2.7,-1.5) {\text{\tiny$[\tilde{b}_1]$}};

\draw[line width=0.5mm] (4.5,1)--(4.5,2);
\draw (4.6,2) node[above] {\text{\tiny$\epsilon_2$}};

\draw[thick,green] (4.3,-0.7) rectangle (6.1,2.7);
\node[green] at (5.2,-1.5) {\text{\tiny$[\mB_2(\epsilon_2)]$}};
\end{scope}
\end{tikzpicture}
\end{tikzcd}
\end{center}

\begin{center}
\begin{tikzcd}[row sep=0.5pc,column sep=1pc,ampersand replacement=\&]

\arrow[r,equal,shift left=5,"{\text{\tiny$(iii)$}}"]\&

%%% diagram 4
\begin{tikzpicture}[baseline=-.5ex,scale=0.5]
\begin{scope}

\draw[thin] (-1,0)--(-0.5,0) to [out=0,in=180] (1.5,1)--(4.5,1) to [out=0,in=180] (6.5,2);

\draw[white,fill=white] (0.5,0.5) circle(0.1);
\draw[white,fill=white] (5.5,1.5) circle(0.1);

\draw[thin] (-1,1)--(-0.5,1) to [out=0,in=180] (1.5,0)--(6.5,0);

\draw[thin] (-1,2)--(4.5,2) to [out=0,in=180] (6.5,1);

\draw (-0.7,0) node[left] {\text{\tiny$1$}};
\draw (-0.7,1) node[left] {\text{\tiny$2$}};
\draw (-0.7,2) node[left] {\text{\tiny$3$}};

\draw[line width=0.5mm] (-0.5,0)--(-0.5,1);
\node at (-0.5,-0.5) {\text{\tiny$\hat{\epsilon}_1$}};

\draw[thick,green] (-0.9,-0.7) rectangle (0.9,2.2);
\node[green] at (0,-1.5) {\text{\tiny$[\mB_1(\hat{\epsilon}_1)]$}};

\draw[line width=0.5mm] (1.5,1)--(1.5,2);
\node at (1.5,2.5) {\text{\tiny$\hat{\epsilon}_2 = \lambda_3\lambda_1^{-1}(\epsilon_2+\lambda_{13} + \epsilon_1\lambda_{23})$}};

\draw[line width=0.5mm] (4.5,0)--(4.5,2);
\node at (4.5,-0.4) {\text{\tiny$\lambda_{23}$}};

\node at (3,-0.1) {\text{\tiny$\ast$}};
\node at (3,-0.5) {\text{\tiny$\lambda_2$}};

\node at (3,0.9) {\text{\tiny$\ast$}};
\node at (3,0.5) {\text{\tiny$\lambda_1$}};

\node at (3,1.9) {\text{\tiny$\ast$}};
\node at (3,1.5) {\text{\tiny$\lambda_3$}};
\end{scope}
\end{tikzpicture}

\arrow[r,equal,shift left=5,"{\text{\tiny$(iv)$}}"]\&

%%% diagram 5
\begin{tikzpicture}[baseline=-.5ex,scale=0.5]
\begin{scope}

\draw[thin] (-1,0)--(-0.5,0) to [out=0,in=180] (1.5,1) to [out=0,in=180] (3.5,2)--(6.5,2);

\draw[white,fill=white] (0.5,0.5) circle(0.1);
\draw[white,fill=white] (2.5,1.5) circle(0.1);

\draw[thin] (-1,1)--(-0.5,1) to [out=0,in=180] (1.5,0)--(6.5,0);

\draw[thin] (-1,2)--(1.5,2) to [out=0,in=180] (3.5,1)--(6.5,1);

\draw (-0.7,0) node[left] {\text{\tiny$1$}};
\draw (-0.7,1) node[left] {\text{\tiny$2$}};
\draw (-0.7,2) node[left] {\text{\tiny$3$}};

\draw[line width=0.5mm] (-0.5,0)--(-0.5,1);
\node at (-0.5,-0.5) {\text{\tiny$\hat{\epsilon}_1$}};

\draw[thick,green] (-0.9,-0.7) rectangle (0.9,2.7);
\node[green] at (0,-1.5) {\text{\tiny$[\mB_1(\hat{\epsilon}_1)]$}};

\draw[line width=0.5mm] (1.5,1)--(1.5,2);
\node at (1.6,2.3) {\text{\tiny$\hat{\epsilon}_2$}};

\draw[thick,green] (1.2,-0.7) rectangle (3.2,2.7);
\node[green] at (2.2,-1.5) {\text{\tiny$[\mB_2(\hat{\epsilon}_2)]$}};

\draw[line width=0.5mm] (6,0)--(6,1);
\node at (6,-0.4) {\text{\tiny$\lambda_{23}$}};

\node at (4.5,-0.1) {\text{\tiny$\ast$}};
\node at (4.5,-0.5) {\text{\tiny$\lambda_2$}};

\node at (4.5,0.9) {\text{\tiny$\ast$}};
\node at (4.5,0.5) {\text{\tiny$\lambda_3$}};

\node at (4.5,1.9) {\text{\tiny$\ast$}};
\node at (4.5,1.5) {\text{\tiny$\lambda_1$}};

\draw[thick,green] (3.8,-0.7) rectangle (6.4,2.7);
\node[green] at (5.1,-1.5) {\text{\tiny$[\tilde{b}_2]$}};
\end{scope}
\end{tikzpicture}
\end{tikzcd}
\end{center}
\vspace{-0.5cm}
\caption{The $B$-action associated to $\beta$ on $\mathbb{A}^{\ell}$. Here $n=3$, $\ell=2$, and $\beta=\sigma_2\sigma_1$. By elementary moves in Figure \ref{fig:elementary_moves}, we obtain an identity in $\underline{\FBD}_3$ from the first diagram
$
[b^{-1}|\mB_1(\epsilon_1)|\mB_2(\epsilon_2)]
$
to the last diagram
$
[\mB_1(\hat{\epsilon}_1)|\mB_2(\hat{\epsilon}_2)|\tilde{b}_2].
$
Hence
$
b\cdot(\epsilon_2,\epsilon_1)=(\hat{\epsilon}_2,\hat{\epsilon}_1).
$}
\label{fig:B-action_on_braid_variety}
\end{figure}

\noindent\textbf{Convention} \setword{{\color{blue}$4$}}{convention:group_action}:
if the context is clear, for a group action of $h\in H$ on a variety $Y\ni y$, we denote
\begin{equation}\label{eqn:convention_for_an_action}
\hat{y}:=h\cdot y.
\end{equation}

Next, we recall the cell decomposition of braid varieties.
Fix $\beta=\sigma_{i_{\ell}}\cdots\sigma_{i_1}\in\FBr_n^+$. Define
\begin{equation}
p:\mathbb{A}^{\ell}\rightarrow W^{\ell+1}:\vec{\epsilon}\mapsto p(\vec{\epsilon})=(p_{\ell},\cdots,p_0),
\qquad
\mB_{i_m}(\epsilon_m)\cdots\mB_{i_1}(\epsilon_1)\in Bp_mB.
\end{equation}
Alternatively, by Proposition \ref{prop:from_matrices_to_braid_matrix_diagrams}, we have
\[
[\mB_{i_m}(\epsilon_m)\cdots\mB_{i_1}(\epsilon_1)]\in [B]\circ[p_m]\circ[B]\subset \BD_n.
\]

\begin{definition}[{\cite[\S.5.4]{Mel25}}]\label{def:walks}
Let $p=(p_{\ell},\cdots,p_0)\in W^{\ell+1}$.
Recall that throughout we use the Bruhat order on the Weyl group $W$.
We say that $p$ is a \emph{walk} of $\beta$ if $p_0=p_{\ell}=\id$ and, for each position $1\leq m\leq \ell$,
\begin{equation*}
p_m=
\begin{cases}
\ms_{i_m}p_{m-1}\quad (\text{go-up}) & \text{if } \ms_{i_m}p_{m-1}>p_{m-1},\\
\ms_{i_m}p_{m-1}\quad (\text{go-down})\ \text{ or }\ p_{m-1}\quad (\text{stay}) & \text{if } \ms_{i_m}p_{m-1}<p_{m-1}.
\end{cases}
\end{equation*}
We denote
\[
U_p:=\{\text{go-up positions}\},\qquad
S_p:=\{\text{stay positions}\},\qquad
D_p:=\{\text{go-down positions}\}.
\]
Then
$
[\ell]=\{1,\cdots,\ell\}=U_p\sqcup D_p\sqcup S_p.
$
By a length count,
$
|U_p|-|D_p|=\ell(p_{\ell})-\ell(p_0)=0.
$
We write
\[
\cW(\beta):=\{\text{walks of }\beta\}.
\]
\end{definition}

For any $1\leq m\leq \ell=\ell(\beta)$, denote
\begin{equation}\label{eqn:inductive_underlying_permutations}
\ms_{<m}(\beta):= \prod_{q=m-1}^1\ms_{i_q},
\qquad
\ms_{>m}(\beta):= \prod_{q=\ell}^{m+1} \ms_{i_q}.
\end{equation}
We use Convention \ref{convention:group_action}. Recall that $x^y=yxy^{-1}$ in $G$, and write
\[
t=\diag(t_1,\cdots,t_n)\in T.
\]

\begin{proposition}\label{prop:cell_decomposition_of_braid_varieties}
We have \emph{$B$-equivariant} decompositions into locally closed subvarieties:
\begin{eqnarray}\label{eqn:cell_decomposition_of_braid_varieties}
\hspace{-4em}&&X(\beta)=\bigsqcup_{p\in\cW(\beta)} X_p(\beta),
\qquad
\varphi:X_p(\beta):=X_p^{\ell}(\beta)\xrightarrow{\ \isomorphic \ } (\field^{\times})^{|S_p|}\times \field^{|U_p|}:\vec{\epsilon}\mapsto (\epsilon_m')_{m\in S_p\sqcup U_p},\\
\hspace{-4em}&&X(\beta,C)=\bigsqcup_{p\in\cW(\beta)} X_p(\beta,C),
\qquad
X_p(\beta,C):=X_p(\beta)\cap X(\beta,C),\nonumber
\end{eqnarray}
such that
\begin{enumerate}[wide,labelwidth=!,labelindent=0pt]
\item
The inherited action of $b\in B$ on $(\epsilon_m')_{m\in S_p\sqcup U_p}\in(\field^{\times})^{|S_p|}\times\field^{|U_p|}$ satisfies:
\begin{enumerate}
\item
If $b=u\in U\subset B$, then
\[
\hat{\epsilon}_m'=\epsilon_m',
\qquad
\forall\,m\in S_p.
\]

\item
If $b=t\in T\subset B$, then
\begin{equation}
\hat{\epsilon}_m'=(t^{p_{m-1}})_{i_m}(t^{p_{m-1}})_{i_m+1}^{-1}\epsilon_m',
\qquad
\forall\,m\in S_p\sqcup U_p.
\end{equation}
\end{enumerate}

\item
Under the isomorphism $\varphi$, the map $\mumon:X_p(\beta)\to T$ is given by
\begin{equation}\label{eqn:formula_for_mumon}
\mumon((\epsilon_m')_{m\in S_p\sqcup U_p})
=
\mumon((\epsilon_m')_{m\in S_p})
=
\prod_{m\in S_p}
(\mK_{i_m}(-\epsilon_m'^{-1})\mK_{i_m+1}(\epsilon_m'))^{\ms_{>m}(\beta)}.
\end{equation}
In particular,
$
\det(\mumon((\epsilon_m')_{m\in S_p\sqcup U_p}))=(-1)^{|S_p|}.
$
\end{enumerate}
\end{proposition}

\begin{proof}
The proof follows from the diagram calculus developed above; for the details, see Appendix \ref{sec:cell_decomposition_of_braid_varieties}.
\end{proof}

\section{Cell decomposition of character varieties}\label{sec:cell_decomposition_of_character_varieties}

In this section, we prove a strong form (Theorem \ref{thm:cell_decomposition_of_very_generic_character_varieties}) of A. Mellit's cell decomposition for very generic character varieties \cite{Mel25}. This will be used to prove our main theorem, Theorem \ref{thm:weak_geometric_P=W_conjecture_for_very_generic_character_varieties}.

Recall that we obtained a decomposition (see \eqref{eqn:Bruhat_cell_of_character_variety} and \eqref{eqn:Equivariant_Bruhat_cell_of_character_variety})
\[
\modulispace_{\type}
=
\bigsqcup_{\vec{w}\in W^{2g}\times\prod_{i=1}^{k-1}W/W(C_i)}
\modulispace_{\type}(\vec{w}),
\qquad
\modulispace_{\type}(\vec{w})=M_B'(\vec{w})/PB_{\pa},
\]
and the defining equation of $M_B'(\vec{w})$ was interpreted via braid matrix diagrams as \eqref{eqn:Bruhat_cell_of_character_variety_via_braid_matrix_diagrams}:
\[
[\mM_{\vec{w}}]
=
\prod_{j=1}^g
([A_{2j-1}]\circ[A_{2j}]\circ[A_{2j-1}^{-1}]\circ[A_{2j}^{-1}])
\circ
\prod_{i=1}^{k-1}
([x_iC_ix_i^{-1}]')
\circ
[u_k]\circ[C_k]
\weakequivalent \id_n
\in \underline{\FBD}_n.
\]
To obtain the actual cell decomposition of $\modulispace_{\type}$, we would like to decompose $\modulispace_{\type}(\vec{w})$ further. This amounts to decomposing $M_B'(\vec{w})$ equivariantly. For this, as explained at the end of Section \ref{subsec:diagram_calculus_of_matrices}, the next step is to canonicalize $[\mM_{\vec{w}}]$ by diagram calculus.

This will be carried out in the next three subsections: Section \ref{subsec:diagram_calculus_for_punctures} and Section \ref{subsec:diagram_calculus_for_genera} treat the puncture and genus calculations respectively; Section \ref{subsec:connection_to_braid_varieties} combines these local calculations to describe $M_B'(\vec{w})$ in terms of braid varieties. Finally, in Section \ref{subsec:the_cell_decomposition}, we prove a strong form (Theorem \ref{thm:cell_decomposition_of_very_generic_character_varieties}) of the cell decomposition for $\modulispace_{\type}$.

To simplify the computations, we use the following \textbf{device}.
\begin{itemize}
\item
Fix an embedding $S_n\hookrightarrow \FBr_n^+\subset \FBD_n$ lifting the map of sets $[-]:S_n\rightarrow\Br_n^+$.
\item
Via Lemma \ref{lem:closed_subgroups_of_U_via_handleslides}, also fix an embedding
$
[-]:B=U\rtimes T\hookrightarrow \FBD_n
$
(not as a morphism of monoids) lifting the map
$
[-]:B\rightarrow\underline{\FBD}_n.
$
\end{itemize}

\begin{definition/proposition}\label{def/prop:trick}
Fix a subset $\cS\subset [N]=\{1,\cdots,N\}$. Let $Y_1,\cdots,Y_N$ be such that:
for $i\notin\cS$, $Y_i\subset B$ is a locally closed $\field$-subvariety; and for $i\in\cS$, $Y_i$ is a singleton consisting of a permutation\footnote{By a slight abuse of notation, we identify $w\in S_n$ with the singleton $\{w\}\subset S_n$. Similarly, we identify $b\in B$ with the singleton $\{b\}\subset B$.} in $S_n$.
\begin{enumerate}[wide,labelwidth=!,labelindent=0pt]
\item
For any $\field$-variety $Z$, define
\[
[Y_1]\circ\cdots\circ[Y_N]\times Z
:=
\{[y_1]\circ\cdots\circ[y_N]\in\FBD_n: y_i\in Y_i\}\times Z.
\]
As a $\field$-variety,
$
[Y_1]\circ\cdots\circ[Y_N]\times Z
\cong
\prod_{i\notin\cS} Y_i\times Z.
$
Define the natural map
\[
E:[Y_1]\circ\cdots\circ[Y_N]\times Z
\longrightarrow
[Y_1]\circ\cdots\circ[Y_N]
\hookrightarrow
\FBD_n
\longrightarrow
\underline{\FBD}_n.
\]

\item
If
$
\varphi:[Y_1]\circ\cdots\circ[Y_N]\times Z
\longrightarrow
[Y_1']\circ\cdots\circ[Y_M']\times Z'
$
is an isomorphism of $\field$-varieties such that the induced maps to $\underline{\FBD}_n$ agree via $E$, then we say that $\varphi$ is \textbf{elementary}. In this case, we write
\[
[Y_1]\circ\cdots\circ[Y_N]\times Z
\isomorphic
[Y_1']\circ\cdots\circ[Y_M']\times Z'.
\]
In particular, any elementary isomorphism respects the maps
\begin{eqnarray*}
&&g_{-}:[Y_1]\circ\cdots\circ[Y_N]\times Z \rightarrow G=GL(n,\field):
([y_1]\circ\cdots\circ[y_N],z)\mapsto y_1\cdots y_N,\\
&&\beta_{-}:[Y_1]\circ\cdots\circ[Y_N]\times Z \rightarrow \FBr_n^+:
([y_1]\circ\cdots\circ[y_N],z)\mapsto \beta_{[y_1]}\circ\cdots\circ\beta_{[y_N]}.
\end{eqnarray*}

\item
We have the following basic \textbf{elementary isomorphisms}:
\begin{enumerate}[wide,labelwidth=!,labelindent=0pt]
\item
Let $Y,Y_1,Y_2\subset B$. If multiplication induces an isomorphism
$m:Y_1\times Y_2\xrightarrow{\ \isomorphic\ }Y: (y_1,y_2)\mapsto y_1y_2$,
then
$
[Y_1]\circ[Y_2]\isomorphic [Y].
$

\item
If $Y_1\subset Y_2\subset B$ and $Y_2$ is a closed subgroup, then
$
[Y_1]\circ[Y_2]\isomorphic [Y_2]\times Y_1 \isomorphic [Y_2]\circ[Y_1].
$

\item
If $w\in S_n$ and $Y\subset U_w^+$, then
$
[w]\circ[Y]\isomorphic [wYw^{-1}]\circ[w].
$

\item
For any $C\in T$ and any $Y\subset B$, we have
$
[C]\circ[Y]\isomorphic [CYC^{-1}]\circ[C].
$
\end{enumerate}
\end{enumerate}
\end{definition/proposition}

\subsection{Diagram calculus for punctures}\label{subsec:diagram_calculus_for_punctures}

We return to the discussion at the end of Section \ref{subsec:diagram_calculus_of_matrices}. Recall \eqref{eqn:decomposition_for_Bw_iP_i} and \eqref{eqn:braid_matrix_diagram_for_a_puncture}:
\begin{eqnarray*}
&&U_{\dot{w}_i^{-1}}^-\times N_i\times Z(C_i)\isomorphic B\dot{w}_i P_i:
(\nu_i,n_i,z_i)\mapsto x_i=\nu_i\dot{w}_i n_i z_i,\\
&&N_i\xrightarrow{\ \isomorphic\ } N_i:\quad n_i\mapsto n_i'=n_iC_in_i^{-1}C_i^{-1},\\
&&[x_iC_ix_i^{-1}]'
=
[\nu_i]\circ[\dot{w}_i]\circ[n_i']\circ[C_i]\circ[\dot{w}_i^{-1}]\circ[\nu_i^{-1}]
\in\underline{\FBD}_n.
\end{eqnarray*}

Using the notation in \eqref{eqn:decomposition_for_U}, define
\[
{}^+n_i':=R_{\dot{w}_i}^+(n_i')\in U_{\dot{w}_i}^+\cap N_i,
\]
as in \eqref{eqn:decomposition_for_N_i}.

\begin{lemma}\label{lem:diagram_calculus_for_punctures}
For any $D_{i+1}\in T$, there is a natural isomorphism of $\field$-varieties
\begin{equation}
B\dot{w}_iP_i\times U
\xrightarrow[]{\simeq}
U\times U_{\dot{w}_i}^-\times U_{\dot{w}_i^{-1}}^-\times (U_{\dot{w}_i}^+\cap N_i)\times Z(C_i): (x_i,u_{i+1})\mapsto (u_i,\xi_i',\xi_i,{}^+n_i',z_i),
\end{equation}
such that
\begin{equation}\label{eqn:inductive_formula_for_a_puncture}
[x_iC_ix_i^{-1}]'\circ[u_{i+1}]\circ[D_{i+1}]
=
[u_i]\circ[D_i]\circ[\dot{w}_i]\circ[\xi_i']\circ[\dot{w}_i^{-1}]\circ[\xi_i]
\in\underline{\FBD}_n,
\end{equation}
where
$
D_i=C_i^{\dot{w}_i}D_{i+1}\in T.
$
Moreover, the above equation uniquely determines $(u_i,\xi_i',\xi_i)\in U\times U_{\dot{w}_i}^-\times U_{\dot{w}_i^{-1}}^-$.
\end{lemma}

\noindent\textbf{Note.}
Up to a canonical isomorphism, we have
\[
[\dot{w}_i]\circ[\xi_i']=[\mB_{[\dot{w}_i]}(\xi_i')]\in\underline{\FBD}_n,
\quad
[\dot{w}_i^{-1}]\circ[\xi_i]=[\mB_{[\dot{w}_i^{-1}]}(\xi_i)]\in\underline{\FBD}_n.
\]
This observation will be used to relate the above construction to braid varieties.

\begin{proof}[Proof of Lemma \ref{lem:diagram_calculus_for_punctures}]
We first rewrite the variables. There is an isomorphism of $\field$-varieties
\[
B\dot{w}_iP_i\times U
\xrightarrow[]{\simeq}
U_{\dot{w}_i^{-1}}^-\times N_i\times Z(C_i)\times U:
(x_i,u_{i+1})\mapsto (\nu_i,n_i,z_i,u_{i+1}':=\nu_i^{-1}u_{i+1}),
\]
under which
\[
[x_iC_ix_i^{-1}]'\circ[u_{i+1}]
=
[\nu_i]\circ[\dot{w}_i]\circ[n_i']\circ[C_i]\circ[\dot{w}_i^{-1}]\circ[u_{i+1}']
\in\underline{\FBD}_n.
\]

By Definition/Proposition \ref{def/prop:trick}, it suffices to compute
\[
[U_{\dot{w}_i^{-1}}^-]\circ[\dot{w}_i]\circ[N_i]\circ[C_i]\circ[\dot{w}_i^{-1}]\circ[U]\circ[D_{i+1}] \ni [\nu_i]\circ[\dot{w}_i]\circ[n_i']\circ[C_i]\circ[\dot{w}_i^{-1}]\circ[u_{i+1}']\circ[D_{i+1}].
\]

We proceed by a sequence of elementary isomorphisms:
\begin{eqnarray*}
\hspace{-1em}&&[U_{\dot{w}_i^{-1}}^-]\circ[\dot{w}_i]\circ[N_i]\circ[C_i]\circ[\dot{w}_i^{-1}]\circ[U]\circ[D_{i+1}]\\
\hspace{-1em}&\isomorphic&[U_{\dot{w}_i^{-1}}^-]\circ[\dot{w}_i]\circ[U_{\dot{w}_i}^-]\circ[U_{\dot{w}_i}^+\cap N_i]\circ[C_i]\circ[\dot{w}_i^{-1}]\circ[U]\circ[D_{i+1}]~~(n_i'={}^-n_i'\cdot {}^+n_i')\\
\hspace{-1em}&\isomorphic&[U_{\dot{w}_i^{-1}}^-]\circ[\dot{w}_i]\circ[U_{\dot{w}_i}^-]\circ[C_i]\circ[\dot{w}_i^{-1}]\circ[\dot{w}_iC_i^{-1}(U_{\dot{w}_i}^+\cap N_i)C_i\dot{w}_i^{-1}]\circ[U]\circ[D_{i+1}]\\
\hspace{-1em}&\isomorphic&[U_{\dot{w}_i^{-1}}^-]\circ[\dot{w}_i]\circ[U_{\dot{w}_i}^-]\circ[C_i]\circ[\dot{w}_i^{-1}]\circ[U]\circ[D_{i+1}]\times (U_{\dot{w}_i}^+\cap N_i)~~(\text{direct factor}: {}^+n_i'\in U_{\dot{w}_i}^+\cap N_i)\\
\hspace{-1em}&\isomorphic&[U_{\dot{w}_i^{-1}}^-]\circ[\dot{w}_i]\circ[U_{\dot{w}_i}^-]\circ[C_i]\circ[\dot{w}_i^{-1}]\circ[U_{\dot{w}_i^{-1}}^+]\circ[U_{\dot{w}_i^{-1}}^-]\circ[D_{i+1}]\times (U_{\dot{w}_i}^+\cap N_i)~~(U=U_{\dot{w}_i^{-1}}^+U_{\dot{w}_i^{-1}}^-)\\
\hspace{-1em}&\isomorphic&[U_{\dot{w}_i^{-1}}^-]\circ[\dot{w}_i]\circ[U_{\dot{w}_i}^-]\circ[U_{\dot{w}_i}^+]\circ[C_i]\circ[\dot{w}_i^{-1}]\circ[U_{\dot{w}_i^{-1}}^-]\circ[D_{i+1}]\times (U_{\dot{w}_i}^+\cap N_i)\\
\hspace{-1em}&\isomorphic&[U_{\dot{w}_i^{-1}}^-]\circ[\dot{w}_i]\circ[U_{\dot{w}_i}^+]\circ[U_{\dot{w}_i}^-]\circ[C_i]\circ[\dot{w}_i^{-1}]\circ[U_{\dot{w}_i^{-1}}^-]\circ[D_{i+1}]\times (U_{\dot{w}_i}^+\cap N_i)~~(U_{\dot{w}_i}^-U_{\dot{w}_i}^+=U_{\dot{w}_i}^+U_{\dot{w}_i}^-)\\
\hspace{-1em}&\isomorphic&[U]\circ[\dot{w}_i]\circ[U_{\dot{w}_i}^-]\circ[C_i]\circ[\dot{w}_i^{-1}]\circ[U_{\dot{w}_i^{-1}}^-]\circ[D_{i+1}]\times (U_{\dot{w}_i}^+\cap N_i)\\
\hspace{-1em}&\isomorphic&[U]\circ[D_i]\circ[\dot{w}_i]\circ[U_{\dot{w}_i}^-]\circ[\dot{w}_i^{-1}]\circ[U_{\dot{w}_i^{-1}}^-]\times (U_{\dot{w}_i}^+\cap N_i).
\end{eqnarray*}
This yields the desired formula \eqref{eqn:inductive_formula_for_a_puncture}. The uniqueness follows from the definition of $\underline{\FBD}_n$.
\end{proof}

The key point is that Lemma \ref{lem:diagram_calculus_for_punctures} provides the inductive step for the diagram calculus at punctures. In our situation, we have $u_k\in U$, and we take $D_k:=C_k$. Defining $D_i$ inductively as above, we obtain
\begin{equation}\label{eqn:D_1}
D_1=\Big(\prod_{i=1}^{k-1}C_i^{\dot{w}_i}\Big)C_k\in T.
\end{equation}

Consequently, the defining equation \eqref{eqn:Bruhat_cell_of_character_variety_via_braid_matrix_diagrams} for $M_B'(\vec{w})$ reduces to
\begin{equation}\label{eqn:diagram_calculus_for_punctures}
[\mM_{\vec{w}}]
=
\prod_{j=1}^g
([A_{2j-1}]\circ[A_{2j}]\circ[A_{2j-1}^{-1}]\circ[A_{2j}^{-1}])
\circ
[u_1]\circ[D_1]
\circ
\prod_{i=1}^{k-1}
([\dot{w}_i\xi_i']\circ[\dot{w}_i^{-1}\xi_i])
\weakequivalent \id_n.
\end{equation}

For the genus calculation below, we set
\[
u^g:=u_1\in U,
\qquad
D^g:=D_1\in T.
\]

\subsection{Diagram calculus for genera}\label{subsec:diagram_calculus_for_genera}

For each $1\leq j\leq g$, consider the isomorphisms
\begin{eqnarray}\label{eqn:decomposition_for_Btau_jB}
\hspace{-2em}&&U_{\tau_{2j-1}^{-1}}^-\times T\times U
\xrightarrow[]{\ \isomorphic\ } B\tau_{2j-1}B:
(\mu_{2j-1},y_{2j-1},\eta_{2j-1})\mapsto A_{2j-1}=\mu_{2j-1}\tau_{2j-1}y_{2j-1}\eta_{2j-1},\\
\hspace{-2em}&&U\times T\times U_{\tau_{2j}}^-
\xrightarrow[]{\ \isomorphic\ } B\tau_{2j}B:
(\eta_{2j},y_{2j},\mu_{2j})\mapsto A_{2j}=\eta_{2j}y_{2j}\tau_{2j}\mu_{2j}.\nonumber
\end{eqnarray}
Recall that $A_m\in B\tau_mB$ for all $1\leq m\leq 2g$. By Proposition \ref{prop:from_matrices_to_braid_matrix_diagrams}, we have
\begin{eqnarray*}
\hspace{-2em}&&[A_{2j-1}]
=
[\mu_{2j-1}]\circ[\tau_{2j-1}]\circ[y_{2j-1}]\circ[\eta_{2j-1}],
\qquad
[A_{2j}]
=
[\eta_{2j}]\circ[y_{2j}]\circ[\tau_{2j}]\circ[\mu_{2j}]
\in\underline{\FBD}_n,\\
\hspace{-2em}&&[A_{2j-1}^{-1}]
=
[\eta_{2j-1}^{-1}]\circ[y_{2j-1}^{-1}]\circ[\tau_{2j-1}^{-1}]\circ[\mu_{2j-1}^{-1}],
\qquad
[A_{2j}^{-1}]
=
[\mu_{2j}^{-1}]\circ[\tau_{2j}^{-1}]\circ[y_{2j}^{-1}]\circ[\eta_{2j}^{-1}]
\in\underline{\FBD}_n.
\end{eqnarray*}

\begin{lemma}\label{lem:diagram_calculus_for_genera}
For any $D^j\in T$, there is a natural isomorphism of $\field$-varieties
\begin{eqnarray}\label{eqn:reparametrization_for_genera}
\hspace{-2em}&&B\tau_{2j-1}B\times B\tau_{2j}B\times U
\xrightarrow[]{\simeq}
U\times T^2\times (U_{\tau_{2j-1}}^-\times U_{\tau_{2j}}^-\times U_{\tau_{2j-1}^{-1}}^-\times U_{\tau_{2j}^{-1}}^-)\times (U_{\tau_{2j}^{-1}}^+\times U_{\tau_{2j-1}}^+),\\
\hspace{-2em}&&(A_{2j-1},A_{2j},u^j)\mapsto
(u^{j-1},y_{2j-1},y_{2j},\zeta_{2j-1}',\zeta_{2j}',\zeta_{2j-1},\zeta_{2j},{}^+n^{2j},{}^+n^{2j-1}),\nonumber
\end{eqnarray}
such that
\begin{eqnarray}\label{eqn:diagram_calculus_for_genera}
\hspace{-2em}&&[A_{2j-1}]\circ[A_{2j}]\circ[A_{2j-1}^{-1}]\circ[A_{2j}^{-1}]\circ[u^j]\circ[D^j]\\
\hspace{-2em}&=&
[u^{j-1}]\circ[D^{j-1}]
\circ[\tau_{2j-1}]\circ[\zeta_{2j-1}']
\circ[\tau_{2j}]\circ[\zeta_{2j}']
\circ[\tau_{2j-1}^{-1}]\circ[\zeta_{2j-1}]
\circ[\tau_{2j}^{-1}]\circ[\zeta_{2j}]
\in\underline{\FBD}_n,\nonumber
\end{eqnarray}
where $D^{j-1}\in T$ is uniquely determined by
\begin{equation}\label{eqn:inductive_formula_for_D^j}
D^{j-1}
:=
y_{2j-1}^{\tau_{2j-1}}
y_{2j}^{\tau_{2j-1}}
(y_{2j-1}^{-1})^{\tau_{2j-1}\tau_{2j}}
(y_{2j}^{-1})^{\tau_{2j-1}\tau_{2j}\tau_{2j-1}^{-1}\tau_{2j}^{-1}}
(D^j)^{\tau_{2j-1}\tau_{2j}\tau_{2j-1}^{-1}\tau_{2j}^{-1}}
\in T.
\end{equation}
Moreover, \eqref{eqn:diagram_calculus_for_genera} uniquely determines
$
(u^{j-1},\zeta_{2j-1}',\zeta_{2j}',\zeta_{2j-1},\zeta_{2j})
\in
U\times U_{\tau_{2j-1}}^-\times U_{\tau_{2j}}^-\times U_{\tau_{2j-1}^{-1}}^-\times U_{\tau_{2j}^{-1}}^-.
$
\end{lemma}

This provides the inductive step for the diagram calculus for genera. In our situation, we have $u^g\in U$ and $D^g=D_1\in T$, and we define the elements $D^j$ inductively by \eqref{eqn:inductive_formula_for_D^j}. Thus,
\begin{equation}\label{eqn:D^0}
D^0((y_j)_{j=1}^{2g})
=
\prod_{j=1}^g
\Bigl(
y_{2j-1}^{\tau_{2j-1}}
(y_{2j-1}^{-1})^{\tau_{2j-1}\tau_{2j}}
y_{2j}^{\tau_{2j-1}}
(y_{2j}^{-1})^{\lbracket \tau_{2j-1}, \tau_{2j} \rbracket}
\Bigr)^{\prod_{m=1}^{j-1}\lbracket \tau_{2m-1}, \tau_{2m} \rbracket}
D_1^{\prod_{m=1}^g\lbracket \tau_{2m-1}, \tau_{2m} \rbracket}.
\end{equation}
Altogether, the defining equation \eqref{eqn:Bruhat_cell_of_character_variety_via_braid_matrix_diagrams} for $M_B'(\vec{w})$ reduces to
\begin{equation*}
[u^0D^0]
\circ
\prod_{j=1}^g
\bigl(
[\tau_{2j-1}\zeta_{2j-1}']
\circ[\tau_{2j}\zeta_{2j}']
\circ[\tau_{2j-1}^{-1}\zeta_{2j-1}]
\circ[\tau_{2j}^{-1}\zeta_{2j}]
\bigr)
\circ
\prod_{i=1}^{k-1}
\bigl(
[\dot{w}_i\xi_i']\circ[\dot{w}_i^{-1}\xi_i]
\bigr)
\weakequivalent \id_n.
\end{equation*}

\begin{proof}[Proof of Lemma \ref{lem:diagram_calculus_for_genera}]
The proof is carried out by diagram calculus.

\smallskip
\noindent\textbf{Step 1.}
Set
\[
u'^j:=\eta_{2j}^{-1}u^j\in U,
\qquad
\eta_{2j}':=\eta_{2j-1}\eta_{2j}\in U,
\qquad
\eta_{2j-1}':=\eta_{2j-1}\mu_{2j}^{-1}\in U.
\]
Then $\eta_{2j-1}'^{-1}=\mu_{2j}\eta_{2j-1}^{-1}$. Hence
\[
(A_{2j-1},A_{2j},u^j)\mapsto (\mu_{2j-1},y_{2j-1},\eta_{2j-1}',\eta_{2j}',y_{2j},\mu_{2j},u'^j)
\]
defines an isomorphism
\[
B\tau_{2j-1}B\times B\tau_{2j}B\times U
\xrightarrow[]{\isomorphic}
U_{\tau_{2j-1}^{-1}}^-\times T\times U^2\times T\times U_{\tau_{2j}}^-\times U
\]
such that we obtain an identity in $\underline{\FBD}_n$:
\begin{eqnarray}
&&[A_{2j-1}]\circ[A_{2j}]\circ[A_{2j-1}^{-1}]\circ[A_{2j}^{-1}]\circ[u^j]\label{eqn:genera_step1}\\
&=&[\mu_{2j-1}]\circ[\tau_{2j-1}]\circ[y_{2j-1}]\circ[\eta_{2j}']\circ[y_{2j}]\circ[\tau_{2j}]\nonumber\\
&&\circ[\eta_{2j-1}'^{-1}]\circ[y_{2j-1}^{-1}]\circ[\tau_{2j-1}^{-1}]\circ[\mu_{2j-1}^{-1}]\circ[\mu_{2j}^{-1}]\circ[\tau_{2j}^{-1}]\circ[y_{2j}^{-1}]\circ[u'^j].\nonumber
\end{eqnarray}

\noindent\textbf{Note.} The variable $\mu_{2j}^{-1}\in U_{\tau_{2j}}^-$ is now free, in the sense that it appears only once.

\smallskip
\noindent\textbf{Step 2.}
We first compute
$
[\mu_{2j-1}^{-1}]\circ[\mu_{2j}^{-1}]\circ[\tau_{2j}^{-1}]\circ[y_{2j}^{-1}]\circ[u'^j].
$

By Definition/Proposition \ref{def/prop:trick},
\begin{eqnarray*}
&&[\mu_{2j-1}^{-1}]\circ[U_{\tau_{2j}}^-]\circ[\tau_{2j}^{-1}]\circ[y_{2j}^{-1}]\circ[U]\\
&\isomorphic&
[\mu_{2j-1}^{-1}]\circ[U_{\tau_{2j}}^-]\circ[\tau_{2j}^{-1}]\circ[U]\circ[y_{2j}^{-1}]\\
&\isomorphic&
[\mu_{2j-1}^{-1}]\circ[U_{\tau_{2j}}^-]\circ[\tau_{2j}^{-1}]
\circ[U_{\tau_{2j}^{-1}}^+]\circ[U_{\tau_{2j}^{-1}}^-]\circ[y_{2j}^{-1}]
\qquad (U=U_{\tau_{2j}^{-1}}^+U_{\tau_{2j}^{-1}}^-)\\
&\isomorphic&
[\mu_{2j-1}^{-1}]\circ[U_{\tau_{2j}}^-]\circ[U_{\tau_{2j}}^+]\circ[\tau_{2j}^{-1}]
\circ[U_{\tau_{2j}^{-1}}^-]\circ[y_{2j}^{-1}]\\
&\isomorphic&
[\mu_{2j-1}^{-1}]\circ[U]\circ[\tau_{2j}^{-1}]\circ[U_{\tau_{2j}^{-1}}^-]\circ[y_{2j}^{-1}]\\
&\isomorphic&
[U]\circ[\tau_{2j}^{-1}]\circ[U_{\tau_{2j}^{-1}}^-]\circ[y_{2j}^{-1}].
\end{eqnarray*}
Therefore we obtain an isomorphism of $\field$-varieties
\begin{eqnarray}
&&U_{\tau_{2j-1}^{-1}}^-\times T\times U^2\times T\times U_{\tau_{2j}}^-\times U
\xrightarrow[]{\isomorphic}
U_{\tau_{2j-1}^{-1}}^-\times T\times U^3\times T\times U_{\tau_{2j}^{-1}}^-,
\\
&&(\mu_{2j-1},y_{2j-1},\eta_{2j-1}',\eta_{2j}',y_{2j},\mu_{2j},u'^j)
\mapsto
(\mu_{2j-1},y_{2j-1},\eta_{2j-1}',\eta_{2j}',u_3^j,y_{2j},L_1^-),\nonumber
\end{eqnarray}
such that
\[
[\mu_{2j-1}^{-1}]\circ[\mu_{2j}^{-1}]\circ[\tau_{2j}^{-1}]\circ[y_{2j}^{-1}]\circ[u'^j]
=
[u_3^j]\circ[\tau_{2j}^{-1}]\circ[L_1^-]\circ[y_{2j}^{-1}]
\in\underline{\FBD}_n.
\]
Substituting into \eqref{eqn:genera_step1}, we get
\begin{eqnarray}
&&[A_{2j-1}]\circ[A_{2j}]\circ[A_{2j-1}^{-1}]\circ[A_{2j}^{-1}]\circ[u^j]\label{eqn:genera_step2}\\
&=&
[\mu_{2j-1}]\circ[\tau_{2j-1}]\circ[y_{2j-1}]\circ[\eta_{2j}']\circ[y_{2j}]\circ[\tau_{2j}]\nonumber\\
&&\circ[\eta_{2j-1}'^{-1}]\circ[y_{2j-1}^{-1}]
\circ[\tau_{2j-1}^{-1}]\circ[u_3^j]
\circ[\tau_{2j}^{-1}]\circ[L_1^-]\circ[y_{2j}^{-1}].\nonumber
\end{eqnarray}

\noindent\textbf{Note.} The non-torus variables
\[
\mu_{2j-1}\in U_{\tau_{2j-1}^{-1}}^-,
\qquad
\eta_{2j}',\eta_{2j-1}'^{-1},u_3^j\in U,
\qquad
L_1^-\in U_{\tau_{2j}^{-1}}^-
\]
are now free.

\smallskip
\noindent\textbf{Step 3.}
By Definition/Proposition \ref{def/prop:trick}, the computation of
$
[A_{2j-1}]\circ[A_{2j}]\circ[A_{2j-1}^{-1}]\circ[A_{2j}^{-1}]\circ[u^j]
$
reduces to the following variety in $\underline{\FBD}_n$:
\begin{equation}
[U_{\tau_{2j-1}^{-1}}^-]\circ[\tau_{2j-1}]\circ[y_{2j-1}]\circ[U]\circ[y_{2j}]\circ[\tau_{2j}]
\circ[U]\circ[y_{2j-1}^{-1}]\circ[\tau_{2j-1}^{-1}]
\circ[U]\circ[\tau_{2j}^{-1}]\circ[U_{\tau_{2j}^{-1}}^-]\circ[y_{2j}^{-1}].
\end{equation}
The idea is to canonicalize this expression. We first reorder the variables:
\begin{eqnarray}
&&U_{\tau_{2j-1}^{-1}}^-\times T\times U^3\times T\times U_{\tau_{2j}^{-1}}^-
\xrightarrow[]{\isomorphic}
U_{\tau_{2j-1}^{-1}}^-\times (T\times U)^2\times U\times U_{\tau_{2j}^{-1}}^-,
\\
&&(\mu_{2j-1},y_{2j-1},\eta_{2j-1}',\eta_{2j}',u_3^j,y_{2j},L_1^-)
\mapsto
(\mu_{2j-1},y_{2j-1},\eta_{2j}',y_{2j},\eta_{2j-1}'^{-1},u_3^j,L_1^-).\nonumber
\end{eqnarray}

\smallskip
\noindent\textbf{Step 3.1.}
We first compute
$
[U]\circ[y_{2j-1}^{-1}]\circ[\tau_{2j-1}^{-1}]\circ[U]
\ni
[\eta_{2j-1}'^{-1}]\circ[y_{2j-1}^{-1}]\circ[\tau_{2j-1}^{-1}]\circ[u_3^j].
$

Using again Definition/Proposition \ref{def/prop:trick},
\begin{eqnarray*}
&&[U]\circ[y_{2j-1}^{-1}]\circ[\tau_{2j-1}^{-1}]\circ[U]\\
&\isomorphic&
[U]\circ[y_{2j-1}^{-1}]\circ[\tau_{2j-1}^{-1}]
\circ[U_{\tau_{2j-1}^{-1}}^+]\circ[U_{\tau_{2j-1}^{-1}}^-]
\qquad (U=U_{\tau_{2j-1}^{-1}}^+U_{\tau_{2j-1}^{-1}}^-)\\
&\isomorphic&
[U]\circ[U_{\tau_{2j-1}}^+]\circ[y_{2j-1}^{-1}]\circ[\tau_{2j-1}^{-1}]\circ[U_{\tau_{2j-1}^{-1}}^-]\\
&\isomorphic&
[U]\circ[y_{2j-1}^{-1}]\circ[\tau_{2j-1}^{-1}]\circ[U_{\tau_{2j-1}^{-1}}^-]\times U_{\tau_{2j-1}}^+,
\end{eqnarray*}
with direct factor $U_{\tau_{2j-1}}^+\ni {}^+n^{2j-1}$. Hence we obtain an isomorphism of $\field$-varieties
\begin{eqnarray}
&&U_{\tau_{2j-1}^{-1}}^-\times (T\times U)^2\times U\times U_{\tau_{2j}^{-1}}^-
\xrightarrow[]{\simeq}
U_{\tau_{2j-1}^{-1}}^-\times (T\times U)^2\times U_{\tau_{2j-1}^{-1}}^-\times U_{\tau_{2j}^{-1}}^-\times U_{\tau_{2j-1}}^+,
\\
&&(\mu_{2j-1},y_{2j-1},\eta_{2j}',y_{2j},\eta_{2j-1}'^{-1},u_3^j,L_1^-)
\mapsto
(\mu_{2j-1},y_{2j-1},\eta_{2j}',y_{2j},u_4^j,L_3^-,L_1^-,{}^+n^{2j-1}),\nonumber
\end{eqnarray}
such that
\[
[\eta_{2j-1}'^{-1}]\circ[y_{2j-1}^{-1}]\circ[\tau_{2j-1}^{-1}]\circ[u_3^j]
=
[u_4^j]\circ[y_{2j-1}^{-1}]\circ[\tau_{2j-1}^{-1}]\circ[L_3^-]
\in\underline{\FBD}_n.
\]

\smallskip
\noindent\textbf{Step 3.2.}
Similarly, we compute
$
[U]\circ[y_{2j}]\circ[\tau_{2j}]\circ[U]
\ni
[\eta_{2j}']\circ[y_{2j}]\circ[\tau_{2j}]\circ[u_4^j].
$

Indeed,
\begin{eqnarray*}
&&[U]\circ[y_{2j}]\circ[\tau_{2j}]\circ[U]\\
&\isomorphic&
[U]\circ[y_{2j}]\circ[\tau_{2j}]\circ[U_{\tau_{2j}}^+]\circ[U_{\tau_{2j}}^-]
\qquad (U=U_{\tau_{2j}}^+U_{\tau_{2j}}^-)\\
&\isomorphic&
[U]\circ[U_{\tau_{2j}^{-1}}^+]\circ[y_{2j}]\circ[\tau_{2j}]\circ[U_{\tau_{2j}}^-]\\
&\isomorphic&
[U]\circ[y_{2j}]\circ[\tau_{2j}]\circ[U_{\tau_{2j}}^-]\times U_{\tau_{2j}^{-1}}^+,
\end{eqnarray*}
with direct factor $U_{\tau_{2j}^{-1}}^+\ni {}^+n^{2j}$. Thus we obtain an isomorphism of $\field$-varieties
\begin{eqnarray*}
\hspace{-2em}&&U_{\tau_{2j-1}^{-1}}^-\times (T\times U)^2\times U_{\tau_{2j-1}^{-1}}^-\times U_{\tau_{2j}^{-1}}^-\times U_{\tau_{2j-1}}^+
\xrightarrow[]{\simeq}
U_{\tau_{2j-1}^{-1}}^-\times T\times U\times T\times U_{\tau_{2j}}^-\times U_{\tau_{2j-1}^{-1}}^-\times U_{\tau_{2j}^{-1}}^-\times U_{\tau_{2j}^{-1}}^+\times U_{\tau_{2j-1}}^+,\\
\hspace{-2em}&&(\mu_{2j-1},y_{2j-1},\eta_{2j}',y_{2j},u_4^j,L_3^-,L_1^-,{}^+n^{2j-1})
\mapsto
(\mu_{2j-1},y_{2j-1},u_5^j,y_{2j},L_5^-,L_3^-,L_1^-,{}^+n^{2j},{}^+n^{2j-1}),
\end{eqnarray*}
such that
\[
[\eta_{2j}']\circ[y_{2j}]\circ[\tau_{2j}]\circ[u_4^j]
=
[u_5^j]\circ[y_{2j}]\circ[\tau_{2j}]\circ[L_5^-]
\in\underline{\FBD}_n.
\]

\smallskip
\noindent\textbf{Step 3.3.}
Finally, we compute
$
[U_{\tau_{2j-1}^{-1}}^-]\circ[\tau_{2j-1}]\circ[y_{2j-1}]\circ[U]
\ni
[\mu_{2j-1}]\circ[\tau_{2j-1}]\circ[y_{2j-1}]\circ[u_5^j].
$

We have
\begin{eqnarray*}
&&[U_{\tau_{2j-1}^{-1}}^-]\circ[\tau_{2j-1}]\circ[y_{2j-1}]\circ[U]\\
&\isomorphic&
[U_{\tau_{2j-1}^{-1}}^-]\circ[\tau_{2j-1}]\circ[y_{2j-1}]\circ[U_{\tau_{2j-1}}^+]\circ[U_{\tau_{2j-1}}^-]\\
&\isomorphic&
[U_{\tau_{2j-1}^{-1}}^-]\circ[U_{\tau_{2j-1}^{-1}}^+]\circ[\tau_{2j-1}]\circ[y_{2j-1}]\circ[U_{\tau_{2j-1}}^-]\\
&\isomorphic&
[U]\circ[\tau_{2j-1}]\circ[y_{2j-1}]\circ[U_{\tau_{2j-1}}^-].
\end{eqnarray*}
Hence we obtain an isomorphism of $\field$-varieties
\begin{eqnarray*}
\hspace{-1em}&&U_{\tau_{2j-1}^{-1}}^-\times T\times U\times T\times U_{\tau_{2j}}^-\times U_{\tau_{2j-1}^{-1}}^-\times U_{\tau_{2j}^{-1}}^-\times U_{\tau_{2j}^{-1}}^+\times U_{\tau_{2j-1}}^+\\
\hspace{-2em}&\xrightarrow[]{\simeq}&
U\times T\times U_{\tau_{2j-1}}^-\times T\times U_{\tau_{2j}}^-\times U_{\tau_{2j-1}^{-1}}^-\times U_{\tau_{2j}^{-1}}^-\times U_{\tau_{2j}^{-1}}^+\times U_{\tau_{2j-1}}^+,\\
\hspace{-1em}&&(\mu_{2j-1},y_{2j-1},u_5^j,y_{2j},L_5^-,L_3^-,L_1^-,{}^+n^{2j},{}^+n^{2j-1})
\mapsto
(u^{j-1},y_{2j-1},L_7^-,y_{2j},L_5^-,L_3^-,L_1^-,{}^+n^{2j},{}^+n^{2j-1}),
\end{eqnarray*}
such that
\[
[\mu_{2j-1}]\circ[\tau_{2j-1}]\circ[y_{2j-1}]\circ[u_5^j]
=
[u^{j-1}]\circ[\tau_{2j-1}]\circ[y_{2j-1}]\circ[L_7^-]
\in\underline{\FBD}_n.
\]

Combining the previous steps, we obtain
\begin{eqnarray}
\hspace{-2em}&&[A_{2j-1}]\circ[A_{2j}]\circ[A_{2j-1}^{-1}]\circ[A_{2j}^{-1}]\circ[u^j]\label{eqn:genera_summary}\\
\hspace{-2em}&=&
[u^{j-1}]
\circ[\tau_{2j-1}]
\circ[y_{2j-1}L_7^-y_{2j}]
\circ[\tau_{2j}]
\circ[L_5^-y_{2j-1}^{-1}]
\circ[\tau_{2j-1}^{-1}]
\circ[L_3^-]
\circ[\tau_{2j}^{-1}]
\circ[L_1^-y_{2j}^{-1}],\nonumber
\end{eqnarray}
and
\begin{eqnarray*}
&&[U_{\tau_{2j-1}^{-1}}^-]\circ[\tau_{2j-1}]\circ[y_{2j-1}]\circ[U]\circ[y_{2j}]\circ[\tau_{2j}]\circ[U]\circ[y_{2j-1}^{-1}]\\
&&\circ[\tau_{2j-1}^{-1}]\circ[U]\circ[\tau_{2j}^{-1}]\circ[U_{\tau_{2j}^{-1}}^-]\circ[y_{2j}^{-1}]\\
&\isomorphic&
[U]\circ[\tau_{2j-1}]\circ[y_{2j-1}U_{\tau_{2j-1}}^-y_{2j}]\circ[\tau_{2j}]\circ[U_{\tau_{2j}}^-y_{2j-1}^{-1}]\\
&&\circ[\tau_{2j-1}^{-1}]\circ[U_{\tau_{2j-1}^{-1}}^-]\circ[\tau_{2j}^{-1}]\circ[U_{\tau_{2j}^{-1}}^-y_{2j}^{-1}]\times U_{\tau_{2j}^{-1}}^+\times U_{\tau_{2j-1}}^+.
\end{eqnarray*}

\smallskip
\noindent\textbf{Step 4.}
We now incorporate the torus variables. For any $D^j\in T$, it follows from \eqref{eqn:genera_summary} that
\begin{eqnarray*}
&&[A_{2j-1}]\circ[A_{2j}]\circ[A_{2j-1}^{-1}]\circ[A_{2j}^{-1}]\circ[u^j]\circ[D^j]\\
&=&
[u^{j-1}]
\circ[\tau_{2j-1}]
\circ[y_{2j-1}L_7^-y_{2j}]
\circ[\tau_{2j}]
\circ[L_5^-y_{2j-1}^{-1}]
\circ[\tau_{2j-1}^{-1}]
\circ[L_3^-]
\circ[\tau_{2j}^{-1}]
\circ[L_1^-y_{2j}^{-1}]
\circ[D^j]\\
&=&
[u^{j-1}]\circ[D^{j-1}]
\circ[\tau_{2j-1}]\circ[\zeta_{2j-1}']
\circ[\tau_{2j}]\circ[\zeta_{2j}']
\circ[\tau_{2j-1}^{-1}]\circ[\zeta_{2j-1}]
\circ[\tau_{2j}^{-1}]\circ[\zeta_{2j}]
\in\underline{\FBD}_n,
\end{eqnarray*}
where
\[
(\zeta_{2j-1}',\zeta_{2j}',\zeta_{2j-1},\zeta_{2j})
\in
U_{\tau_{2j-1}}^-\times U_{\tau_{2j}}^-\times U_{\tau_{2j-1}^{-1}}^-\times U_{\tau_{2j}^{-1}}^-,
\]
and $D^{j-1}$ is given by \eqref{eqn:inductive_formula_for_D^j}.

This gives an isomorphism of the form \eqref{eqn:reparametrization_for_genera} such that \eqref{eqn:diagram_calculus_for_genera} and \eqref{eqn:inductive_formula_for_D^j} hold. The uniqueness statement follows from Proposition \ref{prop:equality_in_underline_FBD_n}.
\end{proof}

\begin{remark}\label{rem:action_on_+n^j}
A careful inspection of the proof gives explicit formulas for ${}^+n^{2j-1}$ and ${}^+n^{2j}$ in terms of \eqref{eqn:decomposition_for_U}:
\begin{eqnarray}
&&{}^+n^{2j-1}
=
y_{2j-1}^{-1}\tau_{2j-1}^{-1}
L_{\tau_{2j-1}^{-1}}^+
\bigl(
\mu_{2j-1}^{-1}\mu_{2j}^{-1}\tau_{2j}^{-1}
L_{\tau_{2j}^{-1}}^+(y_{2j}^{-1}\eta_{2j}^{-1}u^jy_{2j})
\tau_{2j}
\bigr)
\tau_{2j-1}y_{2j-1},\\
&&{}^+n^{2j}
=
y_{2j}\tau_{2j}
L_{\tau_{2j}}^+
\bigl(
\mu_{2j}\eta_{2j-1}^{-1}{}^+n^{2j-1}
\bigr)
\tau_{2j}^{-1}y_{2j}^{-1}.\nonumber
\end{eqnarray}
\end{remark}

\subsection{Connection to braid varieties}\label{subsec:connection_to_braid_varieties}

We now relate $M_B'(\vec{w})$ to braid varieties.

Summarizing Sections \ref{subsec:diagram_calculus_for_punctures}--\ref{subsec:diagram_calculus_for_genera}, we obtained an isomorphism
\begin{eqnarray}\label{eqn:reparameterization_for_connection_to_braid_varieties}
&&
\Bigl(\prod_{j=1}^{2g} B\tau_jB\Bigr)
\times
\Bigl(\prod_{i=1}^{k-1} B\dot{w}_iP_i\Bigr)
\times U
\xrightarrow[]{\isomorphic}
U \times
\prod_{j=1}^g
\Bigl(
T^2
\times
(U_{\tau_{2j-1}}^- \times U_{\tau_{2j}}^- \times U_{\tau_{2j-1}^{-1}}^- \times U_{\tau_{2j}^{-1}}^-)
\times
(U_{\tau_{2j}^{-1}}^+ \times U_{\tau_{2j-1}}^+)
\Bigr)
\nonumber\\
&&
\hspace{15.4em} \times
\prod_{i=1}^{k-1}
\Bigl(
U_{\dot{w}_i}^- \times U_{\dot{w}_i^{-1}}^- \times (U_{\dot{w}_i}^+\cap N_i) \times Z(C_i)
\Bigr),\\
&&
((A_j)_j,(x_i)_i,u_k)
\mapsto
\Bigl(
u^0,\,
(y_{2j-1},y_{2j},\zeta_{2j-1}',\zeta_{2j}',\zeta_{2j-1},\zeta_{2j},{}^+n^{2j},{}^+n^{2j-1})_{j=1}^g,\,
(\xi_i',\xi_i,{}^+n_i',z_i)_{i=1}^{k-1}
\Bigr),
\nonumber
\end{eqnarray}
such that we have an equality in $\underline{\FBD}_n$:
\begin{eqnarray*}
&&
[\mM_{\vec{w}}]
=
\prod_{j=1}^g
([A_{2j-1}]\circ[A_{2j}]\circ[A_{2j-1}^{-1}]\circ[A_{2j}^{-1}])
\circ
\prod_{i=1}^{k-1}[x_iC_ix_i^{-1}]'
\circ[u_k]\circ[D_k]
\\
&=&
[u^0D^0]
\text{\tiny$\circ$}
\prod_{j=1}^g
\bigl(
[\tau_{2j-1}]\text{\tiny$\circ$}[\zeta_{2j-1}']\text{\tiny$\circ$}
[\tau_{2j}]\text{\tiny$\circ$}[\zeta_{2j}']\text{\tiny$\circ$}
[\tau_{2j-1}^{-1}]\text{\tiny$\circ$}[\zeta_{2j-1}]\text{\tiny$\circ$}
[\tau_{2j}^{-1}]\text{\tiny$\circ$}[\zeta_{2j}]
\bigr)
\text{\tiny$\circ$}
\prod_{i=1}^{k-1}
\bigl(
[\dot{w}_i]\text{\tiny$\circ$}[\xi_i']\text{\tiny$\circ$}
[\dot{w}_i^{-1}]\text{\tiny$\circ$}[\xi_i]
\bigr).
\end{eqnarray*}

Recall from \eqref{eqn:shape} that
\[
\beta(\vec{w})
=
\prod_{j=1}^g
\bigl(
[\tau_{2j-1}]\circ[\tau_{2j}]\circ[\tau_{2j-1}^{-1}]\circ[\tau_{2j}^{-1}]
\bigr)
\circ
\prod_{i=1}^{k-1}
\bigl(
[\dot{w}_i]\circ[\dot{w}_i^{-1}]
\bigr)
\in \FBr_n^+.
\]
By Definition \ref{def:braid_matrices}, the defining equation \eqref{eqn:Bruhat_cell_of_character_variety_via_braid_matrix_diagrams} for $M_B'(\vec{w})$ in $\underline{\FBD}_n$ becomes
\begin{equation}\label{eqn:Bruhat_cell_of_character_variety_via_braid_varieties}
[\mB_{\beta(\vec{w})}(\vec{\epsilon})]'
\weakequivalent
[(D^0)^{-1}]\circ[(u^0)^{-1}],
\quad
\vec{\epsilon}:=((\zeta_{2j-1}',\zeta_{2j}',\zeta_{2j-1},\zeta_{2j})_{j=1}^g,(\xi_i',\xi_i)_{i=1}^{k-1})
\in \field^{\ell(\beta(\vec{w}))}.
\end{equation}

By \eqref{eqn:D^0}, \eqref{eqn:D_1}, and the genericity assumption (Definition \ref{def:generic_monodromy}), we have
\[
\det D^0=\det D_1=1.
\]
Define
\begin{equation}\label{eqn:phi_w}
\phi_{\vec{w}}:T^{2g}\times X(\beta(\vec{w}))\rightarrow T_1,
\qquad
((y_j)_{j=1}^{2g},\vec{\epsilon})\mapsto D^0\mumon(\vec{\epsilon}),
\qquad
T_1:=\{t\in T:\det t=1\}.
\end{equation}
Here, $\det\circ\mumon=1$ by Proposition \ref{prop:cell_decomposition_of_braid_varieties}: for any $p\in\cW(\beta(\vec{w}))$,
\[
|S_p|
=
\ell(\beta(\vec{w}))-2|U_p|
=
\sum_{j=1}^g2\ell(\tau_j)+\sum_{i=1}^{k-1}2\ell(\dot{w}_i)-2|U_p|,
\]
hence $|S_p|$ is even. Define the closed $\field$-subvariety
\begin{equation}\label{eqn:restricted_twisted_braid_variety}
\tilde{X}(\beta(\vec{w}))
:=
\phi_{\vec{w}}^{-1}(I_n)
\subset T^{2g}\times X(\beta(\vec{w})).
\end{equation}
Then \eqref{eqn:Bruhat_cell_of_character_variety_via_braid_varieties} shows that $M_B'(\vec{w})$ is related to the restricted twisted braid variety $\tilde{X}(\beta(\vec{w}))$ by
\begin{equation}
M_B'(\vec{w})
\isomorphic
\tilde{X}(\beta(\vec{w}))
\times
\prod_{j=1}^g (U_{\tau_{2j}^{-1}}^+\times U_{\tau_{2j-1}}^+)
\times
\prod_{i=1}^{k-1} ((U_{\dot{w}_i}^+\cap N_i)\times Z(C_i)).
\end{equation}

Now $(b,(h_i)_{i=1}^{k-1})\in B_{\pa}$ acts on
\[
(u^0,(y_j)_{j=1}^{2g},\vec{\epsilon},({}^+n^j)_{j=1}^{2g},({}^+n_i')_{i=1}^{k-1},(z_i)_{i=1}^{k-1})
\in M_B'(\vec{w}).
\]
Our main interest is the action of
\[
T_{\pa}:=T\times\prod_{i=1}^{k-1}Z(C_i)\subset B_{\pa}.
\]
We use Convention \ref{convention:group_action}. Recall from \eqref{eqn:action_on_M_B'} and \eqref{eqn:decomposition_for_Bw_iP_i} that
\[
\hat{A}_j=bA_jb^{-1},
\qquad
\hat{x}_i=bx_ih_i^{-1},
\qquad
\hat{u}_k=bu_k(b^{C_k})^{-1},
\qquad
x_i=\nu_i\dot{w}_in_iz_i.
\]
Since $h_i\in Z(C_i)$ acts only on $x_i$, it acts only on the factor $z_i\in Z(C_i)$ via
$
z_i\mapsto z_ih_i^{-1}.
$
Thus we define
\begin{equation}
M_B''(\vec{w})
:=
\Bigl\{
(u^0,(y_j)_{j=1}^{2g},\vec{\epsilon},({}^+n^j)_{j=1}^{2g},({}^+n_i')_{i=1}^{k-1})
:
u^0D^0\mB_{\beta(\vec{w})}(\vec{\epsilon})=\id_n
\Bigr\}.
\end{equation}
Then $M_B''(\vec{w})$ carries an induced action of $b\in B$, and the above computation shows that
\begin{equation}\label{eqn:M_B''(w)_via_twisted_braid_variety}
M_B''(\vec{w})
\isomorphic
\tilde{X}(\beta(\vec{w}))
\times
\prod_{j=1}^g (U_{\tau_{2j}^{-1}}^+\times U_{\tau_{2j-1}}^+)
\times
\prod_{i=1}^{k-1} (U_{\dot{w}_i}^+\cap N_i),
\qquad
M_B'(\vec{w})
=
M_B''(\vec{w})\times \prod_{i=1}^{k-1} Z(C_i).
\end{equation}

Recall that $PB_{\pa}=B_{\pa}/\field^{\times}$ acts freely on $M_B'(\vec{w})$, and that
\[
M_B'(\vec{w})\rightarrow \modulispace_{\type}(\vec{w})=M_B'(\vec{w})/PB_{\pa}
\]
is a principal $PB_{\pa}$-bundle. Observe that
\[
\prod_{i=1}^{k-1} Z(C_i)
\isomorphic
(\field^{\times}\times \prod_{i=1}^{k-1} Z(C_i))/\field^{\times}
\hookrightarrow
PB_{\pa}
=
(B\times \prod_{i=1}^{k-1} Z(C_i))/\field^{\times}
\]
is a normal closed subgroup, with quotient
\[
PB_{\pa}/\prod_{i=1}^{k-1} Z(C_i)\isomorphic PB:=B/\field^{\times}.
\]
Clearly,
\[
M_B'(\vec{w})/\prod_{i=1}^{k-1}Z(C_i)\isomorphic M_B''(\vec{w}).
\]
Hence, by Proposition \ref{prop:subgroup_action_on_a_principal_bundle} and Proposition \ref{prop:associated_fiber_bundles},
\begin{equation}\label{eqn:Bruhat_cell_of_character_variety_via_geometric_quotient_by_PB}
\modulispace_{\type}(\vec{w})
=
M_B'(\vec{w})/PB_{\pa}
\isomorphic
M_B''(\vec{w})/PB,
\end{equation}
and the induced quotient map
\[
\pi_{\vec{w}}'':
M_B''(\vec{w})\rightarrow \modulispace_{\type}(\vec{w})=M_B''(\vec{w})/PB
\]
is a principal $PB$-bundle. We are thus led to study the $B$-action on $M_B''(\vec{w})$.

\begin{lemma}\label{lem:B-action_on_M_B''(w)}
The induced action of $b\in B$ on
$
(u^0,(y_j)_{j=1}^{2g},\vec{\epsilon},({}^+n^j)_{j=1}^{2g},({}^+n_i')_{i=1}^{k-1})
\in M_B''(\vec{w})
$
satisfies:
\begin{enumerate}[wide,labelwidth=!,labelindent=0pt,label=\alph*)]
\item
the canonical projection
\[
M_B''(\vec{w})\rightarrow \tilde{X}(\beta(\vec{w}))\subset T^{2g}\times X(\beta(\vec{w}))
\]
is $B$-equivariant.
Here, $b\in B$ acts diagonally on $((y_j)_{j=1}^{2g},\vec{\epsilon})\in T^{2g}\times X(\beta(\vec{w}))$: the action on $X(\beta(\vec{w}))$ is that of Definition \ref{def:braid_varieties}, while
\begin{equation}\label{eqn:Borel_group_action_on_y_j}
\hat{y}_{2j-1}=D(b)^{\tau_{2j-1}^{-1}}D(b)^{-1}y_{2j-1},
\qquad
\hat{y}_{2j}=D(b)(D(b)^{\tau_{2j}})^{-1}y_{2j}.
\end{equation}

\item
if $b=t\in T$, then
\begin{eqnarray}\label{eqn:torus_action_on_M_B''}
&&
{}^+\hat{n}_i'
=
t^{\dot{w}_i^{-1}}\,{}^+n_i'\,(t^{-1})^{\dot{w}_i^{-1}},
\qquad
{}^+\hat{n}^{2j-1}=t\,{}^+n^{2j-1}t^{-1},
\qquad
{}^+\hat{n}^{2j}=t\,{}^+n^{2j}t^{-1},
\\
&&
\hat{y}_{2j-1}=t^{\tau_{2j-1}^{-1}}t^{-1}y_{2j-1},
\qquad
\hat{y}_{2j}=t(t^{\tau_{2j}})^{-1}y_{2j},
\qquad
\hat{u}^0=tu^0t^{-1},
\qquad
\hat{D}^0=tD^0(t^{-1})^{\prod_{m=1}^g\lbracket \tau_{2m-1}, \tau_{2m}\rbracket},
\nonumber\\
&&
[\hat{u}^0]\circ[\hat{D}^0]\circ[\mB_{\beta(\vec{w})}(\hat{\vec{\epsilon}})]'
=
[t]\circ[u^0]\circ[D^0]\circ[\mB_{\beta(\vec{w})}(\vec{\epsilon})]'\circ[t^{-1}]
\in\underline{\FBD}_n.
\nonumber
\end{eqnarray}
Here, the last equation uniquely determines $(\hat{u}^0,\hat{D}^0,\hat{\vec{\epsilon}})$.
\end{enumerate}
\end{lemma}

\begin{proof}
We have
$\hat{A}_j=bA_jb^{-1}$, and $\hat{x}_i=bx_i$.
Hence
\[
[\hat{A}_j]=[b] \circ [A_j] \circ [b^{-1}],
\qquad
[\hat{x}_iC_i\hat{x}_i^{-1}]'=[b] \circ [x_iC_ix_i^{-1}]' \circ [b^{-1}].
\]

\noindent\textbf{Proof of a).}
By Section \ref{subsec:diagram_calculus_for_genera},
$(\hat{\mu}_{2j-1},\hat{y}_{2j-1},\hat{\eta}_{2j-1})\in U_{\tau_{2j-1}^{-1}}^-\times T\times U$ is uniquely determined by
\[
\hat{\mu}_{2j-1}\tau_{2j-1}\hat{y}_{2j-1}\hat{\eta}_{2j-1}
=
\hat{A}_{2j-1}
=
bA_{2j-1}b^{-1}
=
b\mu_{2j-1}\tau_{2j-1}y_{2j-1}\eta_{2j-1}b^{-1}.
\]
Therefore
\[
\hat{y}_{2j-1}=D(b)^{\tau_{2j-1}^{-1}}y_{2j-1}D(b)^{-1}
=
D(b)^{\tau_{2j-1}^{-1}}D(b)^{-1}y_{2j-1}\in T,
\]
as desired. Similarly,
\[
\hat{y}_{2j}=D(b)y_{2j}(D(b)^{-1})^{\tau_{2j}}
=
D(b)(D(b)^{\tau_{2j}})^{-1}y_{2j}\in T.
\]

Recall that $\hat{u}_k=bu_k(b^{C_k})^{-1}$ and $D_k=C_k$. Then we have equalities in $\underline{\FBD}_n$:
\begin{eqnarray}\label{eqn:Borel_group_action_on_M_B''(w)_vs_braid_variety}
\hspace{-2em}&&
[\hat{u}^0\hat{D}^0]\circ[\mB_{\beta(\vec{w})}(\hat{\vec{\epsilon}})]'
=
[\hat{\mM}_{\vec{w}}]
=
\prod_{j=1}^g
([\hat{A}_{2j-1}] \text{\tiny$\circ$} [\hat{A}_{2j}] \text{\tiny$\circ$} [\hat{A}_{2j-1}^{-1}] \text{\tiny$\circ$} [\hat{A}_{2j}^{-1}])
\circ
\prod_{i=1}^{k-1}
[\hat{x}_iC_i\hat{x}_i^{-1}]'
\circ
[\hat{u}_kD_k]
\nonumber\\
\hspace{-2em}&=&
[b]
\circ 
\prod_{j=1}^g ([A_{2j-1}] \text{\tiny$\circ$} [A_{2j}] \text{\tiny$\circ$} [A_{2j-1}^{-1}] \text{\tiny$\circ$} [A_{2j}^{-1}])
\circ
\prod_{i=1}^{k-1}
[x_iC_ix_i^{-1}]'
\circ
[b^{-1}]
\circ
[\hat{u}_kD_k]
\nonumber\\
\hspace{-2em}&=&
[b] \circ [u^0D^0] \circ [\mB_{\beta(\vec{w})}(\vec{\epsilon})]'
\circ
[D_k^{-1}u_k^{-1}b^{-1}\hat{u}_kD_k]
=
[b] \circ [u^0D^0] \circ [\mB_{\beta(\vec{w})}(\vec{\epsilon})]' \circ [b^{-1}].
\end{eqnarray}
This equation uniquely determines $(\hat{u}^0,\hat{D}^0,\hat{\vec{\epsilon}})$, and shows that $\hat{\vec{\epsilon}}$ coincides with the action of $b\in B$ on $\vec{\epsilon}\in X(\beta(\vec{w}))$ from Definition \ref{def:braid_varieties}. This proves part a) and most of part b).

\medskip
\noindent\textbf{Proof of b).}
It remains to determine the action of $t\in T$ on $(({}^+n^j)_j,({}^+n_i')_i)$.

From the equality
\[
\hat{\nu}_i\dot{w}_i\hat{n}_i\hat{z}_i
=
\hat{x}_i
=
tx_i
=
t\nu_i\dot{w}_in_iz_i
=
(t\nu_it^{-1})\dot{w}_i\bigl(t^{\dot{w}_i^{-1}}n_i(t^{\dot{w}_i^{-1}})^{-1}\bigr)\bigl(t^{\dot{w}_i^{-1}}z_i\bigr)
\]
and the decomposition \eqref{eqn:decomposition_for_Bw_iP_i}, we obtain
$
\hat{n}_i=t^{\dot{w}_i^{-1}}n_i(t^{\dot{w}_i^{-1}})^{-1}.
$
Hence
$
\hat{n}_i'
=
\hat{n}_iC_i\hat{n}_i^{-1}C_i^{-1}
=
t^{\dot{w}_i^{-1}}n_i'(t^{\dot{w}_i^{-1}})^{-1}.
$
Therefore
$
{}^+\hat{n}_i'
=
R_{\dot{w}_i}^+(\hat{n}_i')
=
t^{\dot{w}_i^{-1}}\,{}^+n_i'\,(t^{\dot{w}_i^{-1}})^{-1},
$
as required.

It remains to compute ${}^+\hat{n}^j$. By \eqref{eqn:diagram_calculus_for_punctures} and, similarly, by \eqref{eqn:Borel_group_action_on_M_B''(w)_vs_braid_variety}, we get inductively that
$
\hat{u}_i=tu_it^{-1}.
$
Likewise, by \eqref{eqn:diagram_calculus_for_genera}, we obtain
$
\hat{u}^j=tu^jt^{-1}.
$
From
\[
\hat{\mu}_{2j-1}\tau_{2j-1}\hat{y}_{2j-1}\hat{\eta}_{2j-1}
=
t\mu_{2j-1}\tau_{2j-1}y_{2j-1}\eta_{2j-1}t^{-1}
=
(t\mu_{2j-1}t^{-1})\tau_{2j-1}(t^{\tau_{2j-1}^{-1}}y_{2j-1}t^{-1})(t\eta_{2j-1}t^{-1}),
\]
together with the decomposition \eqref{eqn:decomposition_for_Btau_jB}, we deduce
$\hat{\mu}_{2j-1}=t\mu_{2j-1}t^{-1}$, and $\hat{\eta}_{2j-1}=t\eta_{2j-1}t^{-1}$.
Similarly,
$\hat{\mu}_{2j}=t\mu_{2j}t^{-1}$, and $\hat{\eta}_{2j}=t\eta_{2j}t^{-1}$.
Now Remark \ref{rem:action_on_+n^j} implies
$
{}^+\hat{n}^j=t\,{}^+n^j\,t^{-1},
$
as desired.
\end{proof}

\subsection{The cell decomposition}\label{subsec:the_cell_decomposition}

Recall that
\[
\pi_{\vec{w}}'':
M_B''(\vec{w})\rightarrow \modulispace_{\type}(\vec{w})=M_B''(\vec{w})/PB
\]
is a principal $PB$-bundle. Moreover, by \eqref{eqn:M_B''(w)_via_twisted_braid_variety},
\[
M_B''(\vec{w})
\isomorphic
\tilde{X}(\beta(\vec{w}))
\times
\prod_{j=1}^g (U_{\tau_{2j}^{-1}}^+\times U_{\tau_{2j-1}}^+)
\times
\prod_{i=1}^{k-1} (U_{\dot{w}_i}^+\cap N_i).
\]
Write
\[
\beta:=\beta(\vec{w})=\sigma_{i_{\ell}} \cdots \sigma_{i_1}\in\FBr_n^+.
\]
For each $p\in\cW(\beta)$, define the $B$-invariant locally closed subvariety
\begin{equation}\label{eqn:cell_of_restricted_twisted_braid_variety}
\tilde{X}_p(\beta):=\tilde{X}(\beta)\cap(T^{2g}\times X_p(\beta))\subset T^{2g}\times X(\beta).
\end{equation}
Thus by Lemma \ref{lem:B-action_on_M_B''(w)}, the $B$-equivariant decomposition
\[
\tilde{X}(\beta) = \sqcup_{p\in\cW(\beta)} \tilde{X}_p(\beta)
\]
induces corresponding decompositions of $M_B''(\vec{w})$ and of $\modulispace_{\type}(\vec{w})$. It remains to describe these pieces more concretely. This reduces to understanding $\tilde{X}_p(\beta)$.

\medskip

We begin with the following observation. Since $PB$ acts freely on $M_B''(\vec{w})$, it follows from \eqref{eqn:torus_action_on_M_B''} that
$
PT=T/\field^{\times}\subset PB
$
preserves $\tilde{X}(\beta)\isomorphic \tilde{X}(\beta)\times\{0\}\subset M_B''(\vec{w})$ and acts on it freely. This free torus action leads to a more concrete description of $\tilde{X}(\beta)$.

By Proposition \ref{prop:cell_decomposition_of_braid_varieties},
\[
X_p(\beta)\isomorphic (\field^{\times})^{|S_p|}\times\field^{|U_p|},
\qquad
\vec{\epsilon}=(\epsilon_m)_{m=\ell}^1\mapsto (\epsilon_m')_{m\in S_p\sqcup U_p}.
\]
Define
\begin{eqnarray}
\hspace{-2em}&&\overline{\phi}_{\vec{w},p}:T^{2g}\times (\field^{\times})^{|S_p|}\rightarrow T_1,
\qquad
(\vec{y}=(y_j)_{j=1}^{2g},(\epsilon_m')_{m\in S_p})\mapsto D^0(\vec{y})\mumon((\epsilon_m')_{m\in S_p}),\\
\hspace{-2em}&&\tilde{T}_p(\beta):=\overline{\phi}_{\vec{w},p}^{-1}(I_n)\subset T^{2g}\times (\field^{\times})^{|S_p|}.\nonumber
\end{eqnarray}
Here, $D^0(\vec{y})$ is given by \eqref{eqn:D^0}, and $\mumon((\epsilon_m')_{m\in S_p})$ is given by Proposition \ref{prop:cell_decomposition_of_braid_varieties} (2).

Clearly, $\phi_{\vec{w},p}:=\phi_{\vec{w}}|_{T^{2g}\times X_p(\beta)}$ is the composition
\[
\phi_{\vec{w},p}:
T^{2g}\times X_p(\beta)
\isomorphic
T^{2g}\times (\field^{\times})^{|S_p|}\times\field^{|U_p|}
\twoheadrightarrow
T^{2g}\times (\field^{\times})^{|S_p|}
\xrightarrow[]{\overline{\phi}_{\vec{w},p}}
T_1.
\]
Hence
\begin{equation}
\tilde{X}_p(\beta)
=
\phi_{\vec{w},p}^{-1}(I_n)
=
\tilde{T}_p(\beta)\times \field^{|U_p|}.
\end{equation}
In particular,
\[
\tilde{X}_p(\beta)\neq\emptyset
\quad\Longleftrightarrow\quad
\tilde{T}_p(\beta)\neq\emptyset.
\]

Also, by Proposition \ref{prop:cell_decomposition_of_braid_varieties} and \eqref{eqn:Borel_group_action_on_y_j}, the variety
$
T^{2g}\times (\field^{\times})^{|S_p|}
$
inherits a $T$-action such that the projection
\[
T^{2g}\times X_p(\beta)
\isomorphic
T^{2g}\times (\field^{\times})^{|S_p|}\times\field^{|U_p|}
\twoheadrightarrow
T^{2g}\times (\field^{\times})^{|S_p|}
\]
is equivariant with respect to the projection $B\rightarrow T: b\mapsto D(b)$.
Moreover,
\begin{equation}\label{eqn:free_torus_action_on_torus_factor_of_cell_decomposition_of_character_variety}
\text{$PT$ acts freely on $\tilde{X}_p(\beta)$ if and only if it acts freely on $\tilde{T}_p(\beta)$.}
\end{equation}
By the observation above, this is indeed the case whenever $\tilde{X}_p(\beta)\neq\emptyset$.

\begin{lemma}\label{lem:free_torus_action_on_twisted_braid_variety}
The $PT$-action on $\tilde{T}_p(\beta)$ is free if and only if
$
\overline{\phi}_{\vec{w},p}:T^{2g}\times(\field^{\times})^{|S_p|}\rightarrow T_1
$
is surjective.
\end{lemma}

\begin{proof}
We first analyze the $PT$-action.

\smallskip
\noindent\textbf{Step 1. Fixed subtori of \(T\).}
For any $\tau\in S_n$, define the subtorus
\[
{}^{\tau}T:=\{t\in T:t^{\tau}=t\}=\{t\in T:t_a=t_{\tau(a)},\ \forall a\in[n]\}\subset T.
\]
In particular,
$
{}^{(a~b)}T=\{t\in T:t_a=t_b\}.
$
For any subset $I\subset S_n$, define
\begin{equation}\label{eqn:^IT_1}
{}^IT:=\{t\in T:t^{\tau}=t,\ \forall \tau\in I\}
=
\bigcap_{\tau\in I}{}^{\tau}T
=
\bigcap_{\tau\in I,\ a\in[n]}{}^{(a~\tau(a))}T.
\end{equation}

Since ${}^{\tau}T={}^{\tau^{-1}}T$, and $t=t^{\tau_1}, t=t^{\tau_2}$ imply $t=t^{\tau_1\tau_2}$, we get
\begin{equation}\label{eqn:^IT_2}
{}^IT={}^{\langle I\rangle}T={}^{\overline{\langle I\rangle}}T,
\end{equation}
where $\langle I\rangle\subset S_n$ is the subgroup generated by $I$, and
\[
\overline{\langle I\rangle}:=\langle (a~\tau(a)):\ a\in[n],\ \tau\in I\rangle\subset S_n.
\]

Now let $[n]=O_1\sqcup\cdots\sqcup O_r$ be a partition, written $\vec{O}$. Define
\[
{}^{\vec{O}}T
:=
\{t\in T: a,b\in O_m\Rightarrow t_a=t_b,\ \forall 1\leq m\leq r\}\subset T,
\qquad
S_{\vec{O}}:=S_{|O_1|}\times\cdots\times S_{|O_r|}\subset S_n.
\]
Then for any $I\subset S_n$,
\begin{equation}\label{eqn:^IT_3}
{}^IT={}^{\vec{O}}T
\quad\Longleftrightarrow\quad
\{O_i\}=\{\text{$\langle I\rangle$-orbits on }[n]\}
\quad\Longleftrightarrow\quad
\overline{\langle I\rangle}=S_{\vec{O}}.
\end{equation}
Now, we see that ${}^IT = {}^JT$ if and only if $\overline{\langle I\rangle} = \overline{\langle J\rangle}$.
In particular,
$
{}^IT=\field^{\times}I_n
\quad\Longleftrightarrow\quad
\overline{\langle I\rangle}=S_n.
$

\smallskip
\noindent\textbf{Step 2. Criterion for freeness of the \(PT\)-action.}
The $PT$-action on $\tilde{T}_p(\beta)$ is free if and only if, for every $t\in T$ and every point
$
((y_j)_{j=1}^{2g},(\epsilon_m')_{m\in S_p})\in \tilde{T}_p(\beta),
$
the equality
$
t\cdot ((y_j)_{j=1}^{2g},(\epsilon_m')_{m\in S_p})
=
((y_j)_{j=1}^{2g},(\epsilon_m')_{m\in S_p})
$
implies $t\in\field^{\times}I_n$.

By \eqref{eqn:torus_action_on_M_B''}, we have $\hat{y}_j=y_j$ if and only if $t^{\tau_j}=t$.
Also, by Proposition \ref{prop:cell_decomposition_of_braid_varieties} (1), for any $m\in S_p$, we have $\hat{\epsilon}_m'=\epsilon_m'$ if and only if
$
(t^{p_{m-1}})_{i_m}=(t^{p_{m-1}})_{i_m+1}.
$
Equivalently,
$(t^{p_{m-1}})^{\ms_{i_m}}=t^{p_{m-1}}$, i.e. $t^{\underline{\ms}_{i_m}}=t$,
where
\begin{equation}\label{eqn:underline_s_i_m}
\underline{\ms}_{i_m}:=p_{m-1}^{-1}\ms_{i_m}p_{m-1}\in S_n.
\end{equation}
Thus, if we define
\begin{equation}\label{eqn:J_p}
J_p:=\{\tau_j:1\leq j\leq 2g\}\cup\{\underline{\ms}_{i_m}:m\in S_p\}\subset S_n,
\end{equation}
then by the discussion above,
\begin{equation}\label{eqn:free_torus_action_on_twisted_braid_variety}
\text{$PT$ acts freely on $\tilde{T}_p(\beta)$}
\quad\Longleftrightarrow\quad
{}^{J_p}T=\field^{\times}I_n
\quad\Longleftrightarrow\quad
\overline{\langle J_p\rangle}=S_n.
\end{equation}

\smallskip
\noindent\textbf{Step 3. Image of \(\overline{\phi}_{\vec{w},p}\).}
We now analyze the image of $\overline{\phi}_{\vec{w},p}$.

For any $\tau\in S_n$, define the subtorus
\[
{}_{\tau}T:=\{y(y^{-1})^{\tau}:y\in T\}\subset T_1.
\]
In particular,
$
{}_{(a~b)}T=\{\mK_a(\lambda)\mK_b(\lambda^{-1}):\lambda\in\field^{\times}\}.
$
For any subset $I\subset S_n$, define the subgroup
\[
{}_IT:=\langle {}_{\tau}T:\tau\in I\rangle
=
\langle \mK_a(\lambda)\mK_{\tau(a)}(\lambda^{-1}):\lambda\in\field^{\times},\ a\in[n],\ \tau\in I\rangle
\subset T_1.
\]
Note that ${}_{\tau}T={}_{\tau^{-1}}T$, and $y(y^{-1})^{\tau_1\tau_2}=(y(y^{-1})^{\tau_2})(y^{\tau_2}((y^{-1})^{\tau_2})^{\tau_1})$, it follows that
\[
{}_IT={}_{\langle I\rangle}T={}_{\overline{\langle I\rangle}}T.
\]
Moreover, if $[n]=O_1\sqcup\cdots\sqcup O_r$ is a partition, define
\[
{}_{\vec{O}}T:=\{t\in T_1:\prod_{q\in O_a}t_q=1,\ \forall 1\leq a\leq r\}\subset T_1.
\]
Then
\[
{}_IT={}_{\vec{O}}T
\quad\Longleftrightarrow\quad
\{O_i\}=\{\text{$\langle I\rangle$-orbits on }[n]\}
\quad\Longleftrightarrow\quad
\overline{\langle I\rangle}=S_{\vec{O}}.
\]
Now, we see that ${}_IT = {}_JT$ if and only if $\overline{\langle I\rangle} = \overline{\langle J\rangle}$.
In particular,
$
{}_IT=T_1
\quad\Longleftrightarrow\quad
\overline{\langle I\rangle}=S_n.
$

Now return to our setting. From \eqref{eqn:shape}, the underlying permutation of $\beta=\beta(\vec{w})$ is
\[
\ms(\beta) = \prod_{j=1}^g \lbracket \tau_{2j-1},\tau_{2j}\rbracket.
\]
Inspired by Proposition \ref{prop:cell_decomposition_of_braid_varieties} (2) and \eqref{eqn:D^0}, define an isomorphism
\[
m_p:T_1\xrightarrow[]{\simeq} T_1,
\qquad
t\mapsto
t(D_1^{-1})^{\ms(\beta)}
\prod_{m\in S_p}
(\mK_{i_m}(-1))^{\ms_{>m}(\beta)}.
\]
Then $\im(\overline{\phi}_{\vec{w},p})=T_1$ if and only if $\im(m_p\circ\overline{\phi}_{\vec{w},p})=T_1$.

For each $1\leq j\leq g$, set
\[
c_{<j} = c_{<j}(\vec{\tau}) := \prod_{m=1}^{j-1} \lbracket \tau_{2m-1},\tau_{2m}\rbracket.
\]
From \eqref{eqn:D^0}, we observe that
\begin{eqnarray*}
&&(y_{2j-1}^{\tau_{2j-1}}(y_{2j-1}^{-1})^{\tau_{2j-1}\tau_{2j}})^{\prod_{m=1}^{j-1}\lbracket \tau_{2m-1}, \tau_{2m}\rbracket}
=
\tilde{y}_{2j-1}(\tilde{y}_{2j-1}^{-1})^{c_{<j+1}\tau_{2j}c_{<j}^{-1}},
\qquad
\tilde{y}_{2j-1}:=y_{2j-1}^{c_{<j}\tau_{2j-1}},\\
&&(y_{2j}^{\tau_{2j-1}}(y_{2j}^{-1})^{\lbracket \tau_{2j-1}, \tau_{2j}\rbracket})^{\prod_{m=1}^{j-1}\lbracket \tau_{2m-1}, \tau_{2m}\rbracket}
=
\tilde{y}_{2j}(\tilde{y}_{2j}^{-1})^{c_{<j+1}\tau_{2j-1}^{-1}c_{<j}^{-1}},
\qquad
\tilde{y}_{2j}:=y_{2j}^{c_{<j}\tau_{2j-1}}.
\end{eqnarray*}
Also,
\[
\Bigl\{
(\mK_{i_m}(\epsilon_m'^{-1})\mK_{i_m+1}(\epsilon_m'))^{\ms_{>m}(\beta)}
:\epsilon_m'\in\field^{\times}
\Bigr\}
=
{}_{(\ms_{>m}(\beta)(i_m)\ \ms_{>m}(\beta)(i_m+1))}T
=
{}_{\ms_{>m}(\beta)\ms_{i_m}\ms_{>m}(\beta)^{-1}}T.
\]
Combining these with Proposition \ref{prop:cell_decomposition_of_braid_varieties} (2), we obtain that
$
\im(m_p\circ\overline{\phi}_{\vec{w},p})={}_{\tilde{J}_p}T,
$
where
\begin{equation}
\tilde{J}_p
:=
\{c_{<j+1}\tau_{2j}c_{<j}^{-1},\ c_{<j+1}\tau_{2j-1}^{-1}c_{<j}^{-1}:1\leq j\leq g\}
\cup
\{\ms_{>m}(\beta)\ms_{i_m}\ms_{>m}(\beta)^{-1}:m\in S_p\}.
\end{equation}
Hence
\begin{equation}\label{eqn:microlocal_monodromy_on_twisted_braid_variety}
\im(m_p\circ\overline{\phi}_{\vec{w},p})=T_1
\quad\Longleftrightarrow\quad
{}_{\tilde{J}_p}T=T_1
\quad\Longleftrightarrow\quad
\overline{\langle \tilde{J}_p\rangle}=S_n.
\end{equation}

Combining \eqref{eqn:free_torus_action_on_twisted_braid_variety} and \eqref{eqn:microlocal_monodromy_on_twisted_braid_variety}, it remains to show
\[
\overline{\langle J_p\rangle}=S_n
\quad\Longleftrightarrow\quad
\overline{\langle \tilde{J}_p\rangle}=S_n.
\]
This follows from the following claim.

\smallskip
\noindent\textbf{Claim.}
$
\langle J_p\rangle=\langle \tilde{J}_p\rangle.
$

\smallskip
\noindent\emph{Proof of Claim.}
Set
\[
J_{\tau}:=\{\tau_m:1\leq m\leq 2g\}\subset S_n,
\qquad
\tilde{J}_{\tau}:=\{\tilde{\tau}_{2j}:=c_{<j+1}\tau_{2j}c_{<j}^{-1},\ \tilde{\tau}_{2j-1}:=c_{<j+1}\tau_{2j-1}^{-1}c_{<j}^{-1}:1\leq j\leq g\}.
\]
We first prove by induction on $j$ that
\begin{equation}\label{eqn:group_generated_by_tau_m}
\langle \tau_m:1\leq m\leq 2j\rangle
=
\langle \tilde{\tau}_m:1\leq m\leq 2j\rangle.
\end{equation}

For $j=1$, we have
$\tilde{\tau}_2=\tau_1\tau_2\tau_1^{-1}$, and $\tilde{\tau}_1=\tau_1\tau_2\tau_1^{-1}\tau_2^{-1}\tau_1^{-1}$,
so
$
\tilde{\tau}_2^{-1}\tilde{\tau}_1\tilde{\tau}_2=\tau_1^{-1}\in\langle \tilde{\tau}_1,\tilde{\tau}_2\rangle.
$
Hence
$
\langle \tilde{\tau}_1,\tilde{\tau}_2\rangle
=
\langle \tau_1,\tau_1\tau_2\tau_1^{-1}\rangle
=
\langle \tau_1,\tau_2\rangle.
$

Now assume \eqref{eqn:group_generated_by_tau_m} holds for $j-1$. Then
\[
c_{<j}\in \langle \tau_m:1\leq m\leq 2(j-1)\rangle
=
\langle \tilde{\tau}_m:1\leq m\leq 2(j-1)\rangle.
\]
Moreover,
$
\tilde{\tau}_{2j}^{-1}\tilde{\tau}_{2j-1}\tilde{\tau}_{2j}
=
c_{<j}\tau_{2j-1}^{-1}c_{<j}^{-1}.
$
Therefore
\begin{eqnarray*}
&&\langle \tilde{\tau}_m:1\leq m\leq 2j\rangle
=
\langle \tau_m:1\leq m\leq 2(j-1),\ c_{<j}\tau_{2j-1}^{-1}c_{<j}^{-1},\ \tilde{\tau}_{2j}=c_{<j+1}\tau_{2j}c_{<j}^{-1}\rangle\\
&=&
\langle \tau_m:1\leq m\leq 2(j-1),\ \tau_{2j-1},\ \tau_{2j-1}\tau_{2j}\tau_{2j-1}^{-1}\rangle
=
\langle \tau_m:1\leq m\leq 2j\rangle.
\end{eqnarray*}
Thus \eqref{eqn:group_generated_by_tau_m} holds for all $j$, and in particular
$
\langle J_{\tau}\rangle=\langle \tilde{J}_{\tau}\rangle.
$

Next, since
$
\ms(\beta)=\prod_{j=1}^g\lbracket \tau_{2j-1},\tau_{2j}\rbracket\in \langle \tilde{J}_{\tau}\rangle=\langle J_{\tau}\rangle,
$
we have
\[
\langle \tilde{J}_p\rangle
=
\langle J_{\tau},\ \ms(\beta)^{-1}\ms_{>m}(\beta)\ms_{i_m}\ms_{>m}(\beta)^{-1}\ms(\beta):m\in S_p\rangle
=
\langle J_{\tau},\ \ms_{<m}(\beta)^{-1}\ms_{i_m}\ms_{<m}(\beta):m\in S_p\rangle.
\]
Set
\[
S_p=\{m_1<\cdots<m_N\},
\qquad
J_s:=\{\underline{\ms}_{i_m}:m\in S_p\},
\qquad
\tilde{J}_s:=\{\tilde{\ms}_{i_m}:=\ms_{<m}(\beta)^{-1}\ms_{i_m}\ms_{<m}(\beta):m\in S_p\}.
\]
It remains to show
\begin{equation}\label{eqn:group_generated_by_stays}
\langle \underline{\ms}_{i_{m_j}}:1\leq j\leq L\rangle
=
\langle \tilde{\ms}_{i_{m_j}}:1\leq j\leq L\rangle,
\qquad \forall 1\leq L\leq N.
\end{equation}

First we prove by induction that, for each $1\leq j\leq N$,
\begin{equation}\label{eqn:inductive_conjugation_over_a_stay}
\prod_{q=j}^1 \underline{\ms}_{i_{m_q}}
=
p_{m_j-1}^{-1}\ms_{<m_j+1}(\beta).
\end{equation}
For $j=1$, since $p_{m_1-1}=\ms_{<m_1}(\beta)$,
$
\underline{\ms}_{i_{m_1}}
=
p_{m_1-1}^{-1}\ms_{i_{m_1}}p_{m_1-1}
=
p_{m_1-1}^{-1}\ms_{<m_1+1}(\beta).
$

Assume \eqref{eqn:inductive_conjugation_over_a_stay} holds for $j$. Then
\[
\prod_{q=j+1}^1 \underline{\ms}_{i_{m_q}}
=
\underline{\ms}_{i_{m_{j+1}}} \prod_{q=j}^1 \underline{\ms}_{i_{m_q}}\\
=
(p_{m_{j+1}-1}^{-1}\ms_{i_{m_{j+1}}}p_{m_{j+1}-1})\,
p_{m_j-1}^{-1}\ms_{<m_j+1}(\beta)\\
=
p_{m_{j+1}-1}^{-1}\ms_{<m_{j+1}+1}(\beta).
\]
So \eqref{eqn:inductive_conjugation_over_a_stay} holds for all $j$.

Now for each $1\leq j\leq N$,
\begin{eqnarray*}
\Bigl( \prod_{q=j-1}^1 \underline{\ms}_{i_{m_q}}\Bigr)^{-1}
\underline{\ms}_{i_{m_j}}
\Bigl( \prod_{q=j-1}^1 \underline{\ms}_{i_{m_q}}\Bigr)
&=&
(p_{m_{j-1}-1}^{-1}\ms_{<m_{j-1}+1}(\beta))^{-1}
p_{m_j-1}^{-1}\ms_{i_{m_j}}p_{m_j-1}
(p_{m_{j-1}-1}^{-1}\ms_{<m_{j-1}+1}(\beta))\\
&=&
\ms_{<m_j}(\beta)^{-1}\ms_{i_{m_j}}\ms_{<m_j}(\beta)
=
\tilde{\ms}_{i_{m_j}}.
\end{eqnarray*}
Hence \eqref{eqn:group_generated_by_stays} follows, so $\langle J_s\rangle=\langle \tilde{J}_s\rangle$. Combining this with $\langle J_{\tau}\rangle=\langle \tilde{J}_{\tau}\rangle$, we obtain
$
\langle J_p\rangle=\langle \tilde{J}_p\rangle,
$
proving the claim and therefore the lemma.
\end{proof}

Inspired by this lemma, we make the following definition.

\begin{definition}\label{def:admissible_walks}
For $\beta=\beta(\vec{w})$, define
\begin{equation}\label{eqn:admissible_walks}
\cW^*(\beta):=\{p\in\cW(\beta):\tilde{T}_p(\beta)\neq\emptyset\},
\end{equation}
and call each $p\in \cW^*(\beta)$ an \emph{admissible walk}.

By the proof of Lemma \ref{lem:free_torus_action_on_twisted_braid_variety},
\[
p\in\cW^*(\beta)
\quad\Longleftrightarrow\quad
\langle J_p\rangle \text{ acts transitively on }[n]
\quad\Longleftrightarrow\quad
\overline{\langle J_p\rangle}=S_n,
\qquad \forall p\in\cW(\beta),
\]
where
\[
J_p=\{\tau_j:1\leq j\leq 2g\}\cup\{\underline{\ms}_{i_m}=p_{m-1}^{-1}\ms_{i_m}p_{m-1}:m\in S_p\}\subset S_n.
\]
Alternatively, define
\[
J_p':=\{\tau_j:1\leq j\leq 2g\}\cup\{\tilde{\ms}_{i_m}=\ms_{<m}(\beta)^{-1}\ms_{i_m}\ms_{<m}(\beta):m\in S_p\}\subset S_n.
\]
Since $\langle J_p\rangle=\langle J_p'\rangle$, the same criterion holds with $J_p'$ in place of $J_p$.
\end{definition}

As an immediate corollary of Lemma \ref{lem:free_torus_action_on_twisted_braid_variety}, we obtain:

\begin{corollary}\label{cor:cell_decomposition_of_twisted_braid_variety}
If $\tilde{X}_p(\beta)\neq\emptyset$, then there is a $B$-equivariant isomorphism
\[
\tilde{X}_p(\beta)=\tilde{T}_p(\beta)\times\field^{|U_p|}
\isomorphic
(\field^{\times})^{|S_p|+2gn-n+1}\times \field^{|U_p|}.
\]
\end{corollary}

\begin{proof}
If $\tilde{X}_p(\beta)\neq\emptyset$, then $PT$ acts freely on $\tilde{X}_p(\beta)$, hence also freely on $\tilde{T}_p(\beta)$ by \eqref{eqn:free_torus_action_on_torus_factor_of_cell_decomposition_of_character_variety}. By Lemma \ref{lem:free_torus_action_on_twisted_braid_variety}, the map
\[
\overline{\phi}_{\vec{w},p}:T^{2g}\times(\field^{\times})^{|S_p|}
\isomorphic
(\field^{\times})^{|S_p|+2gn}
\rightarrow T_1\isomorphic (\field^{\times})^{n-1}
\]
is surjective.

By \eqref{eqn:phi_w}, Proposition \ref{prop:cell_decomposition_of_braid_varieties} (2), and \eqref{eqn:D^0}, the composition $m_p\circ\overline{\phi}_{\vec{w},p}$ is a surjective group homomorphism of algebraic tori. Therefore,
\[
\tilde{T}_p(\beta)
=
\overline{\phi}_{\vec{w},p}^{-1}(I_n)
=
(m_p\circ\overline{\phi}_{\vec{w},p})^{-1}
\Bigl(
(D_1^{-1})^{\ms(\beta)}
\prod_{m\in S_p}
(\mK_{i_m}(-1))^{\ms_{>m}(\beta)}
\Bigr)
\isomorphic
(\field^{\times})^{|S_p|+2gn-n+1}.
\]
Hence
\[
\tilde{X}_p(\beta)=\tilde{T}_p(\beta)\times\field^{|U_p|}
\isomorphic
(\field^{\times})^{|S_p|+2gn-n+1}\times\field^{|U_p|}.
\qedhere
\]
\end{proof}

Now define
\begin{equation}\label{eqn:n_w}
n_{\vec{w}}
:=
\sum_{j=1}^{2g} |U_{\tau_j}^+|
+
\sum_{i=1}^{k-1} |U_{\dot{w}_i}^+\cap N_i|
=
2g\dim U + \sum_{i=1}^{k-1} |N_i|-\frac{1}{2}\ell(\beta).
\end{equation}
By Lemma \ref{lem:B-action_on_M_B''(w)} and Corollary \ref{cor:cell_decomposition_of_twisted_braid_variety}, we obtain:

\begin{proposition}\label{prop:B-equivariant_cell_decomposition_of_character_variety}
There is a $B$-equivariant decomposition into locally closed $\field$-subvarieties
\[
M_B''(\vec{w})=\bigsqcup_{p\in\cW^*(\beta)} M_B''(\vec{w},p),\quad M_B''(\vec{w},p)
:=
\tilde{X}_p(\beta)
\times
\prod_{j=1}^g (U_{\tau_{2j}^{-1}}^+\times U_{\tau_{2j-1}}^+)
\times
\prod_{i=1}^{k-1} (U_{\dot{w}_i}^+\cap N_i).
\]
In particular,
\[
M_B''(\vec{w},p)
\isomorphic
(\field^{\times})^{a(\vec{w},p)}\times\field^{b(\vec{w},p)}
=
(\field^{\times})^{|S_p|+2gn-n+1}\times \field^{|U_p|+n_{\vec{w}}},
\]
and
$
a(\vec{w},p)+2b(\vec{w},p)
=
4g\dim U+2 \sum_{i=1}^{k-1} |N_i|+2gn-n+1
$
is independent of $\vec{w}$ and $p$.
\end{proposition}

Recall from \eqref{eqn:Bruhat_cell_of_character_variety_via_geometric_quotient_by_PB} that $PB$ acts freely on $M_B''(\vec{w})$, and that
\[
\pi_{\vec{w}}'':
M_B''(\vec{w})\rightarrow \modulispace_{\type}(\vec{w})=M_B''(\vec{w})/PB
\]
is a principal $PB$-bundle. Define
\begin{equation}\label{eqn:cell_of_character_variety}
\modulispace_{\type}(\vec{w},p):=\pi_{\vec{w}}''(M_B''(\vec{w},p))\hookrightarrow \modulispace_{\type}(\vec{w})
\end{equation}
as a locally closed $\field$-subvariety. By base change (Corollary \ref{cor:base_change_of_principal_bundles}), the restriction
\[
\pi_{\vec{w},p}:M_B''(\vec{w},p)\rightarrow \modulispace_{\type}(\vec{w},p)=M_B''(\vec{w},p)/PB
\]
is again a principal $PB$-bundle, hence a geometric quotient.

Our first main result improves Mellit’s cell decomposition theorem \cite[\S 7]{Mel25}.

\begin{theorem}\label{thm:cell_decomposition_of_very_generic_character_varieties}
If $(C_1,\cdots,C_k)\in T^k$ is very generic (Definition \ref{def:generic_monodromy} and Assumption \ref{ass:very_generic_assumption}) of type $\type$, and if $\modulispace_{\type}$ is nonempty, then there is a decomposition into locally closed affine $\field$-subvarieties
\[
\modulispace_{\type}
=
\bigsqcup_{\vec{w}\in W^{2g} \times \prod_{i=1}^{k-1} W/W(C_i)}
\modulispace_{\type}(\vec{w})
=
\bigsqcup_{\vec{w}\in W^{2g} \times \prod_{i=1}^{k-1} W/W(C_i)}
\ \bigsqcup_{p\in\cW^*(\beta(\vec{w}))}
\modulispace_{\type}(\vec{w},p),
\]
such that:
\begin{enumerate}[wide,labelwidth=!,labelindent=0pt]
\item
we have
\begin{eqnarray}\label{eqn:structure_of_cell_decomposition_of_character_variety}
&&
\modulispace_{\type}(\vec{w},p) \times \field^{\dim U} \isomorphic (\field^{\times})^{\overline{a}(\vec{w},p)}\times \field^{b(\vec{w},p)},\\
&&
\overline{a}(\vec{w},p)=a(\vec{w},p)-n+1=|S_p|+2gn-2n+2,
\qquad
\overline{b}(\vec{w},p):=b(\vec{w},p)-\dim U.
\nonumber
\end{eqnarray}
In particular,
$
\dim\modulispace_{\type}(\vec{w},p)=\overline{a}(\vec{w},p)+\overline{b}(\vec{w},p),
$
and
$
\overline{a}(\vec{w},p)+2\overline{b}(\vec{w},p)
=
d_{\type}
=
n^2(2g-2+k)-\sum_{i,j}(\mu_j^i)^2+2
$
is independent of $(\vec{w},p)$.

\item
there exists a unique pair $(\vec{w}_{\max},p_{\max})$ such that
$
\dim \modulispace_{\type}(\vec{w}_{\max},p_{\max})=d_{\type}.
$
That is,
$\overline{b}(\vec{w}_{\max},p_{\max})=0$,
$\overline{a}(\vec{w}_{\max},p_{\max})=d_{\type}$.
In particular,
\[
\modulispace_{\type}(\vec{w}_{\max},p_{\max})
\isomorphic
(\field^{\times})^{d_{\type}},
\]
so this cell is an open dense algebraic torus.
\end{enumerate}
\end{theorem}

\begin{proof}
\noindent\textbf{(0)} 
By \eqref{eqn:Bruhat_cell_of_character_variety} and Proposition \ref{prop:B-equivariant_cell_decomposition_of_character_variety}, the decomposition of $\modulispace_{\type}$ is clear.

We next show that each $\modulispace_{\type}(\vec{w})$ is affine. Inspired by \eqref{eqn:Equivariant_Bruhat_cell_of_character_variety}, define
\[
\underline{M}_B'(\vec{w})
:=
M_B'\cap
\Bigl(
\prod_{j=1}^{2g} B\tau_jB
\times
\prod_{i=1}^{k-1} B\dot{w}_iP_i
\times
\{I_n\}
\Bigr)
\subset M_B'(\vec{w}).
\]
There are mutually inverse $U$-equivariant isomorphisms
\begin{eqnarray}\label{eqn:U-equivariant_isomorphism_for_M_B'_w}
\hspace{-2em}&&
U \times \underline{M}_B'(\vec{w})
\rightarrow
M_B'(\vec{w}),
\qquad
(u,(A_1,\cdots,x_{k-1},I_n))
\mapsto
(uA_1u^{-1},\cdots,ux_{k-1},u(u^{C_k})^{-1}),
\\
\hspace{-2em}&&
M_B'(\vec{w})
\rightarrow
U \times \underline{M}_B'(\vec{w}),
\qquad
(A_1,\cdots,x_{k-1},u_k=\tilde{u}(\tilde{u}^{C_k})^{-1})
\mapsto
(\tilde{u},(\tilde{u}^{-1}A_1\tilde{u},\cdots,\tilde{u}^{-1}x_{k-1},I_n)),
\nonumber
\end{eqnarray}
where $\tilde{u}\in U$ is uniquely determined by $u_k=\tilde{u}(\tilde{u}^{C_k})^{-1}$. Hence
\[
\underline{M}_B'(\vec{w})=M_B'(\vec{w})/U.
\]
Since $U\subset PB_{\pa}$ is normal and
\[
PB_{\pa}/U\isomorphic PT_{\pa}:=T\times \prod_{i=1}^{k-1} Z(C_i),
\]
Proposition \ref{prop:subgroup_action_on_a_principal_bundle} and Proposition \ref{prop:associated_fiber_bundles} give
\[
\modulispace_{\type}(\vec{w})
=
M_B'(\vec{w})/PB_{\pa}
\isomorphic
\underline{M}_B'(\vec{w})/PT_{\pa}.
\]
Moreover, the quotient map
$
\underline{M}_B'(\vec{w})\rightarrow \modulispace_{\type}(\vec{w})
$
is a principal $PT_{\pa}$-bundle and a geometric quotient. Since $\underline{M}_B'(\vec{w})$ is affine and $PT_{\pa}$ is reductive, we must have
\[
\modulispace_{\type}(\vec{w})
\isomorphic
\underline{M}_B'(\vec{w})//PT_{\pa}
=
\Spec~\cO(\underline{M}_B'(\vec{w}))^{PT_{\pa}}.
\]
Thus $\modulispace_{\type}(\vec{w})$ is affine.

It remains to prove (1) and (2).

\medskip
\noindent\textbf{(1)} 
We first show that each $\modulispace_{\type}(\vec{w},p)$ is affine.

By \eqref{eqn:reparameterization_for_connection_to_braid_varieties}, we have a $PB_{\pa}$-equivariant isomorphism
\begin{eqnarray*}
&&
\Bigl( \prod_{j=1}^{2g} B\tau_jB\Bigr) \times 
\Bigl( \prod_{i=1}^{k-1} B\dot{w}_iP_i\Bigr) \times U
\xrightarrow[]{\isomorphic}
(\text{an affine variety built from }U,T,U_{\tau}^{\pm},U_{\dot{w}_i}^{\pm},Z(C_i)).
\end{eqnarray*}
Hence $M_B'(\vec{w})$ is a closed affine $PB_{\pa}$-subvariety of this affine variety. Also,
\[
M_B'(\vec{w},p):=M_B''(\vec{w},p)\times \prod_{i=1}^{k-1} Z(C_i)\subset M_B'(\vec{w}) = M_B''(\vec{w})\times \prod_{i=1}^{k-1} Z(C_i)
\]
is a locally closed affine $PB_{\pa}$-subvariety. Define
\[
\underline{M}_B'(\vec{w},p)
:=
M_B'(\vec{w},p)\cap
\Bigl(
\prod_{j=1}^{2g} B\tau_jB
\times
\prod_{i=1}^{k-1} B\dot{w}_iP_i
\times
\{I_n\}
\Bigr)
\subset M_B'(\vec{w}).
\]
By the same formula as in \eqref{eqn:U-equivariant_isomorphism_for_M_B'_w}, we obtain a $U$-equivariant isomorphism
\[
U\times \underline{M}_B'(\vec{w},p)\isomorphic M_B'(\vec{w},p).
\]
Thus $\underline{M}_B'(\vec{w},p)\isomorphic M_B'(\vec{w},p)/U$ is affine. Since
\[
\prod_{i=1}^{k-1} Z(C_i)\triangleleft PB_{\pa},
\qquad
PB_{\pa}/\prod_{i=1}^{k-1} Z(C_i)\isomorphic PB,
\]
we deduce, as above,
\[
\modulispace_{\type}(\vec{w},p)
=
M_B''(\vec{w},p)/PB
\isomorphic
M_B'(\vec{w},p)/PB_{\pa}
\isomorphic
\underline{M}_B'(\vec{w},p)/PT_{\pa}
\isomorphic
\Spec~\cO(\underline{M}_B'(\vec{w},p))^{PT_{\pa}}.
\]
Hence $\modulispace_{\type}(\vec{w},p)$ is affine.

Next, recall that the quotient $\pi_{\vec{w},p}: M_B''(\vec{w},p) \to \modulispace_{\type}(\vec{w},p) = M_B''(\vec{w},p)/PB$ is a principal $PB$-bundle, and $U\triangleleft PB$ is normal with quotient $PT$. Then by Proposition \ref{prop:subgroup_action_on_a_principal_bundle} and Proposition \ref{prop:associated_fiber_bundles}, the quotient
\[
\fp_U:
M_B''(\vec{w},p)\isomorphic (\field^{\times})^{a(\vec{w},p)}\times\field^{b(\vec{w},p)}
\rightarrow
\cP(\vec{w},p):=M_B''(\vec{w},p)/U
\]
is a principal $U$-bundle, and the quotient
\[
q_U:\cP(\vec{w},p)\rightarrow \modulispace_{\type}(\vec{w},p)=M_B''(\vec{w},p)/PB
\]
is a principal $PT$-bundle. By Lemma \ref{lem:principal_bundle_in_fpqc_topology}, $q_U$ is affine, hence $\cP(\vec{w},p)$ is affine as well.
Then by Proposition \ref{prop:principal_bundle_for_unipotent_algebraic_group_over_affine_variety_is_trivial}, the principal $U$-bundle $\fp_U$ is trivial. 
By a direct check, it follows that we get a $PB$-equivariant isomorphism 
\[
M_B''(\vec{w},p)= (\field^{\times})^{a(\vec{w},p)}\times\field^{b(\vec{w},p)} \isomorphic U\times \cP(\vec{w},p),
\] 
where
$tu\in PB = PT\ltimes U$ acts on $(v,y)\in U\times \cP(\vec{w},p)$ via: $(tu)\cdot (v,y):= (tuvt^{-1},t\cdot y)$.
Hence by Corollary \ref{cor:standard-quotient}, 
\[
U\times \modulispace_{\type}(\vec{w},p) = U\times (\cP(\vec{w},p)/PT) \isomorphic (\field^{\times})^{\overline{a}(\vec{w},p)}\times \field^{b(\vec{w},p)}.
\]
Finally, by Proposition \ref{prop:B-equivariant_cell_decomposition_of_character_variety},
\begin{eqnarray*}
&&\overline{a}(\vec{w},p)+2\overline{b}(\vec{w},p)
=
a(\vec{w},p)-n+1+2b(\vec{w},p)-2\dim U\\
&=&
(4g-4)\dim U+2\sum_{i=1}^{k} |N_i|+2gn-2n+2
=
(2g-2)n^2 + \sum_{i=1}^{k} \bigl(n^2 - \sum_j (\mu_j^i)^2\bigr)+2\\
&=&
(2g-2+k)n^2 - \sum_{i,j} (\mu_j^i)^2+2
=
d_{\type},
\end{eqnarray*}
which is independent of $(\vec{w},p)$.

\medskip
\noindent\textbf{(2)} 
By Lemma \ref{lem:generic_character_varieties}, $\modulispace_{\type}$ is connected, smooth, affine, and of dimension $d_{\type}$ whenever it is nonempty. Hence there exists a unique pair $(\vec{w}_{\max},p_{\max})$ such that
$
\modulispace_{\type}(\vec{w}_{\max},p_{\max})
$
is open dense of dimension $d_{\type}$.

For any $(\vec{w},p)$ in the decomposition,
\[
\dim\modulispace_{\type}(\vec{w},p)
=
\overline{a}(\vec{w},p)+\overline{b}(\vec{w},p)
\leq
\overline{a}(\vec{w},p)+2\overline{b}(\vec{w},p)
=
d_{\type}.
\]
Therefore
\[
\dim\modulispace_{\type}(\vec{w},p)=d_{\type}
\quad\Longleftrightarrow\quad
\overline{b}(\vec{w},p)=0
\quad\Longleftrightarrow\quad
\overline{a}(\vec{w},p)=d_{\type}.
\]
Finally, we show that $\modulispace_{\type}(\vec{w}_{\max},p_{\max})
\isomorphic
(\field^{\times})^{d_{\type}}$.
By Proposition \ref{prop:B-equivariant_cell_decomposition_of_character_variety}, Lemma \ref{lem:B-action_on_M_B''(w)}, Corollary \ref{cor:cell_decomposition_of_twisted_braid_variety}, and Proposition \ref{prop:cell_decomposition_of_braid_varieties}, the projection
$p_2:M_B''(\vec{w},p)\isomorphic (\field^{\times})^{a(\vec{w},p)}\times\field^{b(\vec{w},p)} \to (\field^{\times})^{a(\vec{w},p)}$ is $PB$-equivariant, where $U$ acts trivially on the base. In particular, $U$ acts freely on each fiber $p_2^{-1}(z) = \{z\} \times \field^{b(\vec{w},p)}$ with closed orbits.
If $(\vec{w},p) = (\vec{w}_{\max},p_{\max})$, then the closed embedding $U\cdot(z,0) \isomorphic U = \field^{\dim U} \hookrightarrow \{z\} \times \field^{b(\vec{w}_{\max},p_{\max})} = \{z\}\times \field^{\dim U}$ becomes an isomorphism.
It follows that $\fp_U = p_2$, and $\cP(\vec{w},p) = M_B''(\vec{w},p)/U =  (\field^{\times})^{a(\vec{w}_{\max},p_{\max})} = (\field^{\times})^{d_{\type} + n -1}$.
Then, $\modulispace_{\type}(\vec{w}_{\max},p_{\max}) = \cP(\vec{w},p)/PT \isomorphic (\field^{\times})^{d_{\type}}$.
\end{proof}

\iffalse

\begin{lemma}\label{lem:algebraic_group_action_on_algebraic_torus}
Let $H$ be a connected algebraic group and let $T'=(\field^{\times})^m_{z_1,\cdots,z_m}$ be an algebraic torus. Then any algebraic $H$-action on $T'$ is given by an algebraic group homomorphism
$
\rho:H\rightarrow T',
$
where $T'$ acts on itself by translation. In particular, if $H$ is unipotent, then the $H$-action on $T'$ is trivial.
\end{lemma}

\begin{proof}
For any $a\in H$, the action induces an automorphism
\[
\rho_a:T'\xrightarrow[]{\simeq}T',
\qquad
z_j\mapsto c_j(a)\prod_{i=1}^mz_i^{a_{ij}},
\]
for some $(a_{ij})\in GL(m,\mathbb{Z})$ and $c_j(a)\in\field^{\times}$. Since $\rho_1=\id$, continuity implies that $a_{ij}=\delta_{ij}$ for all $a\in H$. Thus
\[
\rho:H\rightarrow T',
\qquad
a\mapsto (c_1(a),\cdots,c_m(a))
\]
is an algebraic group homomorphism, and the action is translation by $\rho(a)$.

If $H$ is unipotent, then every $\rho(a)\in T'$ is both unipotent and semisimple, hence trivial. Therefore the action is trivial.
\end{proof}

\fi

\noindent{}\textbf{Question}: if $\type$ is only generic, does $\modulispace_{\type}$ contain an open algebraic torus?

%Recall from the proof of Theorem \ref{thm:cell_decomposition_of_very_generic_character_varieties} that 
%$p_2:M_B''(\vec{w},p)\isomorphic (\field^{\times})^{a(\vec{w},p)}\times\field^{b(\vec{w},p)} \to (\field^{\times})^{a(\vec{w},p)}$ is $PB$-equivariant, where $U$ acts trivially on the base.

%{\color{red}
\begin{remark}\label{rem:Zariski_cancellation_problem}
Fix $(\vec{w},p)$, and write
$a=a(\vec{w},p)$, $b=b(\vec{w},p)$, and similarly
$\overline{a}=\overline{a}(\vec{w},p)$ and
$\overline{b}=\overline{b}(\vec{w},p)$.
Theorem~\ref{thm:cell_decomposition_of_very_generic_character_varieties}
leaves open whether
\[
\modulispace_{\type}(\vec{w},p)
\isomorphic
(\field^{\times})^{\overline{a}}
\times \field^{\overline{b}}.
\]

To explain the cancellation questions involved, fix
$z\in(\field^{\times})^a$ and set
\[
F_z:=\{z\}\times\field^b
\subset M_B''(\vec{w},p),
\qquad
\cA_{\type}(\vec{w},p):=\fp_U(F_z).
\]
Here $F_z$ is a closed $U$-stable subvariety.
By Corollary~\ref{cor:base_change_of_principal_bundles}
and Proposition~\ref{prop:associated_fiber_bundles},
the restriction
\[
\fp_U|_{F_z}\colon
F_z\longrightarrow\cA_{\type}(\vec{w},p)
\isomorphic F_z/U
\]
is a principal $U$-bundle and a geometric quotient.
Moreover, $\cA_{\type}(\vec{w},p)$ is closed in the affine variety
$\cP(\vec{w},p)$, and hence is a connected affine variety of
dimension $\overline{b}=b-\dim U$.
By
Proposition~\ref{prop:principal_bundle_for_unipotent_algebraic_group_over_affine_variety_is_trivial},
this bundle is trivial. Thus
\[
U\times\cA_{\type}(\vec{w},p)
\isomorphic\field^b.
\]
Combining this with the stable isomorphism in
Theorem~\ref{thm:cell_decomposition_of_very_generic_character_varieties}
gives
\[
U\times\modulispace_{\type}(\vec{w},p)
\isomorphic
(\field^{\times})^{\overline{a}}\times\field^b
\isomorphic
U\times(\field^{\times})^{\overline{a}}
\times\cA_{\type}(\vec{w},p).
\]
An isomorphism after taking the product with $U$ does not,
by itself, justify canceling $U$.
It is therefore natural to ask whether
\[
\modulispace_{\type}(\vec{w},p)
\isomorphic
(\field^{\times})^{\overline{a}}
\times\cA_{\type}(\vec{w},p).
\]
A further question is whether
\[
\cA_{\type}(\vec{w},p)
\isomorphic\field^{\overline{b}}.
\]
Together, these two assertions would give the desired
torus-times-affine-space decomposition.

We thank Jialong Sun for pointing out, with the assistance of AI,
a gap in the argument in the published version for the first of these two assertions.
Theorem~\ref{thm:cell_decomposition_of_very_generic_character_varieties}
has been revised accordingly.
This correction does not affect the main result,
Theorem~\ref{thm:weak_geometric_P=W_conjecture_for_very_generic_character_varieties}.

Since $U\isomorphic\field^{\dim U}$ as a variety, these questions
concern cancellation of an affine-space factor.
In particular, the question concerning
$\cA_{\type}(\vec{w},p)$ is a stable version of the classical
\emph{Zariski cancellation problem}:
if an affine $\field$-variety $Y$ of dimension $d$ satisfies
\[
Y\times\mathbb{A}^1\isomorphic\mathbb{A}^{d+1},
\]
must $Y\isomorphic\mathbb{A}^d$?
The answer is affirmative for $d=1,2$
\cite{AHE72,Fuj79,MS80,Rus81},
and negative for $d\geq 3$ in positive characteristic
\cite{Gup14a,Gup14b}.
At the time of the original publication of this paper in July 2026,
the problem was open in characteristic~$0$ for $d\geq 3$.
\end{remark}

\subsection{Example}\label{subsec:examples_of_cell_decomposition_of_very_generic_character_varieties}

\begin{example}[$(g,k,n,\type)=(0,4,2,((1^2),(1^2),(1^2),(1^2)))$: Fricke--Klein cubic]\label{ex:rank_2_over_four-punctured_two-sphere_cell_decomposition}
Let $\Sigma_{0,4}:=(\Sigma_0,\sigma=\{q_1,q_2,q_3,q_4\})$ be the four-punctured sphere, and let $G=GL_2(\field)$. 
Then $T\isomorphic(\field^{\times})^2$ and $W=S_2=\{\id,\ms_1=(1~2)\}$.

Let $\type=((1^2),(1^2),(1^2),(1^2))$. Then $\type$ is very generic (Definition \ref{def:generic_monodromy}, Assumption \ref{ass:very_generic_assumption}), and
\[
C_i=\diag(a_{i,1},a_{i,2})\in T,\quad a_{i,1}\neq a_{i,2},\quad 1\leq i\leq 4,
\]
with
$\prod_{i,j}a_{i,j}=1$, and $\prod_{i=1}^4 a_{i,\psi_i(1)}\neq 1$, $\forall \psi_i\in W$.
Moreover,
$
d_{\type}
=
n^2(2g-2+k)-\sum_{i=1}^k\sum_j(\mu_j^i)^2 +2
=
2.
$
This example goes back to \cite{FK65}.

\begin{enumerate}[wide,labelwidth=!,labelindent=0pt,label=\arabic*)]
\item
We first compute the cell decomposition of $\modulispace_{\type}$ (Theorem \ref{thm:cell_decomposition_of_very_generic_character_varieties}):
\[
\modulispace_{\type}
=\bigsqcup_{(\vec{w},p)\in \cW^*} \modulispace_{\type}(\vec{w},p),\qquad
\modulispace_{\type}(\vec{w},p) \times U
\isomorphic (\field^{\times})^{\overline{a}(\vec{w},p)}\times \mathbb{A}^{\overline{b}(\vec{w},p)+\dim U},
\]
where $\dim U=1$,
and
$\overline{a}(\vec{w},p)=|S_p|-2$,
$\overline{b}(\vec{w},p)=\frac{1}{2}(4-|S_p|)$.
Thus $|S_p|=2$ or $4$, hence
$
(\overline{a},\overline{b})=(0,1)\ \text{or}\ (2,0).
$
By Proposition \ref{prop:torus-cancellation} and Corollary \ref{cor:torus-line-cancellation}, we have
$\modulispace_{\type}(\vec{w},p)\isomorphic \field \ \text{or}\ (\field^{\times})^2$.

We now determine $\cW^*$.  
Fix $\vec{w}=(w_1,w_2,w_3)\in W^3$, and set $\beta=\beta(\vec{w})$ and $\ell=\ell(\beta)$. Recall that
\[
\beta=[w_1]\circ[w_1^{-1}]\circ[w_2]\circ[w_2^{-1}]\circ[w_3]\circ[w_3^{-1}],
\]
so that
$
\ell(\beta)=2\ell(w_1)+2\ell(w_2)+2\ell(w_3)\in 2\mathbb{Z}_{\ge 0}.
$

For any $p\in\cW(\beta)\subset W^{\ell+1}$, denote
\[
J_p=\{\underline{\ms}_{i_m}=p_{m-1}^{-1}\ms_{i_m}p_{m-1}\mid m\in S_p\}\subset W.
\]
Then $p\in\cW^*(\beta)$ if and only if $\langle J_p\rangle$ acts transitively on $[2]=\{1,2\}$.

\medskip

\noindent{}\textbf{Note.} Since $W=S_2$, this is equivalent to $\langle J_p\rangle=W$, hence to $S_p\neq\emptyset$. Thus
\[
\cW^*(\beta)=\{p\in\cW(\beta): S_p\neq\emptyset\}.
\]

\medskip

If there exists $p\in\cW(\beta)$ with $S_p\neq\emptyset$, then $\ell(\beta)>0$, so $\ell(\beta)\ge 2$, and hence $w_i=\ms_1=(1~2)$ for at least one $i\in\{1,2,3\}$.
Recall that $p=(p_\ell=\id,\cdots,p_1,p_0=\id)\in W^{\ell+1}$.  
By Definition \ref{def:walks} of a walk, whenever a go-up is possible, it must be taken.
Thus, there is no $i$ such that $(p_{i+1},p_i)=(\id,\id)$. In particular,
$p_1=\ms_1$ and $p_{\ell-1}=\ms_1$.

If $S_p\neq\emptyset$, then necessarily $\ell\ge 4$. Consequently, at least two of the $w_i$ must be equal to $\ms_1$.

Therefore, we are left with the following four cases:

\begin{enumerate}[wide,labelwidth=!,labelindent=0pt]

\item
$\vec{w}=(\ms_1,\ms_1,\id)\in W^3$. Then $\beta=\sigma_1^4\in\FBr_2^+$ and $\ell=4$. One checks that
\[
\cW(\beta)=\{(\id,\ms_1,\id,\ms_1,\id),\;p^1=(\id,\ms_1,\ms_1,\ms_1,\id)\}\subset W^5.
\]
Thus $\cW^*(\beta)=\{p^1\}$, with $S_{p^1}=\{2,3\}$, $U_{p^1}=\{1\}$, and $D_{p^1}=\{4\}$.
Denote $\vec{w}^1:=(\ms_1,\ms_1,\id)$. Then
\[
\overline{a}(\vec{w}^1,p^1)=0,\quad 
\overline{b}(\vec{w}^1,p^1)=1,\quad 
\modulispace_{\type}(\vec{w}^1,p^1)\isomorphic \field.
\]

\item
Similarly, for $\vec{w}^2=(\ms_1,\id,\ms_1)$ and $\vec{w}^3=(\id,\ms_1,\ms_1)$, we have $\beta = \sigma_1^4\in \FBr_2^+$, and
\[
\cW^*(\beta(\vec{w}^2))=\{p^2 = (\id,s_1,s_1,s_1,\id)\},\quad 
\cW^*(\beta(\vec{w}^3))=\{p^3 = (\id,s_1,s_1,s_1,\id)\},
\]
with
$\modulispace_{\type}(\vec{w}^2,p^2)\isomorphic\field$, 
$\modulispace_{\type}(\vec{w}^3,p^3)\isomorphic\field$.

\item
$\vec{w}=(\ms_1,\ms_1,\ms_1)=:\vec{w}^4\in W^3$. Then $\beta=\sigma_1^6$ and $\ell=6$. One checks that
\begin{eqnarray*}
\cW(\beta)&=&\{(\id,\ms_1,\id,\ms_1,\id,\ms_1,\id), p^4=(\id,\ms_1,\ms_1,\ms_1,\id,\ms_1,\id), p^5=(\id,\ms_1,\ms_1,\id,\ms_1,\ms_1,\id),\\
&&p^6=(\id,\ms_1,\id,\ms_1,\ms_1,\ms_1,\id), p^7=(\id,\ms_1,\ms_1,\ms_1,\ms_1,\ms_1,\id)\}\subset W^{\ell+1}=W^7.
\end{eqnarray*}
Hence
$
\cW^*(\beta) = \{p\in\cW(\beta):S_p\neq\emptyset\} = \{p^4,p^5,p^6,p^7\}.
$
Denote $\vec{w}^j:=(\ms_1,\ms_1,\ms_1)$ for $4\le j\le 7$. Then
$\modulispace_{\type}(\vec{w}^j,p^j)\isomorphic \field$, $4\le j\le 6$, and
$\modulispace_{\type}(\vec{w}^7,p^7)\isomorphic (\field^{\times})^2$.
\end{enumerate}

In summary,
$
\cW^*=\{(\vec{w}^j,p^j): 1\le j\le 7\},
$
and the cell decomposition of $\modulispace_{\type}$ reads
\[
\modulispace_{\type}
=
\coprod_{j=1}^6 \field
\;\sqcup\;
(\field^{\times})^2.
\]
The ordering of indices is \emph{admissible} (Definition \ref{def:admissible_total_order}). 
This agrees with \cite[\S1.4]{Mel25}.

\item
We now give a concrete description of the variety $\modulispace_{\type}$.

The defining equation of $M_B$ reads
\[
x_1C_1x_1^{-1}x_2C_2x_2^{-1}x_3C_3x_3^{-1}x_4C_4x_4^{-1}=\id,\quad x_i\in G,
\]
together with an action of $G_{\pa}=G\times T^4$ by conjugation:
\[
h\cdot(x_1,x_2,x_3,x_4)=(h_0x_1h_1^{-1},h_0x_2h_2^{-1},h_0x_3h_3^{-1},h_0x_4h_4^{-1}),
\quad h=(h_0,h_1,h_2,h_3,h_4)\in G\times T^4.
\]

Define
\[
\underline{M}_B':=M_B\cap (G^3\times\{I_2\})\subset M_B,
\]
so that its defining equation becomes
\[
x_1C_1x_1^{-1}x_2C_2x_2^{-1}x_3C_3x_3^{-1}C_4=\id,\quad x_1,x_2,x_3\in G.
\]

By Proposition~\ref{prop:associated_fiber_bundles} and Proposition~\ref{prop:reduction_of_principal_bundles}, we have
$
\modulispace_{\type}
\simeq M_B/PG_{\pa}
\simeq \underline{M}_B'/PT_{\pa},
$
where $PT_{\pa}=T^4/\field^{\times}\hookrightarrow PG_{\pa}=(G\times T^4)/\field^{\times}$ via
$
(h_0,h_1,h_2,h_3)\longmapsto (h_0,h_1,h_2,h_3,h_0),
$
and acts on $\underline{M}_B'$ by
$
(h_0,h_1,h_2,h_3)\cdot(x_1,x_2,x_3)
=
(h_0x_1h_1^{-1},h_0x_2h_2^{-1},h_0x_3h_3^{-1}).
$

\medskip

Set $X^i:=x_iC_ix_i^{-1}\in G\cdot C_i\simeq G/T$. Since each $G\cdot C_i$ is affine, define
\[
\underline{M}_B'(\vec{X})
:=
\Big\{(X^1,X^2,X^3)\in \prod_{i=1}^3 G\cdot C_i
\;\Big|\;
X^1X^2X^3C_4=\id
\Big\},
\]
equipped with the conjugation action of $T\ni h_0$. Then we obtain a Cartesian diagram
\[
\begin{tikzcd}[row sep=0.5pc,column sep=2pc]
\underline{M}_B'\arrow[d]\arrow[r,hookrightarrow]\arrow[dr,phantom,"\lrcorner",very near start] & G^3\arrow[d]\\
\underline{M}_B'(\vec{X})\arrow[r,hookrightarrow] & \prod_{i=1}^3 G\cdot C_i\simeq (G/T)^3
\end{tikzcd}.
\]
Hence by base change, $\underline{M}_B'\to \underline{M}_B'(\vec{X})$ is a principal $T^3$-bundle.

Since $T^3\hookrightarrow PT_{\pa}$ via $(h_1,h_2,h_3)\mapsto [\id,h_1,h_2,h_3]$ is a normal subgroup with quotient $PT\simeq T/\field^\times$, Proposition~\ref{prop:subgroup_action_on_a_principal_bundle} yields
\[
\modulispace_{\type}
\simeq \underline{M}_B'/PT_{\pa}
\simeq (\underline{M}_B'/T^3)/PT
\simeq \underline{M}_B'(\vec{X})/PT.
\]

\medskip

For $h_0=\diag(h_{0,1},h_{0,2})\in T$, the conjugation action on $X^i$ is given by
\[
h_0\cdot X^i=
\begin{pmatrix}
X_{11}^i & h_{0,1}X_{12}^ih_{0,2}^{-1}\\
h_{0,2}X_{21}^ih_{0,1}^{-1} & X_{22}^i
\end{pmatrix}.
\]

Define
\[
y_3:=\Tr(X^1X^2),\quad
y_1:=\Tr(X^2X^3),\quad
y_2:=\Tr(C_4^{-1}X^3C_4X^1).
\]
Using the relation $X^1X^2X^3C_4=\id$, these can also be expressed as
\[
y_3=\Tr((X^3C_4)^{-1}),\quad
y_1=\Tr((C_4X^1)^{-1}),\quad
y_2=\Tr(C_4^{-1}(X^2)^{-1}).
\]

A direct computation shows that $\modulispace_{\type}$ is a smooth affine cubic surface defined by
\begin{eqnarray}
&&\sum_{i=1}^3\det(C_i)\,y_i^2 + y_1y_2y_3 
- \det(C_4^{-1})
\sum_{\mathrm{cyclic}}(\Tr C_4\Tr C_1+\Tr C_2^{-1}\Tr C_3^{-1})y_1\\
&&\quad
+ \det(C_4^{-1})
\Big(\sum_{i=1}^4 \Tr C_i\Tr C_i^{-1}
+ \Tr C_1\Tr C_2\Tr C_3\Tr C_4
-4\Big)
=0.\nonumber
\end{eqnarray}

\noindent\textbf{Note.}
When $\det C_i=1$ for all $1\le i\le 4$, this recovers the classical Fricke--Klein cubic
\cite{FK65}, \cite[\S5]{GN05}.
\end{enumerate}
\end{example}

\iffalse

Next, we consider a rank $3$ example.
\begin{example}[$(g,k,n,\type) = (0,3,3,((1^3),(1^3),(1^3)))$]\label{ex:rank_3_over_the_pair_of_pants_cell_decomposition}
Let $\Sigma_{0,3}:=(\Sigma_0,\sigma=\{q_1,q_2,q_3\})$ be the pair of pants, $G=GL_3(\field)$. So, $W=S_3$. 
Let $\type=((1^3),(1^3),(1^3))$. So, $\type$ is very generic and
\[
C_i=\diag(a_{i,1},a_{i,2},a_{i,3})\in T,~~a_{i,1}, a_{i,2}, a_{i,3}: \text{pairwise distinct};~~\prod_{i,j}a_{i,j}=1,~\prod_{i=1}^3a_{i,\psi_i(1)}\neq 1, \forall \psi_i\in W. 
\] 
Clearly,
$\dim\modulispace_{\type} = d_{\type}=n^2(2g-2+k)-\sum_{i=1}^k\sum_j(\mu_j^i)^2 +2 = 9 - 3\times 3 +2 =2$.

By a similar computation, the cell decomposition of $\modulispace_{\type}$ reads:
\begin{equation}
\modulispace_{\type} =\sqcup_{j=1}^9 \modulispace_{\type}(\vec{w}^j,p^j) = \field^{\sqcup 8} \sqcup (\field^{\times})^2.
\end{equation}
We have ordered the indices so that we obtain an admissible total order (Definition \ref{def:admissible_total_order}).
\end{example}

\fi

\section{Dual boundary complexes of character varieties}\label{sec:dual_boundary_complexes_of_character_varieties}

\subsection{Preliminaries on dual boundary complexes}\label{subsec:dual_boundary_complexes}

\subsubsection{Setup}
Let $X$ be a smooth quasi-projective $\field$-variety.

\begin{definition}\label{def:log_compactification_with_very_simple_normal_crossing_divisor}
A \emph{log compactification} of $X$ is a smooth projective variety $\overline{X}$ with simple normal crossing (snc) boundary divisor
$
\partial X=\overline{X}\setminus X.
$
Moreover, we say that $\partial X$ is \emph{very simple normal crossing} if every nonempty finite intersection of irreducible components of $\partial X$ is connected.
\end{definition}

\noindent
By Hironaka's resolution of singularities, a log compactification $(\overline{X},\partial X)$ always exists.
After further blowups if necessary, we may moreover assume that $\partial X$ is very simple normal crossing. We henceforth work in this setting.

\begin{definition}\label{def:dual_boundary_complex}
The \emph{dual boundary complex} $\mathbb{D}\partial X$ is the simplicial complex defined as follows.
\begin{itemize}[wide,labelwidth=!,labelindent=0pt]
\item
The vertices of $\mathbb{D}\partial X$ are in bijection with the irreducible components of $\partial X$.

\item
A collection of $k$ vertices spans a $(k-1)$-simplex if and only if the corresponding irreducible components have nonempty intersection.
\end{itemize}
\end{definition}

\begin{proposition}[{\cite{Dan75}}]\label{prop:invariance_of_dual_boundary_complexes}
The homotopy type of $\mathbb{D}\partial X$ depends only on $X$, i.e. it is independent of the choice of log compactification $(\overline{X},\partial X)$.
\end{proposition}

\begin{example}
For any two smooth quasi-projective varieties $X,Y$, choose log compactifications $\overline{X},\overline{Y}$ separately. Then $\overline{X}\times\overline{Y}$ is a log compactification of $X\times Y$ with boundary
$
\partial(X\times Y)=\partial X\times \overline{Y}\cup_{\partial X\times\partial Y}\overline{X}\times\partial Y.
$
A direct computation shows that
\begin{equation}\label{eqn:dual_boundary_complex_of_a_product}
\mathbb{D}\partial(X\times Y)\sim\mathbb{D}\partial X\star\mathbb{D}\partial Y,
\end{equation}
where $\star$ denotes the \emph{join} of simplicial complexes.

Since $\mathbb{D}\partial\field\sim *$ and $\mathbb{D}\partial \field^*\sim S^0$, it follows that
$
\mathbb{D}\partial(\mathbb{A}^1\times Y)\sim *,
$
so the dual boundary complex of $\mathbb{A}^1\times Y$ is contractible.
As another example,
\[
\mathbb{D}\partial(\field^{\times})^d
\sim
S^0\star\mathbb{D}\partial(\field^{\times})^{d-1}
\sim\cdots\sim S^{d-1}.
\]
\end{example}

The cohomology with rational coefficients of the dual boundary complex is computed as follows.

\begin{proposition}[See e.g. {\cite{Pay13}}]\label{prop:dual_boundary_complexes_vs_weight_filtration}
Let $X$ be a connected smooth quasi-projective complex variety of dimension $d$. Then the reduced rational (co)homology of the dual boundary complex is identified with a graded piece of the weight filtration:
\[
\tilde{H}_{i-1}(\mathbb{D}\partial X,\mathbb{Q})
\isomorphic
\Gr_{2d}^WH^{2d-i}(X(\field),\mathbb{Q}),
\qquad
\tilde{H}^{i-1}(\mathbb{D}\partial X,\mathbb{Q})
\isomorphic
\Gr_0^WH_c^i(X(\field),\mathbb{Q}).
\]
\end{proposition}

\noindent
The second is equivalent to the first by Poincar\'e duality.
More recently, the author proved a motivic refinement which detects the integral cohomology of the dual boundary complex:

\begin{proposition}[{\cite[Prop.~0.2]{Su24}}]\label{prop:dual_boundary_complex_is_motivic}
For any smooth complex quasi-projective variety $X$, one has
\begin{equation}
H^{i-1}(\mathbb{D}\partial X;\mathbb{Z})
\isomorphic
H_{\weight,c}^{i,0}(X;\mathbb{Z}),
\end{equation}
where $H_{\weight,c}^{a,b}(X;\mathbb{Z})$ denotes the \emph{integral singular weight cohomology with compact support} of $X$, defined motivically.
\end{proposition}

\subsubsection{A remove/reduction lemma}
In certain cases, the computation of dual boundary complexes can be simplified considerably:

\begin{lemma}[{\cite[Lem.~2.3]{Sim16}}]\label{lem:dual_boundary_complex_remove_lemma}
Let $X$ be a smooth irreducible quasi-projective $\field$-variety, and let $Z\subset X$ be a smooth irreducible closed subvariety of smaller dimension, with complement $U:=X\setminus Z$.
If $N_Z$ denotes the normal bundle of $Z$ in $X$, then there is a natural homotopy cofiber sequence
\[
\mathbb{D}\partial Z
\sim
\mathbb{D}\partial\mathbb{P}(N_Z)
\longrightarrow
\mathbb{D}\partial X
\sim
\mathbb{D}\partial\mathrm{Bl}_Z(X)
\longrightarrow
\mathbb{D}\partial U.
\]
In particular, if $\mathbb{D}\partial Z\sim *$, then the natural map
$
\mathbb{D}\partial X\longrightarrow\mathbb{D}\partial U
$
is a homotopy equivalence.
\end{lemma}

By iterating this lemma, one obtains the following useful criterion:

\begin{lemma}[{\cite[Prop.~2.6]{Sim16}}]\label{lem:dual_boundary_complex_inductive_remove}
Let $X$ be a smooth irreducible quasi-projective $\field$-variety with a nonempty open subset $U$. Assume that the complement $Z:=X\setminus U$ admits a finite decomposition into locally closed smooth subvarieties $Z_j$ such that:
\begin{enumerate}[wide,labelwidth=!,labelindent=0pt]
\item
$\mathbb{D}\partial Z_j\sim *$ for all $j$;

\item
there exists a total order on the indices such that $\cup_{j\le a}Z_j$ is closed for every $a$.
\end{enumerate}
Then the natural map
$
\mathbb{D}\partial X\longrightarrow\mathbb{D}\partial U
$
is a homotopy equivalence.
\end{lemma}

\noindent
This will be our main tool for computing the dual boundary complex of the character variety $\modulispace_{\type}$.

\subsubsection{A fibration over the dual boundary complex}

For completeness, we review the ``fibration at infinity'' associated to a connected smooth quasi-projective $\field$-variety $X$.
For $X=\modulispace_B$, such a fibration $\ev$ appears in the statement of the geometric P=W conjecture \ref{conj:geometric_P=W}.

Fix a log compactification $(\overline{X},D=\partial X)$ with $D$ very simple normal crossing.
Choose a Riemannian metric on $\overline{X}$.
For $0<\delta\ll 1$, let $N_{\delta}(D)\subset \overline{X}$ denote the $\delta$-neighborhood of $D$.
Let $D_i$, $i\in I=\{1,2,\cdots,m\}$, be the irreducible components of $D$.
For each $i$, let $N_{\delta}(D_i)\subset \overline{X}$ be the $\delta$-neighborhood of $D_i$, which is a tubular neighborhood of $D_i$.
Then $\{N_{\delta}(D_i)\}_{i\in I}$ is an open cover of $N_{\delta}(D)$, and
\begin{equation}\label{eqn:intersection_property_of_open_cover}
\bigcap_{j=1}^rN_{\delta}(D_{i_j})\neq\emptyset
\iff
\bigcap_{j=1}^rD_{i_j}\neq\emptyset.
\end{equation}

Choose a partition of unity $\{\rho_i\}_{i\in I}$ subordinate to this cover.
Thus $\rho_i\in C^\infty(N_{\delta}(D))$ satisfies:
\begin{enumerate}[wide,labelwidth=!,labelindent=0pt]
\item
$0\le \rho_i\le 1$ and $\sum_{i\in I}\rho_i=1$;

\item
$\Supp(\rho_i)$ is a compact subset of $N_{\delta}(D_i)$ for each $i$.
\end{enumerate}
Define
\[
\overline{\ev}=\overline{\ev}[\{\rho_i\}]:
N_{\delta}(D)\longrightarrow \mathbb{D}D,
\qquad
\overline{\ev}(x):=\sum_{i=1}^m\rho_i(x)v_i,
\]
where $v_i$ is the vertex of $\mathbb{D}D$ corresponding to $D_i$.
By (\ref{eqn:intersection_property_of_open_cover}), this map is well-defined.

\medskip

\noindent\textbf{Claim.}
The map $\overline{\ev}$ is independent of the choice of partition of unity, up to homotopy.

\begin{proof}[Proof of Claim]
Let $\{\tilde{\rho}_i\}$ be another such partition of unity.
Then
\[
\rho_i^t:=(1-t)\rho_i+t\tilde{\rho}_i,\qquad t\in[0,1],
\]
defines a continuous family of partitions of unity subordinate to the same cover.
Applying the above construction yields a continuous family of maps
\[
\overline{\ev}_t=\sum_i\rho_i^t\,v_i:
N_{\delta}(D)\longrightarrow\mathbb{D}D,
\]
which gives a homotopy between the two resulting maps.
\end{proof}

\begin{remark}\label{rem:fibration_at_infinity_via_dual_boundary_complex}
By the preceding claim and Proposition \ref{prop:invariance_of_dual_boundary_complexes}, the composite
\begin{equation}
\ev:
N_{\delta}^*(D)=N_{\delta}(D)\setminus D
\hookrightarrow
N_{\delta}(D)
\xrightarrow{\overline{\ev}}
\mathbb{D}\partial X
\end{equation}
is well-defined up to homotopy. We refer to $\ev$ as the \emph{evaluation map}.
\end{remark}

\subsection{Dual boundary complexes of very generic character varieties}\label{subsec:dual_boundary_complexes_of_very_generic_character_varieties}

We now state our main result.

\begin{theorem}\label{thm:weak_geometric_P=W_conjecture_for_very_generic_character_varieties}
The weak geometric P=W conjecture \ref{conj:homotopy_type_conjecture} holds for very generic character varieties:
if $(C_1,\cdots,C_k)\in T^k$ is very generic (Definition \ref{def:generic_monodromy} and Assumption \ref{ass:very_generic_assumption}) of type $\type$, and $\modulispace_{\type}$ is nonempty, then
\[
\mathbb{D}\partial\modulispace_{\type}\sim S^{d_{\type}-1},
\qquad
d_{\type}=\dim\modulispace_{\type}.
\]
\end{theorem}

It relies on the following lemma. 
By \emph{stable isomorphism}, we mean an isomorphism after
multiplying both varieties by the \emph{same} affine space $\mathbb{A}^m$.

\begin{lemma}[{\cite[Cor.~0.3]{Su24}}]\label{lem:dual_boundary_complex_of_a_stably_affine_space_is_contractible}
If $Y$ is a $\field$-variety stably isomorphic to $\mathbb{G}_m^r \times \mathbb{A}^s$ with $s>0$, then $\mathbb{D}\partial Y$ is contractible.
\end{lemma}

\begin{proof}[Proof of Theorem \ref{thm:weak_geometric_P=W_conjecture_for_very_generic_character_varieties}]
By Theorem \ref{thm:cell_decomposition_of_very_generic_character_varieties}, we have a decomposition into locally closed subvarieties
\begin{eqnarray*}
&&\modulispace_{\type}
=
\sqcup_{\vec{w}\in W^{2g}\times \prod_{i=1}^{k-1}W/W(C_i)} \modulispace_{\type}(\vec{w})
=
\sqcup_{\vec{w}\in W^{2g}\times\prod_{i=1}^{k-1}W/W(C_i)}
\sqcup_{p\in\cW^*(\beta(\vec{w}))}
\modulispace_{\type}(\vec{w},p),\\
&&\modulispace_{\type}(\vec{w},p) \times \field^{\dim U}
\isomorphic
(\field^{\times})^{\overline{a}(\vec{w},p)}\times \field^{b(\vec{w},p)},
\end{eqnarray*}
and $\dim\modulispace_{\type}(\vec{w},p)=d_{\type}$ if and only if $(\vec{w},p)=(\vec{w}_{\max},p_{\max})$, equivalently
$
\overline{b}(\vec{w},p)=0.
$
Hence, for every $(\vec{w},p)\neq(\vec{w}_{\max},p_{\max})$, by Lemma \ref{lem:dual_boundary_complex_of_a_stably_affine_space_is_contractible},
$
\mathbb{D}\partial\modulispace_{\type}(\vec{w},p)
\sim *.
$

To apply Lemma \ref{lem:dual_boundary_complex_inductive_remove}, it remains to produce an \emph{admissible total order} (Definition \ref{def:admissible_total_order}) on
\[
\cW^*=
\{(\vec{w},p):
\vec{w}\in W^{2g}\times \prod_{i=1}^{k-1} W/W(C_i),
\;
p\in\cW^*(\beta(\vec{w}))\}.
\]
This is done in Corollary \ref{cor:cell_decomposition_of_character_variety_is_admissible} below.
Applying Lemma \ref{lem:dual_boundary_complex_inductive_remove} with
$X=\modulispace_{\type}$ and $U=\modulispace_{\type}(\vec{w}_{\max},p_{\max})=(\field^{\times})^{d_{\type}}$,
we obtain a homotopy equivalence
\[
\mathbb{D}\partial\modulispace_{\type}
\xrightarrow{\sim}
\mathbb{D}\partial\modulispace_{\type}(\vec{w}_{\max},p_{\max})
\sim
S^{d_{\type}-1}.
\]
This proves the theorem.
\end{proof}

To complete the proof of the theorem, it remains to discuss admissible total orders.

\begin{definition}\label{def:admissible_total_order}
Let $X=\sqcup_{i\in I}Z_i$ be a finite decomposition of a $\field$-variety into locally closed subvarieties.
A total order $\le$ on $I$ is called \emph{admissible} if, for every $a\in I$,
\[
Z_{\le a}:=\bigcup_{i\in I,\; i\le a} Z_i\subset X
\]
is closed.
In particular, if $m\in I$ is the maximal element, then $U:=Z_m\subset X$ is open.
In this case, we call $(X=\sqcup_{i\in I}Z_i,\le)$ an \emph{admissible decomposition}.
\end{definition}

For convenience, we also write
\[
I_{\le a}:=\{i\in I:i\le a\},
\qquad
I_{<a}:=\{i\in I:i<a\},
\qquad
Z_{<a}:=\bigcup_{i\in I_{<a}}Z_i.
\]
Observe that $I_{<i}=I_{\le m_{<i}}$, where $m_{<i}$ is the maximal element of $I_{<i}$.
Hence $Z_{<i}\subset Z_{\le i}\subset X$ are both closed, and therefore
$
Z_i=Z_{\le i}\setminus Z_{<i}
$
is open in $Z_{\le i}$.

\begin{lemma}\label{lem:composition_of_admissible_decompositions}
Let $(X=\sqcup_{i\in I}Z_i,\le)$ be an admissible decomposition.
Suppose that, for each $i\in I$, we are given an admissible decomposition
$
Z_i = \sqcup_{j\in J_i} Z_{i,j}.
$
Let
\[
\tilde{I}:=\{(i,j):i\in I,\; j\in J_i\}\simeq \sqcup_{i\in I} J_i,
\]
and define a total order on $\tilde{I}$ lexicographically by
\[
(i,j)\le(i',j')
\iff
i<i',
\quad\text{or}\quad
i=i'\text{ and }j\le j'.
\]
Then
$
(X = \sqcup_{(i,j)\in\tilde{I}} Z_{i,j},\le)
$
is again an admissible decomposition.
\end{lemma}

\begin{proof}
It is clear that
$
X = \sqcup_{(i,j)\in\tilde{I}} Z_{i,j}
$
is a decomposition into locally closed subvarieties.
It therefore suffices to show that the total order on $\tilde{I}$ is admissible.

For any $(i,j)\in\tilde{I}$, one has
\[
Z_{\le(i,j)}=Z_{<i}\cup (Z_i)_{\le j}\hookrightarrow Z_{<i}\cup Z_i=Z_{\le i}\hookrightarrow X.
\]
By assumption, $Z_{\le i}\subset X$ and $(Z_i)_{\le j}\subset Z_i$ are closed.
Since $Z_i=Z_{\le i}\setminus Z_{<i}$ is open in $Z_{\le i}$, the complement
$
Z_{\le i}\setminus Z_{\le(i,j)}
=
Z_i\setminus (Z_i)_{\le j}
$
is open in $Z_{\le i}$.
Hence $Z_{\le(i,j)}$ is closed in $Z_{\le i}$, and therefore also closed in $X$.
\end{proof}

Next, consider the Bruhat cell decomposition
\[
G = \sqcup_{\dot{w}\in W/W(P)} B\dot{w}P,
\]
where $W(P)$ is the Weyl group of a Levi subgroup of $P$.
Recall that $W/W(P)$ carries the Bruhat partial order: $\dot{\lambda}\le \dot{\mu}$ if and only if
$
B\dot{\lambda}P\subset \overline{B\dot{\mu}P}.
$
Equivalently,
\[
\overline{B\dot{\mu}P}
=
\sqcup_{\dot{\lambda}\le\dot{\mu}} B\dot{\lambda}P.
\]
Hence any total order extending the Bruhat partial order is admissible.
From now on, we fix such an extension once and for all.

Let $\beta\in\Br_n^+$ be an $n$-strand positive braid, with a braid presentation
$
\beta=\sigma_{i_{\ell}}\circ \cdots \circ \sigma_{i_1}.
$
Recall that the braid variety $X(\beta)$ admits a $B$-equivariant decomposition
\[
X(\beta) = \sqcup_{p\in\cW(\beta)} X_p(\beta)
\]
as in (\ref{eqn:cell_decomposition_of_braid_varieties}).

\begin{lemma}\label{lem:cell_decomposition_of_braid_variety_is_admissible}
There exists a natural admissible total order on $\cW(\beta)$.
\end{lemma}

\begin{proof}
For each
$
p=(p_{\ell}=\id,\cdots,p_1,p_0=\id)\in\cW(\beta),
$
define locally closed subvarieties of $X(\beta)$ by
\[
X_{(p_i,\cdots,p_0)}(\beta)
=
\bigcap_{1\le j\le i}\overline{f}_j^{-1}(Bp_jB),
\qquad
\overline{f}_j:X(\beta)\to G,\quad
\vec{\epsilon}\mapsto \mB_{i_j}(\epsilon_j) \cdots \mB_{i_1}(\epsilon_1).
\]
Then
$
X_{(p_i, \cdots ,p_0)}(\beta)
=
\sqcup_{p_{i+1}\in W}
X_{(p_{i+1}, \cdots, p_0)}(\beta).
$

By the Bruhat decomposition, for each $p_{i+1}$ we have
\[
X_{(\le p_{i+1},p_i, \cdots, p_0)}(\beta)
=
\bigcup_{w\le p_{i+1}}
\Bigl(
X_{(p_i, \cdots, p_0)}(\beta)\cap \overline{f}_{i+1}^{-1}(\overline{BwB})
\Bigr),
\]
which is closed in $X_{(p_i, \cdots, p_0)}(\beta)$.
In other words, the chosen total Bruhat order on
\[
W_{(p_i, \cdots, p_0)}
:=
\{p_{i+1}\in W:
X_{(p_{i+1}, \cdots, p_0)}(\beta)\neq\emptyset\}
\subset W
\]
is admissible.

Applying Lemma \ref{lem:composition_of_admissible_decompositions} inductively now yields an admissible total order on $\cW(\beta)$.
\end{proof}

Finally, as promised in the proof of Theorem \ref{thm:weak_geometric_P=W_conjecture_for_very_generic_character_varieties}, we obtain:

\begin{corollary}\label{cor:cell_decomposition_of_character_variety_is_admissible}
There exists a natural admissible total order on the cell decomposition
\[
\modulispace_{\type}
=
\sqcup_{(\vec{w},p)\in \cW^*}\modulispace_{\type}(\vec{w},p),
\]
such that $(\vec{w}_{\max},p_{\max})$ is the maximal index.
\end{corollary}

\begin{proof}
This is essentially a consequence of Lemma \ref{lem:cell_decomposition_of_braid_variety_is_admissible}, and the argument is similar.

First consider the decomposition
$
\modulispace_{\type}
=
\sqcup_{\vec{w}\in W^{2g}\times\prod_{i=1}^{k-1}W/W(C_i)} 
\modulispace_{\type}(\vec{w}).
$
We claim that this decomposition is admissible.
Equivalently, by (\ref{eqn:Bruhat_cell_of_character_variety}) and (\ref{eqn:Equivariant_Bruhat_cell_of_character_variety}), the equivariant decomposition
\[
M_B'
=
\sqcup_{\vec{w}\in W^{2g}\times\prod_{i=1}^{k-1}W/W(C_i)} 
M_B'(\vec{w})
\]
is admissible.

Indeed, arguing exactly as in the proof of Lemma \ref{lem:cell_decomposition_of_braid_variety_is_admissible}, and using Lemma \ref{lem:composition_of_admissible_decompositions}, one obtains an admissible total order on
$
W^{2g}\times \prod_{i=1}^{k-1} W/W(C_i)
$
by composing the chosen total Bruhat orders on
\[
G = \sqcup_{w\in W} BwB
\qquad\text{and}\qquad
G = \sqcup_{\dot{w}_i\in W/W(C_i)} B\dot{w}_iP_i.
\]

It therefore remains, by another application of Lemma \ref{lem:composition_of_admissible_decompositions}, to show that for each fixed $\vec{w}$, the decomposition
$
\modulispace_{\type}(\vec{w})
=
\sqcup_{p\in\cW^*(\beta(\vec{w}))}
\modulispace_{\type}(\vec{w},p)
$
is admissible.

By (\ref{eqn:cell_of_character_variety}), it suffices to show that the equivariant decomposition
$
M_B''(\vec{w})
=
\sqcup_{p\in\cW^*(\beta(\vec{w}))}
M_B''(\vec{w},p)
$
from Proposition \ref{prop:B-equivariant_cell_decomposition_of_character_variety} is admissible.
By (\ref{eqn:M_B''(w)_via_twisted_braid_variety}), this reduces to the admissibility of
\[
\tilde{X}(\beta(\vec{w}))
=
\sqcup_{p\in\cW^*(\beta(\vec{w}))}
\tilde{X}_p(\beta(\vec{w})),
\]
where $\tilde{X}_p(\beta(\vec{w}))$ is defined by (\ref{eqn:restricted_twisted_braid_variety}) and (\ref{eqn:cell_of_restricted_twisted_braid_variety}).
But this follows immediately from Lemma \ref{lem:cell_decomposition_of_braid_variety_is_admissible}.
\end{proof}

\begin{remark}\label{rem:weak_P=W_for_generic_character_varieties}
If $\modulispace_{\type}$ is only generic, then a weaker statement holds:
$
\mathbb{D}\partial\modulispace_{\type}
$
is a rational homology $(d_{\type}-1)$-sphere.
This follows from a completely different argument using the curious hard Lefschetz property \cite[Thm.~1.5.3]{Mel25}; see for example \cite[Rmk.~7.0.7]{MMS22}.
\end{remark}

\subsection{Further directions}\label{subsec:further_directions}

To establish the weak geometric P=W conjecture \ref{conj:homotopy_type_conjecture} for generic, but not necessarily very generic, character varieties $\modulispace_{\type}$, two main problems remain open:

\begin{enumerate}[wide,labelwidth=!,labelindent=0pt]
\item
$H^*(\mathbb{D}\partial\modulispace_{\type};\mathbb{Z})$ is torsion-free;

\item
$\pi_1(\mathbb{D}\partial\modulispace_{\type})=1$.
\end{enumerate}

We now make several comments on these questions.

\begin{lemma}[{\cite[Lem.~3.1]{Su24}}]\label{lem:fundamental_group_for_bulk_vs_dual_boundary_complex}
Let $X$ be a smooth connected affine algebraic variety of dimension $\ge 3$, and let $(\overline{X},D=\overline{X}\setminus X)$ be any log compactification. Then $\mathbb{D}\partial X$ is connected, and there is a natural sequence of surjections
\[
\pi_1(X)\twoheadrightarrow \pi_1(\overline{X})\simeq \pi_1(D)\twoheadrightarrow \pi_1(\mathbb{D}\partial X).
\]
\end{lemma}

Assuming (1), problem (2) appears comparatively manageable.
Indeed, if $\dim\modulispace_{\type}>2$, then by Remark \ref{rem:weak_P=W_for_generic_character_varieties}, assumption (1), and Lemma \ref{lem:fundamental_group_for_bulk_vs_dual_boundary_complex}, problem (2) would follow once one knows that $\pi_1(\modulispace_{\type})$ is abelian.
This is expected to be true.

If $\modulispace_{\type}$ is very generic, then by Theorem \ref{thm:cell_decomposition_of_very_generic_character_varieties}, it contains an open dense algebraic torus $(\field^{\times})^{d_{\type}}$, which induces a surjection
\[
\pi_1((\field^{\times})^{d_{\type}})=\mathbb{Z}^{d_{\type}}\twoheadrightarrow \pi_1(\modulispace_{\type}).
\]
Hence $\pi_1(\modulispace_{\type})$ is indeed abelian in this case.
If $\modulispace_{\type}$ is only generic, we expect that it still contains an open dense algebraic torus, and then the same argument would apply.

There is also another source of evidence in the merely generic case.
Let $\cN(n,c_1)$ denote the moduli space of stable rank $n$ holomorphic bundles of degree $c_1$ on a genus $g$ Riemann surface $\Sigma_g$.
In \cite[Thm.~3.1]{DU95}, the Yang--Mills functional is used as a Morse--Bott function to show that $\pi_1(\cN(n,c_1))$ is abelian whenever $n$, $c_1$ are coprime and $(g,n)\neq (2,2)$.
In this situation, $T^*\cN(n,c_1)$ is open dense in the Dolbeault moduli space $\modulispace_{\Dol}(n,c_1)$ of stable rank $n$ Higgs bundles of degree $c_1$ on $\Sigma_g$.
Therefore one obtains a surjection
\[
\pi_1(\cN(n,c_1))
\simeq
\pi_1(T^*\cN(n,c_1))
\twoheadrightarrow
\pi_1(\modulispace_{\Dol}(n,c_1)),
\]
so $\pi_1(\modulispace_{\Dol}(n,c_1))$ is also abelian.

Under $\NAH$, there is a diffeomorphism
$
\modulispace_{\Dol}(n,c_1)\simeq \modulispace_B(n,c_1),
$
where $\modulispace_B(n,c_1)=\modulispace_{\type}$ with $k=1$, $\type=((n))$, and
$C_1=\exp(-\frac{2\pi i c_1}{n})$.
Hence $\pi_1(\modulispace_B(n,c_1))$ and $\pi_1(\overline{\modulispace}_B(n,c_1))$ are abelian; see also \cite[Prop.~6.30]{FM22}.
We expect that this argument extends more generally.

Turning now to problem (1), Proposition \ref{prop:dual_boundary_complex_is_motivic} shows that it amounts to the torsion-freeness of part of the weight cohomology with compact support
$
H_{\weight,c}^{a,0}(\modulispace_{\type};\mathbb{Z}).
$
We refer to \cite[\S1]{Su24} for a brief review of weight cohomology with compact support.

Examples suggest that a much stronger statement should hold for all generic character varieties:

\begin{enumerate}[wide,labelwidth=!,labelindent=0pt]
\item
the groups
$H_{\weight,c}^{a,b}(\modulispace_{\type};\mathbb{Z})$, $a,b\ge 0$,
are torsion-free;

\item
the integral cohomological descent spectral sequence \cite[Thm.~3]{GS96} for $\modulispace_{\type}$
\[
E_2^{a,b}=H_{\weight,c}^{a,b}(\modulispace_{\type};\mathbb{Z})
\Rightarrow
H_c^{a+b}(\modulispace_{\type}(\mathbb{C});\mathbb{Z})
\]
degenerates at $E_2$.
In particular,
\[
H_c^j(\modulispace_{\type};\mathbb{Z})
=
\bigoplus_{a+b=j}
H_{\weight,c}^{a,b}(\modulispace_{\type};\mathbb{Z})
\]
would be torsion-free.
\end{enumerate}

Since such degeneration fails dramatically for general varieties, it is natural to ask what special feature of character varieties forces it here.
Because of the motivic nature of $H_{\weight,c}^{a,b}(X;A)$, this suggests that the motives with compact support of generic character varieties should admit a particularly simple description, compatible with the HLRV conjecture \cite[Conj.~1.2.1--1.2.2]{HLRV11}.

Finally, we conclude with a remark on a possible motivic approach to generic character varieties.

\begin{remark}\label{rem:towards_the_full_weak_geometric_P=W_conjecture}

Recall that A.~Mellit established the curious hard Lefschetz theorem \cite{Mel25} for all generic character varieties $\modulispace_{\type}$:
\[
\omega^m\cup-:
\Gr_{d_{\type}-2m}^WH_c^i(\modulispace_{\type};\mathbb{C})
\xrightarrow{\simeq}
\Gr_{d_{\type}+2m}^WH_c^{i+2m}(\modulispace_{\type};\mathbb{C}),
\]
where $\omega$ is the canonical holomorphic symplectic form arising from quasi-Hamiltonian geometry.

Adapted to our setting, Mellit's proof may be divided into two main steps.
First, one proves CHL for very generic character varieties using the cell decomposition together with a gluing property for CHL.
Second, one reduces CHL for a generic $\modulispace_{\type}$ to the very generic case by a degeneration argument:

\begin{enumerate}[wide,labelwidth=!,labelindent=0pt]
\item
by varying the monodromy at a virtual puncture, $\modulispace_{\type}$ is realized inside the central fiber of a singular family of character varieties.
After resolving singularities, one obtains a Cartesian diagram (on the left) morally modeled on the Cartesian diagram (on the right)
\[
\begin{tikzcd}[row sep=1pc, column sep=1pc]
X_{\sm}^{=1}\arrow[d,"{\pi_X}"]\arrow[r,hookrightarrow,"{i_X}"] & X_{\sm}^1 \arrow[r,hookrightarrow]\arrow[d,"{\pi_{X}}"] & X_{\sm}\arrow[d,"{\pi_X}"]\\
\modulispace_{\type}=X_{\sing}^{=1}\arrow[r,hookrightarrow,"{i_X}"] & X_{\sing}^1 \arrow[r,hookrightarrow] & X_{\sing}
\end{tikzcd}
\quad\quad
\begin{tikzcd}[row sep=1pc, column sep=1pc]
G/B\arrow[d,"{\pi}"]\arrow[r,hookrightarrow,"{i}"] & \widetilde{N} \arrow[r,hookrightarrow]\arrow[d,"{\pi}"] & \widetilde{G}\arrow[d,"{\pi}"]\\
\{1\}\arrow[r,hookrightarrow,"{i}"] & N \arrow[r,hookrightarrow] & G
\end{tikzcd}
\]
where $\pi:\widetilde{G}\to G=GL_n(\mathbb{C})$ is the Grothendieck--Springer resolution and $N$ is the subvariety of unipotent matrices.
Moreover, $X_{\sm}^1$ behaves like a very generic character variety and satisfies CHL with middle weight $d_{\type}+2\dim(G/B)$;

\item
the BBDG decomposition theorem \cite{BBDG18} for the Springer resolution $\pi:\widetilde{N}\to N$ yields a natural isomorphism of mixed Hodge complexes of sheaves
\[
(R\pi_*\mathbb{C}_{\widetilde{N}})^-
\isomorphic
i_*\mathbb{C}(-\dim G/B)[-2\dim G/B],
\]
where $(-)^-$ denotes the sign isotypic component for the $S_n$-action;

\item
by base change, one obtains a natural isomorphism of mixed Hodge complexes of sheaves
\[
(R\pi_{X*}\mathbb{C}_{X_{\sm}^1})^-
\isomorphic
i_{X*}\mathbb{C}_{X_{\sing}^{=1}}(-\dim G/B)[-2\dim G/B].
\]
Passing to cohomology with compact support gives an isomorphism of mixed Hodge structures
\[
H_c^*(X_{\sm}^1;\mathbb{C})^-
\isomorphic
H_c^*(X_{\sing}^{=1};\mathbb{C})(-\dim G/B)[-2\dim G/B];
\]

\item
this isomorphism is compatible with the Lefschetz operator $\omega\cup-$.
Since CHL holds for $X_{\sm}^1$, it follows that $\modulispace_{\type}=X_{\sing}^{=1}$ also satisfies CHL with middle weight $d_{\type}$.
\end{enumerate}

In view of this discussion, it is natural to ask:

\begin{center}
\emph{Can the degeneration argument proving CHL for character varieties be upgraded to a motivic statement?}
\end{center}

If so, such a motivic refinement could potentially detect the integral cohomology of the dual boundary complex for generic, but not necessarily very generic, character varieties.
At present, this is the main obstruction to proving the full weak geometric P=W conjecture \ref{conj:homotopy_type_conjecture}.

There appear to be at least two major difficulties in pursuing this motivic question:

\begin{enumerate}[wide,labelwidth=!,labelindent=0pt,itemindent=!]
\item
the existence of a motivic refinement of the decomposition theorem for the Springer resolution; for rational coefficients, see \cite{ES22};

\item
the development of the six-functor formalism for integral motivic sheaves; for related work, see \cite{CD19,Spi18}.
\end{enumerate}
\end{remark}

\appendix

\section{Cell decomposition of braid varieties}\label{sec:cell_decomposition_of_braid_varieties}

In this appendix, we prove Proposition \ref{prop:cell_decomposition_of_braid_varieties}. 
We use the notation of Definitions \ref{def:braid_varieties} and \ref{def:walks}.

Let $p\in \cW(\beta)$ and $0\le m\le \ell$. Define a closed subvariety of $\mathbb A^m$ and subsets of $[\ell]$ by
\begin{eqnarray}\label{eqn:inductive_varieties_for_cell_decomposition_of_braid_varieties}
&&X_p^m(\beta):=\bigcap_{1\le j\le m} f_j^{-1}(Bp_jB)\subset \mathbb A^m,\qquad X_p^0(\beta)=\pt,\\
&&U_p^m:=U_p\cap[m],\qquad S_p^m:=S_p\cap[m],\qquad D_p^m:=D_p\cap[m].\nonumber
\end{eqnarray}
Thus
$
[m]=U_p^m\sqcup S_p^m\sqcup D_p^m.
$

\begin{lemma}\label{lem:inductive_cells_for_braid_varieties}
We have $p(X(\beta))\subset \cW(\beta)$. Moreover, for any $p\in \cW(\beta)$ and $1\le m\le \ell$,
\[
X_p^m(\beta)\isomorphic
\begin{cases}
\field_{\epsilon_m'}\times X_p^{m-1}(\beta), 
& \text{if } p_m=\ms_{i_m}p_{m-1}>p_{m-1}\quad\text{(go-up)},\\[4pt]
X_p^{m-1}(\beta), 
& \text{if } p_m=\ms_{i_m}p_{m-1}<p_{m-1}\quad\text{(go-down)},\\[4pt]
\field_{\epsilon_m'}^\times\times X_p^{m-1}(\beta), 
& \text{if } \ms_{i_m}p_{m-1}<p_{m-1}\ \text{and}\ p_m=p_{m-1}\quad\text{(stay)}.
\end{cases}
\]
Here, $\field_{\epsilon_m'}$ (resp. $\field_{\epsilon_m'}^\times$) denotes $\mathbb A^1$ (resp. $\mathbb G_m$) with coordinate $\epsilon_m'$.
\end{lemma}

\begin{proof}
For any $\vec{\epsilon}\in \mathbb A^\ell$, denote $p:=p(\vec{\epsilon})\in W^{\ell+1}$.

\smallskip
\noindent\textbf{Case 1:} $\ms_{i_m}p_{m-1}>p_{m-1}$. 

Then $\ell(\ms_{i_m}p_{m-1})=\ell(p_{m-1})+1$. By Proposition \ref{prop:from_matrices_to_braid_matrix_diagrams},
\[
[\mB_{i_m}(\epsilon_m)\cdots \mB_{i_1}(\epsilon_1)]
=
[\mB_{i_m}(\epsilon_m)]\circ [\mB_{i_{m-1}}(\epsilon_{m-1})\cdots \mB_{i_1}(\epsilon_1)]
\in \BD_n,
\]
hence necessarily
$p_m=\ms_{i_m}p_{m-1}$. This is the go-up case.

Write
\[
\mB_{i_{m-1}}(\epsilon_{m-1})\cdots \mB_{i_1}(\epsilon_1)
=
b_{m-1}\,p_{m-1}\,L_{p_{m-1}}^-,
\qquad b_{m-1}\in B,
\]
using the unique decomposition $Bp_{m-1}B=Bp_{m-1}U_{p_{m-1}}^-$.
Then there exist unique $\epsilon_m'\in \field$ and $b_m\in B$ such that
\begin{equation}\label{eqn:go-up_case}
[\mB_{i_m}(\epsilon_m)]\circ [b_{m-1}]
=
[b_m]\circ [\mB_{i_m}(\epsilon_m')]
\in \underline{\FBD}_n.
\end{equation}
Equivalently,
\[
[\mB_{i_m}(\epsilon_m)\cdots \mB_{i_1}(\epsilon_1)]
=
[b_m]\circ [\mB_{i_m}(\epsilon_m')]\circ [p_{m-1}]\circ [L_{p_{m-1}}^-]
\in \BD_n.
\]
More concretely,
\[
\epsilon_m'
=
(b_{m-1})_{i_m,i_m}^{-1}(b_{m-1})_{i_m+1,i_m+1}\,\epsilon_m
+
(b_{m-1})_{i_m,i_m}^{-1}(b_{m-1})_{i_m,i_m+1}.
\]
Hence
$
X_p^m(\beta)\isomorphic \field_{\epsilon_m'}\times X_p^{m-1}(\beta).
$

\smallskip
\noindent\textbf{Case 2:} $\ms_{i_m}p_{m-1}<p_{m-1}$.

Write
$
p_{m-1}=\ms_{i_m}p_{m-1}',
$
so that $\ell(p_{m-1})=\ell(p_{m-1}')+1$. By Proposition \ref{prop:from_matrices_to_braid_matrix_diagrams}, we may write
\[
\mB_{i_{m-1}}(\epsilon_{m-1})\cdots \mB_{i_1}(\epsilon_1)
=
\mH_{i_m}(c_{m-1})\,\ms_{i_m}\,b_{m-1}'\,p_{m-1}'\,L_{p_{m-1}'}^-,
\]
using the unique decomposition
$
B\ms_{i_m}p_{m-1}'B
=
U_{\ms_{i_m}^{-1}}^-\ms_{i_m}Bp_{m-1}'U_{p_{m-1}'}^-.
$
Therefore
\[
[\mB_{i_m}(\epsilon_m)\cdots \mB_{i_1}(\epsilon_1)]
=
[\ms_{i_m}\mH_{i_m}(\epsilon_m+c_{m-1})\ms_{i_m}b_{m-1}'p_{m-1}'L_{p_{m-1}'}^-] \in \BD_n.
\]

Write
\[
b_{m-1}'=D_{m-1}'u_{m-1}',
\qquad D_{m-1}'\in T,\quad u_{m-1}'\in U,
\]
and define $\epsilon_m'\in \field$ by
\[
\ms_{i_m}\mH_{i_m}(\epsilon_m+c_{m-1})\ms_{i_m}D_{m-1}'
=
D_{m-1}'\ms_{i_m}\mH_{i_m}(\epsilon_m')\ms_{i_m}.
\]
Equivalently, writing $D_{m-1}' = \diag((D_{m-1}')_1,\cdots,(D_{m-1}')_n)$, we have
\[
\epsilon_m'
=
(D_{m-1}')_{i_m+1}^{-1}(\epsilon_m+c_{m-1})(D_{m-1}')_{i_m}.
\]

Now the Bruhat decomposition for $GL_2(\field)$ gives
\begin{equation}\label{eqn:formula_at_a_stay}
\ms_{i_m}\mH_{i_m}(\epsilon_m')\ms_{i_m}
=
\begin{cases}
I_n, & \epsilon_m'=0,\\[4pt]
a_m'\,\ms_{i_m}\,d_m' \in B\ms_{i_m}U_{\ms_{i_m}}^-,
& \epsilon_m'\neq 0,
\end{cases}
\end{equation}
where
\[
a_m'
=
\mK_{i_m}(-\epsilon_m'^{-1})\mK_{i_m+1}(\epsilon_m')\mH_{i_m}(-\epsilon_m'),
\qquad
d_m'=\mH_{i_m}(\epsilon_m'^{-1}),
\]
according to the following computation:
\begin{eqnarray*}
&&\left(\begin{array}{cc}
0 & 1\\
1 & 0
\end{array}\right)
\left(\begin{array}{cc}
1 & \epsilon_m'\\
0 & 1
\end{array}\right)\left(\begin{array}{cc}
0 & 1\\
1 & 0
\end{array}\right)
=\left(\begin{array}{cc}
-(\epsilon_m')^{-1} & 0\\
0 & \epsilon_m'
\end{array}\right)
\left(\begin{array}{cc}
1 & -\epsilon_m'\\
0 & 1
\end{array}\right)\left(\begin{array}{cc}
0 & 1\\
1 & 0
\end{array}\right)
\left(\begin{array}{cc}
1 & (\epsilon_m')^{-1}\\
0 & 1
\end{array}\right).
\end{eqnarray*}

\smallskip
\noindent\emph{Case 2.a:} $\epsilon_m'=0$.

Then
$
[\mB_{i_m}(\epsilon_m)\cdots \mB_{i_1}(\epsilon_1)]
=
[b_{m-1}'p_{m-1}'L_{p_{m-1}'}^-],
$
so
$
p_m=p_{m-1}'=\ms_{i_m}p_{m-1}.
$
This is the go-down case. Hence
$
X_p^m(\beta)\isomorphic X_p^{m-1}(\beta).
$

\smallskip
\noindent\emph{Case 2.b:} $\epsilon_m'\neq 0$.

Then by Proposition \ref{prop:from_matrices_to_braid_matrix_diagrams}, 
\[
[\mB_{i_m}(\epsilon_m)\cdots \mB_{i_1}(\epsilon_1)]
=
[D_{m-1}'a_m'\ms_{i_m}d_m']\circ [u_{m-1}'p_{m-1}'L_{p_{m-1}'}^-]\in \BD_n,
\]
and
\[
\mB_{i_m}(\epsilon_m)\cdots \mB_{i_1}(\epsilon_1)\in B\ms_{i_m}p_{m-1}'B=Bp_{m-1}B.
\]
So
$
p_m=p_{m-1}.
$
This is the stay case. Therefore
$
X_p^m(\beta)\isomorphic \field_{\epsilon_m'}^\times\times X_p^{m-1}(\beta).
$

Finally, if $\vec{\epsilon}\in X(\beta)$, then by definition $p_0=p_\ell=\id$. Combined with the case-by-case analysis above, this shows that $p\in \cW(\beta)$.
\end{proof}

\begin{proof}[Proof of Proposition \ref{prop:cell_decomposition_of_braid_varieties}]
\noindent\textbf{(0)}
The action of $b\in B$ on $\vec{\epsilon}\in X(\beta)$ is uniquely determined by
\[
[\tilde b_m]\circ [\mB_{i_m}(\hat\epsilon_m)]\circ \cdots \circ [\mB_{i_1}(\hat\epsilon_1)]
=
[\mB_{i_m}(\epsilon_m)]\circ \cdots \circ [\mB_{i_1}(\epsilon_1)]\circ [b^{-1}]
\in \underline{\FBD}_n,
\qquad 1\le m\le \ell.
\]
Applying $g_-$ (Definition \ref{def:braid_matrix_diagram_presentations}), we see that each $X_p(\beta)$ is $B$-invariant. The decomposition
\[
X(\beta)=\bigsqcup_{p\in \cW(\beta)} X_p(\beta),\quad X_p(\beta)\isomorphic (\field^\times)^{|S_p|}\times \field^{|U_p|},
\]
follows from Lemma \ref{lem:inductive_cells_for_braid_varieties}.

\smallskip
\noindent\textbf{(1)}
By definition, $\tilde b_m\in B$ and $\hat\epsilon_m\in\field$ are determined inductively by
\begin{equation}\label{eqn:inductive_formula_for_group_action_on_braid_varieties}
[\tilde b_m]\circ[\mB_{i_m}(\hat{\epsilon}_m)]
=
[\mB_{i_m}(\epsilon_m)]\circ[\tilde b_{m-1}]
\in\underline{\FBD}_n.
\end{equation}
By diagram calculus (Lemma \ref{lem:elementary_moves}), we have $\tilde b_m\in TU_{\ms_{i_m}}^+$, and inductively $\tilde b_m\in U_{\ms_{i_m}}^+$ if $b=u\in U$.

\smallskip
\noindent\textbf{(1.a)} Suppose $b=u\in U$, so $\tilde b_m=\tilde u_m\in U_{\ms_{i_m}}^+$. 
Again by diagram calculus, we have
$
U_{\ms_{i_m}}^+X=XU_{\ms_{i_m}}^+
$
for $X\in T$, $X=\ms_{i_m}$, and $X=\mH_{i_m}(c)$.
In particular, we may define $\tilde u_m'\in U$ by
\[
\tilde u_m^{-1}D_{m-1}'\ms_{i_m}\mH_{i_m}(\epsilon_m')\ms_{i_m}
=
D_{m-1}'\ms_{i_m}\mH_{i_m}(\epsilon_m')\ms_{i_m}\tilde u_m'^{-1}.
\]

Now suppose $m\in S_p$, so $p_m=p_{m-1}=\ms_{i_m}p_{m-1}'$ and $\epsilon_m'\neq 0$. Then by the decomposition (\ref{eqn:decomposition_for_U}),
\begin{eqnarray*}
&&[\mB_{i_m}(\hat{\epsilon}_m)\cdots\mB_{i_1}(\hat{\epsilon}_1)]=[\tilde{u}_m^{-1}\mB_{i_m}(\epsilon_m)\cdots\mB_{i_1}(\epsilon_1)u^{-1}]\\
&=&[\tilde{u}_m^{-1}D_{m-1}'\ms_{i_m}\mH_{i_m}(\epsilon_m')\ms_{i_m}u_{m-1}'p_{m-1}'L_{p_{m-1}'}^-u^{-1}]\quad (\text{By the stay case of Lemma \ref{lem:inductive_cells_for_braid_varieties}})\\
&=&[D_{m-1}'\ms_{i_m}\mH_{i_m}(\epsilon_m')\ms_{i_m}\tilde{u}_m'^{-1}u_{m-1}'(p_{m-1}'L_{p_{m-1}'}^+(L_{p_{m-1}'}^-u^{-1})p_{m-1}'^{-1})p_{m-1}'L_{p_{m-1}'}^-(L_{p_{m-1}'}^-u^{-1})]\\
&=&[\hat{D}_{m-1}'\ms_{i_m}\mH_{i_m}(\hat{\epsilon}_m')\ms_{i_m}\hat{u}_{m-1}'p_{m-1}'\hat{L}_{p_{m-1}'}^-]\in\BD_n.
\end{eqnarray*}
Hence $\hat D_{m-1}'=D_{m-1}'$ and $\hat\epsilon_m'=\epsilon_m'$. This proves that if $b=u\in U$, then
\[
\hat\epsilon_m'=\epsilon_m',\qquad \forall m\in S_p.
\]

\smallskip
\noindent\textbf{(1.b)} Suppose $b=t\in T$. By diagram calculus, for any $\lambda=\diag(\lambda_1,\dots,\lambda_n)\in T$, we have
\[
[\mB_j(\epsilon)]\circ[\lambda^{-1}]
=
[(\lambda^{\ms_j})^{-1}]\circ[\mB_j(\lambda_j\lambda_{j+1}^{-1}\epsilon)].
\]
Then by \eqref{eqn:inductive_formula_for_group_action_on_braid_varieties} and induction,
\[
\hat{\epsilon}_m
=
(t^{\ms_{i_{m-1}}\cdots\ms_{i_1}})_{i_m}
(t^{\ms_{i_{m-1}}\cdots\ms_{i_1}})_{i_m+1}^{-1}\epsilon_m,
\qquad
\tilde b_m=(t^{\ms_{i_m}\cdots\ms_{i_1}})^{-1}.
\]

\begin{itemize}[wide,labelwidth=!,labelindent=0pt,itemindent=!]
\item
If $m\in U_p$, then $p_m=\ms_{i_m}p_{m-1}$. 
By the go-up case of Lemma \ref{lem:inductive_cells_for_braid_varieties},
\[
\hat{b}_{m-1}p_{m-1}\hat{L}_{p_{m-1}}^-=\mB_{i_{m-1}}(\hat{\epsilon}_{m-1})\cdots\mB_{i_1}(\hat{\epsilon}_1)=\tilde{b}_{m-1}^{-1}\mB_{i_{m-1}}(\epsilon_{m-1})\cdots\mB_{i_1}(\epsilon_1)t^{-1}=\tilde{b}_{m-1}^{-1}b_{m-1}p_{m-1}L_{p_{m-1}}^-t^{-1}.
\]
Then a simple computation gives $\hat{b}_{m-1}=t^{\ms_{i_{m-1}}\cdots\ms_{i_1}}b_{m-1}(t^{p_{m-1}})^{-1}$, $\hat{L}_{p_{m-1}}^-=tL_{p_{m-1}}^-t^{-1}$. So,
\begin{eqnarray*}
&&[\mB_{i_m}(\hat{\epsilon}_m)]\circ[\hat{b}_{m-1}]\\
&=&[t^{\ms_{i_m}\cdots\ms_{i_1}}\mB_{i_m}(\epsilon_m)(t^{\ms_{i_{m-1}}\cdots \ms_{i_1}})^{-1}]\circ[t^{\ms_{i_{m-1}}\cdots\ms_{i_1}}b_{m-1}(t^{p_{m-1}})^{-1}]\\
&=&[t^{\ms_{i_m}\cdots\ms_{i_1}}]\circ[b_m]\circ[\mB_{i_m}(\epsilon_m')]\circ[(t^{p_{m-1}})^{-1}]\\
&=&[\hat{b}_m]\circ[\mB_{i_m}(\hat{\epsilon}_m')]\in\underline{\FBD}_n.
\end{eqnarray*}
This implies that
\[
\hat{\epsilon}_m'=(t^{p_{m-1}})_{i_m}(t^{p_{m-1}})_{i_m+1}^{-1}\epsilon_m'.
\]

\item
If $m\in S_p$, so $p_m=p_{m-1}=\ms_{i_m}p_{m-1}'$ and $\epsilon_m'\neq 0$. By the stay case of Lemma \ref{lem:inductive_cells_for_braid_varieties},
\begin{eqnarray*}
&&[\mB_{i_m}(\hat{\epsilon}_m)\cdots\mB_{i_1}(\hat{\epsilon}_1)]\\
&=&[\tilde{b}_m^{-1}\mB_{i_m}(\epsilon_m)\cdots\mB_{i_1}(\epsilon_1)t^{-1}]\\
&=&[\tilde{b}_m^{-1}D_{m-1}'\ms_{i_m}\mH_{i_m}(\epsilon_m')\ms_{i_m}u_{m-1}'p_{m-1}'L_{p_{m-1}'}^-t^{-1}]\\
&=&[\hat{D}_{m-1}'\ms_{i_m}\mH_{i_m}(\hat{\epsilon}_m')\ms_{i_m}\hat{u}_{m-1}'p_{m-1}'\hat{L}_{p_{m-1}'}^-]\in\BD_n,
\end{eqnarray*}
and a direct computation gives
\begin{eqnarray*}
&&\hat{D}_{m-1}'=t^{\ms_{i_m}\cdots\ms_{i_1}}D_{m-1}'(t^{-1})^{p_{m-1}'},\quad \hat{\epsilon}_m'=(t^{p_{m-1}})_{i_m}(t^{p_{m-1}})_{i_m+1}^{-1}\epsilon_m',\\
&&\quad \hat{u}_{m-1}'=t^{p_{m-1}'}u_{m-1}'(t^{p_{m-1}'})^{-1},\quad \hat{L}_{p_{m-1}'}^-=tL_{p_{m-1}'}^-t^{-1}.
\end{eqnarray*}
\end{itemize}
This proves part (1).

\smallskip
\noindent\textbf{(2)}
We now prove the formula for the monodromy map $\mumon$. For each $m$, write
\[
\mB_{i_m}(\epsilon_m)\cdots\mB_{i_1}(\epsilon_1)=b_mp_mL_{p_m}^-,
\qquad b_m\in B,
\]
using the decomposition $Bp_mB=Bp_mU_{p_m}^-$. Define
\[
\mumon_m:X_p^m(\beta)\to T,\qquad
(\epsilon_m,\dots,\epsilon_1)\longmapsto D(b_m).
\]
In particular, $\mumon=\mumon_\ell$, while $\mumon_0=I_n$.
The point is that the torus part of the Bruhat decomposition of the partial product can be tracked inductively, and only the stay steps contribute a nontrivial factor. We prove this by induction on $m$.

\smallskip
\noindent\textbf{Case 1:} $m\in U_p$.

Then $p_m=\ms_{i_m}p_{m-1}$. By the go-up case of Lemma \ref{lem:inductive_cells_for_braid_varieties} (see (\ref{eqn:go-up_case})) and elementary moves,
\begin{equation}\label{eqn:inductive_formula_for_mumon_go_up}
\mumon_m(\epsilon_m,\dots,\epsilon_1)
=
D(b_m)
=
D(b_{m-1})^{\ms_{i_m}}
=
\mumon_{m-1}(\epsilon_{m-1},\dots,\epsilon_1)^{\ms_{i_m}}.
\end{equation}

\smallskip
\noindent\textbf{Case 2:} $p_{m-1}=\ms_{i_m}p_{m-1}'$ with $\ell(p_{m-1})=\ell(p_{m-1}')+1$.

As in Case 2 of Lemma \ref{lem:inductive_cells_for_braid_varieties}, write
\[
\mB_{i_{m-1}}(\epsilon_{m-1})\cdots\mB_{i_1}(\epsilon_1)
=
\mH_{i_m}(c_{m-1})\ms_{i_m}b_{m-1}'p_{m-1}'L_{p_{m-1}'}^-,
\]
with
$
b_{m-1}'=D_{m-1}'u_{m-1}'.
$
Then
$
D(b_{m-1})=(D(b_{m-1}'))^{\ms_{i_m}}=(D_{m-1}')^{\ms_{i_m}}.
$

Moreover,
\[
\mB_{i_m}(\epsilon_m)\cdots\mB_{i_1}(\epsilon_1)
=
D_{m-1}'\ms_{i_m}\mH_{i_m}(\epsilon_m')\ms_{i_m}u_{m-1}'p_{m-1}'L_{p_{m-1}'}^-.
\]

\smallskip
\noindent\emph{Case 2.1:} $m\in D_p$.

Then $p_m=p_{m-1}'$ and $\epsilon_m'=0$. Hence
\[
\mB_{i_m}(\epsilon_m)\cdots\mB_{i_1}(\epsilon_1)=b_{m-1}'p_{m-1}'L_{p_{m-1}'}^-,
\]
so $b_m = b_{m-1}'$ and hence
\begin{equation}\label{eqn:inductive_formula_for_mumon_go_down}
\mumon_m(\epsilon_m,\dots,\epsilon_1)
=
D(b_m)
=
(D(b_{m-1}))^{\ms_{i_m}}
=
\mumon_{m-1}(\epsilon_{m-1},\dots,\epsilon_1)^{\ms_{i_m}}.
\end{equation}

\smallskip
\noindent\emph{Case 2.2:} $m\in S_p$.

Then $p_m=p_{m-1}$ and $\epsilon_m'\neq 0$. By the stay case of Lemma \ref{lem:inductive_cells_for_braid_varieties},
\[
\mB_{i_m}(\epsilon_m)\cdots\mB_{i_1}(\epsilon_1)
=
D_{m-1}'a_m'\ms_{i_m}d_m'u_{m-1}'p_{m-1}'L_{p_{m-1}'}^-,
\]
where $a_m',d_m'$ are given in \eqref{eqn:formula_at_a_stay}. Therefore
$
D(b_m)=D(a_m')D_{m-1}'.
$
Equivalently,
\begin{equation}\label{eqn:inductive_formula_for_mumon_stay}
\mumon_m(\epsilon_m,\dots,\epsilon_1)
=
\mK_{i_m}(-\epsilon_m'^{-1})\mK_{i_m+1}(\epsilon_m')\,
\mumon_{m-1}(\epsilon_{m-1},\dots,\epsilon_1)^{\ms_{i_m}}.
\end{equation}

Now the formula \eqref{eqn:formula_for_mumon} follows by induction from
\eqref{eqn:inductive_formula_for_mumon_go_up},
\eqref{eqn:inductive_formula_for_mumon_go_down}, and
\eqref{eqn:inductive_formula_for_mumon_stay}.
This proves part (2).
\end{proof}

We illustrate Proposition \ref{prop:cell_decomposition_of_braid_varieties} with an example refining \cite[Example~5.26]{Mel25}.

\begin{example}
Let $n=2$, $\ell=4$, and $\beta=\sigma_1^4\in \Br_2^+$. Then by Definition \ref{def:braid_varieties},
\begin{eqnarray*}
&&f_1(\epsilon_1)=\mB_1(\epsilon_1)=
\begin{pmatrix}
0 & 1\\
1 & \epsilon_1
\end{pmatrix};\\
&&f_2(\epsilon_2,\epsilon_1)=\mB_1(\epsilon_2)\mB_1(\epsilon_1)=
\begin{pmatrix}
1 & \epsilon_1\\
\epsilon_2 & 1+\epsilon_1\epsilon_2
\end{pmatrix};\\
&&f_3(\epsilon_3,\epsilon_2,\epsilon_1)=\mB_1(\epsilon_3)\mB_1(\epsilon_2)\mB_1(\epsilon_1)=
\begin{pmatrix}
\epsilon_2 & 1+\epsilon_1\epsilon_2\\
1+\epsilon_2\epsilon_3 & \epsilon_1+\epsilon_3+\epsilon_1\epsilon_2\epsilon_3
\end{pmatrix};\\
&&f_4(\vec{\epsilon})=\mB_1(\epsilon_4)\mB_1(\epsilon_3)\mB_1(\epsilon_2)\mB_1(\epsilon_1)=
\begin{pmatrix}
1+\epsilon_2\epsilon_3 & \epsilon_1+\epsilon_3+\epsilon_1\epsilon_2\epsilon_3\\
\epsilon_2+\epsilon_4+\epsilon_2\epsilon_3\epsilon_4 & 1+\epsilon_1\epsilon_2+\epsilon_1\epsilon_4+\epsilon_3\epsilon_4+\epsilon_1\epsilon_2\epsilon_3\epsilon_4
\end{pmatrix}.
\end{eqnarray*}
It follows that
\[
X(\beta)
=
\{\vec{\epsilon}\in \mathbb A^4: f_4(\vec{\epsilon})\in B\}
=
\{\epsilon_2+\epsilon_4+\epsilon_2\epsilon_3\epsilon_4=0\}\subset \mathbb A^4.
\]

The $B$-action on $X(\beta)$ is given as follows. Let $b\in B$ and write
$
\hat{\vec{\epsilon}}=(\hat{\epsilon}_4,\hat{\epsilon}_3,\hat{\epsilon}_2,\hat{\epsilon}_1):=b\cdot \vec{\epsilon}.
$
Write
\[
b^{-1}
=
\begin{pmatrix}
1 & \lambda_{12}\\
0 & 1
\end{pmatrix}
\begin{pmatrix}
\lambda_1 & 0\\
0 & \lambda_2
\end{pmatrix},
\qquad
\lambda_{12}\in \field,\quad \lambda_1,\lambda_2\in \field^\times.
\]
By diagram calculus, as in Example \ref{ex:B-action_on_braid_variety}, we obtain
\[
\hat{\vec{\epsilon}}
=
(\lambda_1\lambda_2^{-1}\epsilon_4,\,
\lambda_2\lambda_1^{-1}\epsilon_3,\,
\lambda_1\lambda_2^{-1}\epsilon_2,\,
\lambda_2\lambda_1^{-1}(\epsilon_1+\lambda_{12})).
\]
See Figure \ref{fig:B-action_on_braid_variety-2} for an illustration.

\begin{figure}[!htbp]
\vspace{-0.5cm}
\begin{center}
\begin{tikzcd}[row sep=0.5pc,column sep=1pc,ampersand replacement=\&]
%%% diagram 1
\begin{tikzpicture}[baseline=-.5ex,scale=0.5]
\begin{scope}

\draw[thin] (-1,0)--(1.5,0) to [out=0,in=180] (3.5,1);
\draw[white,fill=white] (2.5,0.5) circle(0.1);

\draw[thin] (-1,1)--(1.5,1) to [out=0,in=180] (3.5,0) to [out=0,in=180] (5.5,1);
\draw[white,fill=white] (4.5,0.5) circle(0.1);

\draw[thin] (3.5,1) to [out=0,in=180] (5.5,0) to [out=0,in=180] (7.5,1);
\draw[white,fill=white] (6.5,0.5) circle(0.1);

\draw[thin] (5.5,1) to [out=0,in=180] (7.5,0) to [out=0,in=180] (9.5,1);
\draw[white,fill=white] (8.5,0.5) circle(0.1);

\draw[thin] (7.5,1) to [out=0,in=180] (9.5,0);

\draw (-0.7,0) node[left] {\text{\tiny$1$}};
\draw (-0.7,1) node[left] {\text{\tiny$2$}};

\draw[line width=0.5mm] (0.5,0)--(0.5,1);
\node at (0.4,-0.4) {\text{\tiny$\lambda_{12}$}};

\node at (-0.5,-0.1) {\text{\tiny$\ast$}};
\node at (-0.5,0.3) {\text{\tiny$\lambda_1$}};

\node at (-0.5,0.9) {\text{\tiny$\ast$}};
\node at (-0.5,1.3) {\text{\tiny$\lambda_2$}};

\draw[thick,green] (-0.9,-0.7) rectangle (0.9,1.7);
\node[green] at (0,-1.5) {\text{\tiny$[b^{-1}]$}};

\draw[line width=0.5mm] (1.5,0)--(1.5,1);
\draw (1.5,0) node[below] {\text{\tiny$\epsilon_1$}};

\draw[line width=0.5mm] (3.5,0)--(3.5,1);
\draw (3.5,1) node[above] {\text{\tiny$\epsilon_2$}};

\draw[line width=0.5mm] (5.5,0)--(5.5,1);
\draw (5.5,0) node[below] {\text{\tiny$\epsilon_3$}};

\draw[line width=0.5mm] (7.5,0)--(7.5,1);
\draw (7.5,1) node[above] {\text{\tiny$\epsilon_4$}};

\end{scope}
\end{tikzpicture}

\arrow[r,equal,shift left=2]\&

%%% diagram 2
\begin{tikzpicture}[baseline=-.5ex,scale=0.5]
\begin{scope}

\draw[thin] (-1,0)--(-0.5,0) to [out=0,in=180] (1.5,1);
\draw[white,fill=white] (0.5,0.5) circle(0.1);

\draw[thin] (-1,1)--(-0.5,1) to [out=0,in=180] (1.5,0) to [out=0,in=180] (3.5,1);
\draw[white,fill=white] (2.5,0.5) circle(0.1);

\draw[thin] (1.5,1) to [out=0,in=180] (3.5,0) to [out=0,in=180] (5.5,1);
\draw[white,fill=white] (4.5,0.5) circle(0.1);

\draw[thin] (3.5,1) to [out=0,in=180] (5.5,0) to [out=0,in=180] (7.5,1)--(8.5,1);
\draw[white,fill=white] (6.5,0.5) circle(0.1);

\draw[thin] (5.5,1) to [out=0,in=180] (7.5,0)--(8.5,0);

\draw (-0.7,0) node[left] {\text{\tiny$1$}};
\draw (-0.7,1) node[left] {\text{\tiny$2$}};

\draw[line width=0.5mm] (-0.5,0)--(-0.5,1);
\node at (0,-0.8) {\text{\tiny$\hat{\epsilon}_1=\lambda_2\lambda_1^{-1}(\epsilon_1+\lambda_{12})$}};

\draw[line width=0.5mm] (1.5,0)--(1.5,1);
\draw (1.5,1) node[above] {\text{\tiny$\hat{\epsilon}_2=\lambda_1\lambda_2^{-1}\epsilon_2$}};

\draw[line width=0.5mm] (3.5,0)--(3.5,1);
\draw (3.9,0) node[below] {\text{\tiny$\hat{\epsilon}_3=\lambda_2\lambda_1^{-1}\epsilon_3$}};

\draw[line width=0.5mm] (5.5,0)--(5.5,1);
\draw (5.5,1) node[above] {\text{\tiny$\hat{\epsilon}_4=\lambda_1\lambda_2^{-1}\epsilon_4$}};

\node at (8,-0.1) {\text{\tiny$\ast$}};
\node at (8,0.3) {\text{\tiny$\lambda_1$}};

\node at (8,0.9) {\text{\tiny$\ast$}};
\node at (8,1.3) {\text{\tiny$\lambda_2$}};

\draw[thick,green] (7.6,-0.7) rectangle (8.4,1.7);
\node[green] at (8,-1.5) {\text{\tiny$[\tilde{b}_4]$}};

\end{scope}
\end{tikzpicture}
\end{tikzcd}
\end{center}
\vspace{-0.5cm}
\caption{The action of $B\ni b$ on $\vec{\epsilon}\in X(\beta)$ in the case $n=2$, $\ell=4$, and $\beta=\sigma_1^4$. By elementary moves in Figure \ref{fig:elementary_moves}, we obtain
$
[b^{-1}|\mB_1(\epsilon_1)|\mB_1(\epsilon_2)|\mB_1(\epsilon_3)|\mB_1(\epsilon_4)]
=
[\mB_1(\hat{\epsilon}_1)|\mB_1(\hat{\epsilon}_2)|\mB_1(\hat{\epsilon}_3)|\mB_1(\hat{\epsilon}_4)|\tilde{b}_4]
\in \underline{\FBD}_2.
$
So, $b\cdot\vec{\epsilon}=\hat{\vec{\epsilon}}$.}
\label{fig:B-action_on_braid_variety-2}
\end{figure}

By Definition \ref{def:braid_varieties}, the map $\mumon:X(\beta)\to T$ is given by
\[
\mumon(\vec{\epsilon})=\diag(1+\epsilon_2\epsilon_3,\ 1+\epsilon_3\epsilon_4).
\]

By a direct computation, $\cW(\beta)$ consists of two walks:
\begin{enumerate}[wide,labelwidth=!,labelindent=0pt,itemindent=!,label=(\roman*)]
\item
$
p^1=(\id,s_1,\id,s_1,\id)
$
(go-up, go-down, go-up, go-down). So, $S_{p^1} = \emptyset$, $U_{p^1} = \{1,3\}$, and $D_{p^1} = \{2, 4\}$. Moreover,
\[
X_{p^1}(\beta)
=
f_1^{-1}(Bs_1B)\cap f_2^{-1}(B)\cap f_3^{-1}(Bs_1B)\cap f_4^{-1}(B)
=
\{\epsilon_2=0,\ \epsilon_4=0\}
\isomorphic \field^2_{\epsilon_1,\epsilon_3}.
\]
We follow the notation in the proof of Lemma \ref{lem:inductive_cells_for_braid_varieties}.

\begin{itemize}[wide,labelwidth=!,labelindent=0pt,itemindent=!]
\item
Since $s_{i_1}p_0=s_1>p_0=1$, by Case 1 of Lemma \ref{lem:inductive_cells_for_braid_varieties}, we have $b_0=I_2$, $\epsilon_1'=\epsilon_1\in\field$, and $b_1=I_2$.

\item
Since $p_1':=s_{i_2}p_1=1<p_1=s_1$, by Case 2, using
$
\mB_{i_1}(\epsilon_1)=\ms_1\mH_1(\epsilon_1),
$
we obtain $b_1'=\mH_1(\epsilon_1)$ and $c_1=0$, hence $\epsilon_2'=\epsilon_2$.
By Case 2.a, $m=2$ is a go-down position if and only if $\epsilon_2'=0$, which is indeed the case.

\item
Since $s_{i_3}p_2=s_1>p_2=1$, by Case 1 we have
$
b_2=\mB_{i_2}(\epsilon_2)\mB_{i_1}(\epsilon_1)=f_2(0,\epsilon_1)=\mH_1(\epsilon_1).
$
Since
$
[\mB_{i_3}(\epsilon_3)]\circ[b_2]=[\mB_1(\epsilon_3+\epsilon_1)],
$
we obtain $\epsilon_3'=\epsilon_3+\epsilon_1\in\field$.
\end{itemize}

Altogether,
\[
X_{p^1}(\beta)\isomorphic (\field^\times)^{|S_{p^1}|}\times \field^{|U_{p^1}|}
=\field_{\epsilon_1',\epsilon_3'}^2,
\]
since $S_{p^1}=\emptyset$ and $|U_{p^1}|=2$.

Second, $X_{p^1}(\beta)$ is $B$-stable, with inherited action
\[
b\cdot (\epsilon_1',\epsilon_3')
=
(\hat{\epsilon}_1,\hat{\epsilon}_3+\hat{\epsilon}_1)
=
\bigl(\lambda_2\lambda_1^{-1}(\epsilon_1+\lambda_{12}),
\lambda_2\lambda_1^{-1}(\epsilon_3+\epsilon_1+\lambda_{12})\bigr).
\]
This is compatible with Proposition \ref{prop:cell_decomposition_of_braid_varieties}.(1).

Third, on $X_{p^1}(\beta)$ we have $\epsilon_2=\epsilon_4=0$, hence
\[
\mumon(\epsilon_1',\epsilon_3')=I_2.
\]
Since $S_{p^1}=\emptyset$, this agrees with the formula in Proposition \ref{prop:cell_decomposition_of_braid_varieties}.(2).

\item
$
p^2=(\id,s_1,s_1,s_1,\id)
$
(go-up, stay, stay, go-down). So, $S_{p^2} = \{2,3\}$, $U_{p^2} = \{1\}$, and $D_{p^2} = \{4\}$. Moreover,
\begin{eqnarray*}
X_{p^2}(\beta)
&=&
f_1^{-1}(Bs_1B)\cap f_2^{-1}(Bs_1B)\cap f_3^{-1}(Bs_1B)\cap f_4^{-1}(B)\\
&=&
\left\{
\epsilon_2\neq 0,\ 
\tilde{\epsilon}_3:=1+\epsilon_2\epsilon_3\neq 0,\ 
\epsilon_4=-\frac{\epsilon_2}{1+\epsilon_2\epsilon_3}
\right\} \isomorphic \field_{\epsilon_1}^1\times (\field^\times)^2_{\epsilon_2,\tilde{\epsilon}_3}.
\end{eqnarray*}

\begin{itemize}[wide,labelwidth=!,labelindent=0pt,itemindent=!]
\item
As above, $\epsilon_1'=\epsilon_1\in \field$, and $\epsilon_2'=\epsilon_2$.
By Case 2.b, $m=2$ is a stay position if and only if $\epsilon_2'\in \field^\times$, which is indeed the case.

\item
Since $p_2':=s_{i_3}p_2=1<p_2=s_1$, by Case 2, using
\[
\mB_{i_2}(\epsilon_2)\mB_{i_1}(\epsilon_1)
=
\begin{pmatrix}
1 & \epsilon_1\\
\epsilon_2 & 1+\epsilon_1\epsilon_2
\end{pmatrix}
=
\begin{pmatrix}
1 & \epsilon_2^{-1}\\
0 & 1
\end{pmatrix}
\begin{pmatrix}
0 & 1\\
1 & 0
\end{pmatrix}
\begin{pmatrix}
\epsilon_2 & 1+\epsilon_1\epsilon_2\\
0 & -\epsilon_2^{-1}
\end{pmatrix},
\]
we obtain
\[
c_2=\epsilon_2^{-1},
\qquad
b_2'=
\begin{pmatrix}
\epsilon_2 & 1+\epsilon_1\epsilon_2\\
0 & -\epsilon_2^{-1}
\end{pmatrix},\qquad
D_2'=\diag(\epsilon_2,-\epsilon_2^{-1})\in T.
\]
It follows that
\[
\epsilon_3'
=
(D_2')_{i_3+1}^{-1}(\epsilon_3+c_2)(D_2')_{i_3}
=
-\epsilon_2(1+\epsilon_2\epsilon_3)
=
-\epsilon_2\tilde{\epsilon}_3.
\]
By Case 2.b, $m=3$ is a stay position if and only if $\epsilon_3'\in \field^\times$, which is indeed the case.
\end{itemize}

Altogether,
\[
X_{p^2}(\beta)\isomorphic (\field^\times)^{|S_{p^2}|}\times \field^{|U_{p^2}|}
=
\field_{\epsilon_1'}\times (\field^\times)^2_{\epsilon_2',\epsilon_3'}.
\]

Second, $X_{p^2}(\beta)$ is $B$-stable, with inherited action
\[
b\cdot (\epsilon_1',\epsilon_2',\epsilon_3')
=
(\hat{\epsilon}_1,\hat{\epsilon}_2,-\hat{\epsilon}_2(1+\hat{\epsilon}_2\hat{\epsilon}_3))
=
(\lambda_2\lambda_1^{-1}(\epsilon_1'+\lambda_{12}),
\lambda_1\lambda_2^{-1}\epsilon_2',
\lambda_1\lambda_2^{-1}\epsilon_3').
\]
This is compatible with Proposition \ref{prop:cell_decomposition_of_braid_varieties}.(1).

Third, since $\epsilon_2+\epsilon_4+\epsilon_2\epsilon_3\epsilon_4=0$, we have
$
(1+\epsilon_2\epsilon_3)(1+\epsilon_3\epsilon_4)=1.
$
Hence the induced map $\mumon:X_{p^2}(\beta)\to T$ is
\[
\mumon(\epsilon_1',\epsilon_2',\epsilon_3')
=
\diag(1+\epsilon_2\epsilon_3,(1+\epsilon_2\epsilon_3)^{-1})
=
\diag(-\epsilon_2'^{-1}\epsilon_3',-\epsilon_2'\epsilon_3'^{-1}).
\]
Since $S_{p^2}=\{2,3\}$ and
$\ms_{>2}(\beta)=1$, $\ms_{>3}(\beta)=s_1$,
we obtain
\[
\prod_{m\in S_{p^2}}
(\mK_{i_m}(-\epsilon_m'^{-1})\mK_{i_m+1}(\epsilon_m'))^{\ms_{>m}(\beta)}
=
\diag(-\epsilon_2'^{-1},\epsilon_2')
\diag(\epsilon_3',-\epsilon_3'^{-1})
=
\mumon(\epsilon_1',\epsilon_2',\epsilon_3').
\]
This is compatible with Proposition \ref{prop:cell_decomposition_of_braid_varieties}.(2).
\end{enumerate}
\end{example}

\section{Quotients of varieties}\label{sec:quotients_of_varieties}

We collect some results on quotients of varieties. Hopefully, it helps to digest the main body of this article.
Most statements below are standard, so we skip the proofs whenever possible.

\noindent{}Recall our \textbf{Convention} \ref{convention:quotients}: 
a \emph{$\field$-variety} means a \emph{reduced separated scheme of finite type} over $\field$.\\
\noindent{}\textbf{Convention} \setword{{\color{blue}$7$}}{convention:linear_algebraic_groups}: fix $G,H$ as linear algebraic groups over $\field$, unless otherwise stated.

We refer to \cite{Hos16} for the background on various quotients for algebraic group actions. %categorical quotients, good and geometric quotients, and affine GIT quotients.
Occasionally, the notation $X\geoquotient G$ is used for geometric quotient.

\subsection{Principal bundles}

A principal $G$-bundle means so in the \'{e}tale topology unless stated otherwise. 
If necessary, a left $G$-action $G\acts X$ may be viewed as a right action: $X\times G\rightarrow X: (x,g)\mapsto x\cdot g:= g^{-1}\cdot x$.
In what follows, we will freely use both left and right action notations via this identification.

\iffalse

\begin{definition}[Principal bundles]
Let $G$ be a linear algebraic group over $\field$. A \emph{principal $G$-bundle} over $Y$ (in the \emph{\'{e}tale topology}) is a morphism $\pi:P\rightarrow Y$ of $\field$-varieties such that:
\begin{enumerate}[wide,labelwidth=!,labelindent=0pt]
\item
$P$ is a $G$-variety\footnote{If necessary, a left $G$-action may be viewed as a right action: $P\times G\rightarrow P: (p,g)\mapsto p\cdot g:= g^{-1}\cdot p$.} over $\field$. 

\item
$\forall y\in Y(\field)$ $\Rightarrow$ $\exists$ \'{e}tale neighborhood $\iota:U\rightarrow Y$ and a $G$-isomorphism $U\times_Y P\isomorphic U\times G$ over $U$:
\[
\begin{tikzcd}[row sep=1pc]
U\times G\arrow[d]\arrow[r,"{\simeq}"] & U\times_Y P \arrow[r]\arrow[d]\arrow[dr,phantom,"\lrcorner",very near start] & P\arrow[d,"{\pi}"]\\
U\arrow[r,equal] & U\arrow[r,"{\iota}"] & Y 
\end{tikzcd}
\]
\end{enumerate}
\noindent{}\textbf{Note.} For $\tau=$ \emph{Zariski}, \emph{smooth}, \emph{fppf}, or \emph{fpqc}, define \emph{principal $G$-bundles} in $\tau$-topology similarly.
\end{definition}

\fi

\begin{proposition}\label{prop:free_action_and_flat_orbit_map}
If $G\acts X$ \textbf{freely}, and $\pi:X\rightarrow Y$ is a \textbf{flat} orbit map into a \textbf{$\field$-variety}, then:
\begin{enumerate}
\item
The fiber product $X\times_YX$ is $G$-isomorphic to $X\times G$, i.e. we obtain a cartesian diagram:
\begin{equation}\label{eqn:fpqc_trivialization_for_free_action_and_flat_orbit_map}
\begin{tikzcd}[row sep=1pc, column sep=5pc]
X\times G\arrow{r}{(x,g)\mapsto (x,xg)}[swap]{\simeq}\arrow[d,"{p_1}"] & X\times_Y X\arrow[r,"{p_2}"]\arrow[d,"{p_1}"]\arrow[dr,phantom,"\lrcorner",very near start] & X\arrow[d,"{\pi}"]\\
X\arrow[r,equal] & X\arrow[r,"{\pi}"] & Y
\end{tikzcd}
\end{equation}

\item
$\pi:X\rightarrow Y$ is smooth and affine.

\item
$\pi:X\rightarrow Y$ is a principal $G$-bundle (in the \'{e}tale topology).
\end{enumerate}
\end{proposition}

\begin{remark}\label{rem:geometric_quotient_from_free_action_and_flat_orbit_map}
 In the above proposition, $\pi:X\rightarrow Y$ is in fact a geometric quotient. See Proposition \ref{prop:associated_fiber_bundles}.
\end{remark}

\begin{proof}[Proof of Proposition \ref{prop:free_action_and_flat_orbit_map}]
\noindent{}(1). The proof is covered by the three commutative diagrams:
\[
\begin{tikzcd}[row sep=1pc,column sep=1.5pc]
X\times_Y X\arrow[r,"{\tilde{\Delta}_Y}"]\arrow[d]\arrow[dr,phantom,"\lrcorner",very near start] & X\times X\arrow[d,"{\pi\times\pi}"]\\
Y\arrow[r,hookrightarrow,"{\Delta_Y}"] & Y\times Y
\end{tikzcd}
\quad
\begin{tikzcd}[row sep=2pc, column sep=2pc]
X\times G\arrow[drr,hookrightarrow,"{a}"]\arrow[ddr,swap,"{\pi\circ p_1}"]\arrow[dr,dashed,shift right,"{\exists! c}"] & &\\
 & X\times_Y X\arrow[r,hookrightarrow,swap,"{\tilde{\Delta}_Y}"]\arrow[d]\arrow[dr,phantom,"\lrcorner",very near start] & X\times X\arrow[d,"{\pi\times\pi}"]\\
 & Y\arrow[r,hookrightarrow,"{\Delta_Y}"] & Y\times Y
\end{tikzcd}
\quad
\begin{tikzcd}[row sep=1pc,column sep=1.5pc]
X\times G\arrow[r,hookrightarrow,"{c}"]\arrow[d,"{p_1}"] & X\times_YX\arrow[d,"{p_1}"]\arrow[r,"{p_2}"]\arrow[dr,phantom,"\lrcorner",very near start] & X\arrow[d,"{\pi}"]\\
X\arrow[r,equal] & X\arrow[r,"{\pi}"] & Y
\end{tikzcd}
\]
First, consider the left cartesian diagram. As $Y$ is \textbf{separated}, $\Delta_Y$ is a closed embedding, so is $\tilde{\Delta}_Y:X\times_Y X\rightarrow X\times X$ by base change.
A free action means the morphism $a:X\times G\rightarrow X\times X:(x,g)\mapsto (x,xg)$ is a closed embedding. 
Since $\pi$ is $G$-invariant, we have
$
\pi\circ p_1 = \pi\circ p_1\circ a = \pi\circ p_2\circ a:X\times G\to Y.
$

This induces the middle commutative diagram. 
As $a$ and $\tilde{\Delta}_Y$ are closed embeddings, so is $c:X\times G\rightarrow X\times_Y X: (x,g)\mapsto (x,xg)$, by the \textbf{cancellation property} for closed embeddings. 

Now, consider the right commutative diagram. By assumption, each closed fiber of $\pi$ is isomorphic to $G$. 
Then the closed fibers of $p_1:X\times_Y X\to X$ are also isomorphic to $G$, by base change.
Thus, $c$ is an isomorphism on the closed fibers over $X$, hence a dominant morphism.

By base change and composition, $X\times_Y X$ is of finite type over $\field$, and hence $p_1:X\times_Y X\rightarrow X$ is a surjective \textbf{flat} morphism of schemes \textbf{of finite type} over $\field$, with \emph{reduced} closed fibers and \emph{reduced} base $X$. Thus $X\times_Y X$ is reduced by Lemma \ref{lem:criterion_for_reducedness} below. 
Now, $c$ is a closed embedding and a dominant morphism between reduced schemes, hence an isomorphism. This proves (1).

\noindent{}(2). Clearly, $\pi:X\rightarrow Y$ is \textbf{fpqc}.
By \cite[\href{https://stacks.math.columbia.edu/tag/02VL}{Lemma 02VL}]{stacks-project} (resp. \cite[\href{https://stacks.math.columbia.edu/tag/02L5}{Lemma 02L5}]{stacks-project}), a morphism being smooth (resp. affine) is fpqc local on the target. Then by the cartesian diagram in (1), $\pi$ is smooth and affine, as $p_1:X\times G\rightarrow X$ is.

\noindent{}(3). By (2), $\pi:X\rightarrow Y$ is a \emph{smooth covering}. By \cite[\href{https://stacks.math.columbia.edu/tag/055V}{Lemma 055V}]{stacks-project} (\emph{slicing smooth morphisms and refining a smooth covering by an \'{e}tale covering}), there exists a morphism $s:\cV\rightarrow X$ such that the composition $\pi\circ s: \cV\rightarrow X\rightarrow Y$ is an \'{e}tale covering. Now, by (1), the base change of $\pi:X\rightarrow Y$ along $\pi\circ s:\cV\xrightarrow[]{s} X\xrightarrow[]{\pi} Y$ is $G$-isomorphic to $\cV\times G$ over $\cV$.
This gives a local trivialization of $\pi:X\rightarrow Y$ in the \'{e}tale topology. This finishes the proof.
\end{proof}

\begin{lemma}[Criterion for reducedness]\label{lem:criterion_for_reducedness}
If $f:X\rightarrow Y$ is a surjective morphism of schemes \textbf{of finite type} over any field $k$, and all \emph{closed} fibers (i.e. $X_y$, $\forall$ closed point $y\in Y$) of $f$ are \textbf{reduced}, then
\begin{enumerate}[wide,labelwidth=!,labelindent=0pt]
\item
Any fiber $X_y$ of $f$ is reduced (the point $y\in Y$ may not be closed).

\item
If in addition $Y$ is reduced and $f$ is flat, then $X$ is also reduced.
\end{enumerate}
\end{lemma}

\begin{remark}[Jacobson schemes]\label{rem:Jacobson_schemes}
Let $Y$ be a scheme of finite type over any field $k$.
\begin{enumerate}[wide,labelwidth=!,labelindent=0pt]
\item
Recall by \cite[\href{https://stacks.math.columbia.edu/tag/02J1}{Definition 02J1}]{stacks-project} that, a point $y\in Y$ is a \textbf{finite type point} if the canonical morphism $\Spec~k(y)\rightarrow Y$ is of finite type. Equivalently by \cite[\href{https://stacks.math.columbia.edu/tag/01TA}{Lemma 01TA}]{stacks-project}, $y$ is a closed point in some affine open subset $U=\Spec~R$ of $Y$, and the field extension $k\hookrightarrow k(y)$ is finite. 

\item
By \cite[\href{https://stacks.math.columbia.edu/tag/02J6}{Lemma 02J6}]{stacks-project}, $Y$ is \textbf{Jacobson}: the closed points are dense in every closed subset \cite[\href{https://stacks.math.columbia.edu/tag/01P2}{Definition 01P2}]{stacks-project}. Equivalently, every nonempty locally closed subset contains a closed point.

\item
Now by \cite[\href{https://stacks.math.columbia.edu/tag/01TB}{Lemma 01TB}]{stacks-project}, the closed points in $Y$ are precisely the finite type points.
\end{enumerate}
In particular, if $k$ is an algebraically closed field, then the closed points of $Y$ are the $k$-points, and every nonempty locally closed subset contains a closed (i.e. $k$-) point.
\end{remark}

\begin{proof}[Proof of Lemma \ref{lem:criterion_for_reducedness}]

\noindent{}(1). Take any point $y\in Y$, define $Y':=\overline{\{y\}}\hookrightarrow Y$ equipped with the reduced closed subscheme structure. Then $Y'$ is integral and $y$ is a generic point of $Y'$. Let $f':X'=X\times_Y Y'\rightarrow Y'$ be the base change of $f$ along $Y'\hookrightarrow Y$, which is still a surjective morphism of schemes of finite type over $k$. 
By base change and our assumption, all the closed fibers of $f'$ are also reduced. Moreover, $X_y'=(f')^{-1}(y) = X_y$. 

If $X_y'$ is non-reduced, then by \cite[\href{https://stacks.math.columbia.edu/tag/0575}{Lemma 0575}]{stacks-project}, there exists a nonempty open subset $V\subset Y'$ such that, for all $v\in V$ the fiber $X_v'$ is non-reduced. 
However, by Remark \ref{rem:Jacobson_schemes}, $Y$ is Jacobson\footnote{it is this step that the finite type assumption in the lemma becomes essential.}, hence $V$ contains a closed point. This is a contradiction.

\noindent{}(2). This follows from (1) and \cite[Cor.3.3.5]{Gro65} (alternatively, see \cite[Thm.23.9, Cor.]{Mat89}):\\
\emph{Let $f:X\rightarrow Y$ be a flat morphism between two locally Noetherian schemes. 
If $Y$ is reduced at the \emph{points} of $f(X)$, and $f^{-1}(y)$ is a reduced $k(y)$-scheme, $\forall$ point $y\in f(X)$, then $X$ is reduced.}
%\footnote{Note: here a point $y$ of a scheme $Y$ means any morphism $y:\Spec~k'\rightarrow Y$ for any field $k'$. That is, for $Y=\Spec~A$ locally, the points of $Y$ are the prime ideals of $A$.}  
\end{proof}

\begin{remark}\label{rem:fiber_bundle_in_smooth_topology}
As in the proof of Proposition \ref{prop:free_action_and_flat_orbit_map} (3), a similar argument shows the following.\\
If $\pi:X\rightarrow Y$ is a morphism of $\field$-varieties and $F$ is a $\field$-variety, then $\pi$ is a fiber bundle with fiber $F$ in the \'{e}tale topology if and only if $\pi$ is a fiber bundle with fiber $F$ in the smooth topology.
\end{remark}

The flatness condition in Proposition \ref{prop:free_action_and_flat_orbit_map} can be relaxed when the base is normal:
\begin{proposition}\label{prop:free_action_and_orbit_map_with_normal_target}
If $G$ acts \textbf{freely} on a \textbf{pure dimensional} variety $X$ over $\field$ of characteristic $0$, and $\pi:X\rightarrow Y$ is an orbit map, with $Y$ a \textbf{normal} \textbf{pure dimensional} $\field$-variety, then:
\begin{enumerate}[wide,labelwidth=!,labelindent=0pt]
\item
$\pi$ is flat. So Proposition \ref{prop:free_action_and_flat_orbit_map} applies: $\pi$ is smooth, affine, and a principal $G$-bundle, etc.

\item
$X$ is normal.
\end{enumerate}
\end{proposition}

\begin{proof}
\noindent{}(1). By Lemma \ref{lem:variation_of_miracle_flatness}, $\pi$ is flat, so Proposition \ref{prop:free_action_and_flat_orbit_map} applies.

\noindent{}(2). This follows from (1) and \cite[\href{https://stacks.math.columbia.edu/tag/034F}{Lemma 034F}]{stacks-project}: \emph{a morphism being normal is local in the smooth topology}: if $S'\rightarrow S$ is a smooth morphism, then $S$ is normal $\Rightarrow$ so is $S'$. If $S'\rightarrow S$ is smooth surjective, then $S'$ is normal $\Rightarrow$ so is $S$. 
\end{proof}

\begin{lemma}[{\textbf{A variant of miracle flatness}. \cite[Thm.3.3.27]{Sch10}}]\label{lem:variation_of_miracle_flatness}
If $R\rightarrow S$ is a local morphism of Noetherian local rings, $R$ is an excellent normal local domain with perfect residue field, and the closed fiber is regular of dimension $\dim S-\dim R$, 
then $R\rightarrow S$ is faithfully flat. 
\end{lemma}

As a moral converse to Proposition \ref{prop:free_action_and_flat_orbit_map}, we have
\begin{lemma}\label{lem:principal_bundle_in_fpqc_topology}
If $\pi:X\rightarrow Y$ is a principal $G$-bundle in \textbf{fpqc topology}, with $X,Y$ $\field$-varieties, then:
\begin{enumerate}[wide,labelwidth=!,labelindent=0pt]
\item
$\pi$ is smooth and affine.

\item
We have a natural $G$-isomorphism $c:X\times G\xrightarrow[]{\simeq} X\times_Y X: (x,g)\mapsto (x,xg)$, i.e. the cartesian diagram (\ref{eqn:fpqc_trivialization_for_free_action_and_flat_orbit_map}) holds. In addition, the $G$-action on $X$ is free. 

\item
$\pi:X\rightarrow Y$ is an orbit map and a principal $G$-bundle in the \'{e}tale topology.
\end{enumerate}
\end{lemma}

\begin{proof}
\noindent{}(1). The proof is similar to that of Proposition \ref{prop:free_action_and_flat_orbit_map} (2).

\noindent{}(2). As in the proof of Proposition \ref{prop:free_action_and_flat_orbit_map} (1), we obtain a $G$-morphism $c$ in a commutative diagram:
\[
\begin{tikzcd}[row sep=1pc]
X\times G\arrow[r,dashed,rightarrow,"{c}"]\arrow[d,"{p_1}"] & X\times_YX\arrow[d,"{p_1}"]\arrow[r,"{p_2}"]\arrow[dr,phantom,"\lrcorner",very near start] & X\arrow[d,"{\pi}"]\\
X\arrow[r,equal] & X\arrow[r,"{\pi}"] & Y
\end{tikzcd}
\]
and a closed embedding $X\times_Y X\hookrightarrow X\times X$ ($Y$ is \emph{separated}). It remains to show that $c$ is an isomorphism.

By our assumption and base change, $p_1:X\times_Y X\rightarrow X$ is a principal $G$-bundle in the fpqc topology. So, there exists a fpqc covering $\tau:\cV\rightarrow X$ and a $G$-isomorphism $\phi_{\tau}:\cV\times_Y X\xrightarrow[]{\simeq} \cV\times G$ over $\cV$. 
In other words, we obtain the following commutative diagram:
\[
\begin{tikzcd}[row sep=1pc, column sep=4pc]
& \cV\times G\arrow[d,dashed,"{c_{\tau}}"]\arrow[r,"{\tau\times\id_G}"] \arrow[dl,dashed,swap,"{\tilde{c}_{\tau}}"]\arrow[dr,phantom,"\lrcorner",very near start] & X\times G\arrow[d,dashed,"{c}"]\\
\cV\times G\arrow[dr,swap,"{p_1}"] & \cV\times_Y X\arrow{l}{\simeq}[swap]{\phi_{\tau}} \arrow[d,"{p_1}"]\arrow[r]\arrow[dr,phantom,"\lrcorner",very near start] & X\times_Y X\arrow[d,"{p_1}"]\\ 
& \cV\arrow[r,"{\tau}"] & X
\end{tikzcd}
\]
As a $G$-morphism over $\cV$, $\tilde{c}_{\tau}:=\phi_{\tau}\circ c_{\tau}:\cV\times G\rightarrow \cV\times G$ is then a $G$-isomorphism over $\cV$. 
Thus, so is $c_{\tau}$.
Now, by \cite[\href{https://stacks.math.columbia.edu/tag/02L4}{Lemma 02L4}]{stacks-project}, \emph{a morphism being an isomorphism is fpqc local on the target}. It follows that $c:X\times G \rightarrow X\times_Y X$ is a $G$-isomorphism over $X$. We're done.

\noindent{}(3). As $\pi:X\rightarrow Y$ is smooth surjective, by base change, the closed fibers of $\pi:X\rightarrow Y$ are the same as those of $p_1:X\times_Y X\rightarrow X$, which are isomorphic to $G$ by (2). 
Thus, $\pi:X\rightarrow Y$ is an orbit map. 
Now, the result follows from Proposition \ref{prop:free_action_and_flat_orbit_map} (3).
\end{proof}

\subsection{Associated fiber bundles}\label{subsec:associated_fiber_bundles}

A useful reference is \cite{Kra15}.
Recall \textbf{Convention} \ref{convention:linear_algebraic_groups}.
As promised in Remark \ref{rem:geometric_quotient_from_free_action_and_flat_orbit_map}, we complement Proposition \ref{prop:free_action_and_flat_orbit_map}, Proposition \ref{prop:free_action_and_orbit_map_with_normal_target}, and Lemma \ref{lem:principal_bundle_in_fpqc_topology}. 
%In particular, it provides the first place where we get many geometric quotients.

\begin{proposition}[Associated fiber bundles]\label{prop:associated_fiber_bundles}
Let $F$ be an \textbf{affine} $H$-variety over $\field$, and $\fp:P\rightarrow B$ be an (\'{e}tale) principal $H$-bundle\footnote{
By Proposition \ref{prop:free_action_and_flat_orbit_map}, one may equivalently assume that $\fp$ is a flat orbit map of $\field$-varieties for a free action $H\acts P$.}. 
$H$ acts on $P\times F$ diagonally: $h\cdot(a,z):=(ah^{-1},hz)$. Then:
\begin{enumerate}[wide,labelwidth=!,labelindent=0pt]
\item
The action $H\acts P\times F$ admits a \emph{geometric quotient}
$\pi:P\times F \rightarrow P\times^H F=(P\times F)/H$,
with $P\times^H F$ a $\field$-variety.
In particular, $\fp:P\rightarrow B$ is a geometric quotient, with $F=\Spec~k$.

\item
The canonical map $\pi:P\times F \rightarrow P\times^H F$ is a principal $H$-bundle.

\item
The induced map $q:P\times^H F\rightarrow B: [a,z]\mapsto \fp(a)$ is an (\'{e}tale) \emph{fiber bundle with fiber $F$}. 

\item
We have a fiber product diagram
\begin{equation}\label{eqn:fiber_product_for_associated_fiber_bundles}
\begin{tikzcd}[row sep=1pc]
P\times F\arrow[r,"{p_1}"]\arrow[d,"{\pi}"]\arrow[dr,phantom,"\lrcorner",very near start] & P\arrow[d,"{\fp}"]\\
P\times^H F\arrow[r,"{q}"] & B
\end{tikzcd}
\end{equation}
Here, $p_i$ always stands for the projection to the $i$-th factor.
\end{enumerate}
\end{proposition}

\begin{proof}
This is similar to \cite[Lem.3.4.1]{Kra15}.
\end{proof}

\begin{example}[{Homogeneous spaces. \cite[II.Thm.6.8]{Bor12}}]
Let $H \subset G$ be a closed subgroup. Then $q:G\rightarrow G/H=G\geoquotient H$ is a principal $H$-bundle over a smooth quasi-projective $\field$-variety.
\end{example}

\subsection{Functorial properties and applications}\label{subsec:functorial_properties_and_applications}

Now, we give some applications.

\begin{corollary}[base change]\label{cor:base_change_of_principal_bundles}
If $\fp:P\rightarrow B$ is a principal $H$-bundle and $P'\subset P$ is a locally closed $H$-subvariety, then $B':=\fp(P')\subset B$ is locally closed. 
Moreover, endowing $B'$ with the reduced locally closed subvariety structure, we obtain a cartesian square
\[
\begin{tikzcd}[row sep=1pc]
P'\arrow[r,hookrightarrow]\arrow[d,swap,"{\fp'=\fp|_{P'}}"]\arrow[dr,phantom,"\lrcorner",very near start] & P\arrow[d,"{\fp}"]\\
B'\arrow[r,hookrightarrow] & B
\end{tikzcd}
\]
In particular, $\fp':P'\rightarrow B'$ is a principal $H$-bundle, hence a geometric quotient.
\end{corollary}

\begin{proof}
By the composition $P'\hookrightarrow \overline{P}'\hookrightarrow P$, it suffices to consider the case when $P'\subset P$ is an open or a closed $H$-subvariety.
By Proposition \ref{prop:associated_fiber_bundles}, $\fp:P\rightarrow B$ is a geometric quotient. Then as an $H$-invariant subset,
$P'\subset P$ is open (resp. closed) if and only if $B'\subset B$ is open (resp. closed), and $P'=\fp^{-1}(B')$. 
It remains to show that the commutative diagram is cartesian.

The open case is trivial. For the closed case, take $P'':=B'\times_B P$ and $\fp'':=\fp|_{P''}:P''\rightarrow B'$ in the category of schemes. By base change, $\fp''$ is a principal $H$-bundle. In particular, $\fp''$ is flat, and its closed fibers are isomorphic to $H$, hence reduced. Of course, $B'$ is reduced, then so is  $P''$ by Lemma \ref{lem:criterion_for_reducedness}. On the closed subset $P'=\fp^{-1}(B')$, there is a unique reduced closed subscheme structure. Thus, the natural morphism $P'\rightarrow P''$ is an isomorphism.
\end{proof}

\begin{proposition}[Reduction of principal bundles]\label{prop:reduction_of_principal_bundles}
Let $H\subset G$ be a closed subgroup, $X$ be a $G$-variety, and $i:Z\hookrightarrow X$ be a closed $H$-subvariety. So $H\acts G\times Z: h\cdot(g,z)=(gh^{-1},h\cdot z)$.
If:
\begin{enumerate}
\item
$\pi:X\rightarrow B$ be a principal $G$-bundle over a $\field$-variety $B$. 

\item
$\tilde{a}:G\times Z\rightarrow X:(g,z)\mapsto gz$ is a principal $H$-bundle.
\end{enumerate}
Then $\fr:=\pi\circ i: Z\rightarrow B$ is a principal $H$-bundle, and we have a cartesian diagram:
\begin{equation}\label{eqn:reduction_of_principal_bundle}
\begin{tikzcd}[row sep=1pc]
H\times Z\arrow[r,hookrightarrow]\arrow[d,"{a_Z}"]\arrow[dr,phantom,"\lrcorner",very near start] & G\times Z\arrow[r,"{p_2}"]\arrow[d,"{\tilde{a}}"]\arrow[dr,phantom,"\lrcorner",very near start] & Z\arrow[d,"{\fr}"]\\
Z\arrow[r,hookrightarrow,"{i}"] & X\arrow[r,"{\pi}"] & B
\end{tikzcd}
\quad (\Rightarrow Z/H=Z\geoquotient H \xrightarrow[]{\simeq} X/G=X\geoquotient G \xrightarrow[]{\simeq} B).
\end{equation}
\end{proposition}

\begin{proof}
Clearly, $\fr:Z\rightarrow B$ is an $H$-invariant morphism. We have the following cartesian diagram:
\[
\begin{tikzcd}[row sep=1.2pc, column sep=4pc]
G\times Z\arrow[d,hookrightarrow,"{\id_G\times i}"]\arrow[rr,"{p_2}"]\arrow[drr,phantom,"\lrcorner",very near start] & & Z\arrow[d,hookrightarrow,"{i}"]\\
G\times X\arrow[d,"{a}"]\arrow{r}{(g,x)\mapsto (gx,x)}[swap]{\simeq} & X\times_B X\arrow[r,"{p_2}"]\arrow[d,"{p_1}"]\arrow[dr,phantom,"\lrcorner",very near start] & X\arrow[d,"{\pi}"]\\
X\arrow[r,equal] & X\arrow[r,"{\pi}"] & B
\end{tikzcd}
\]
Here, $p_i$ denotes the projection to the $i$-th factor. $a:G\times X\rightarrow X: (g,x)\mapsto gx$ is the action map. So, $\tilde{a}=a\circ(\id_G\times i)$. The isomorphism $G\times X\xrightarrow[]{\simeq} X\times_B X$ follows from Lemma \ref{lem:principal_bundle_in_fpqc_topology}. Then by assumption and base change, $\fr:Z\rightarrow B$ is a principal $H$-bundle in the smooth topology, hence in the \'{e}tale topology by Lemma \ref{lem:principal_bundle_in_fpqc_topology}. It remains to show (\ref{eqn:reduction_of_principal_bundle}): the right square is cartesian by above; by Lemma \ref{lem:principal_bundle_in_fpqc_topology}, so is the outer square. The rest is due to Proposition \ref{prop:associated_fiber_bundles}.
\end{proof}

\begin{remark}\label{rem:reduction_of_principal_bundles}
In the above proposition, by Propositions \ref{prop:free_action_and_flat_orbit_map}, \ref{prop:free_action_and_orbit_map_with_normal_target}, the conditions (1)-(2) can be relaxed to: 
\begin{itemize}[wide,labelwidth=!,labelindent=0pt]
\item
the $G$-action on $X$ is free;
 
\item
$\pi:X\rightarrow B$ and $\tilde{a}:G\times Z\rightarrow X$ are either \emph{flat} orbit maps, or orbit maps between \emph{pure dimensional} varieties with $B$ \emph{normal}.
\end{itemize}
\end{remark}

\begin{proposition}[Quotient of a principal bundle by a subgroup]\label{prop:subgroup_action_on_a_principal_bundle}
Let $H\subset G$ be a closed subgroup such that $G/H$ is \textbf{affine}. Let $\fp:P\rightarrow B$ be a principal $G$-bundle, then:
\begin{enumerate}
\item
There exists a geometric quotient $\fp_H:P\rightarrow P/H$, and $\fp_H$ is a principal $H$-bundle.

\item
The natural map $q_H:P/H\rightarrow B$ is a fiber bundle with fiber $G/H$ in the \'{e}tale topology.

\item
If in addition, $H\triangleleft G$ is a normal subgroup, then $q_H:P/H\rightarrow B$ is a principal $G/H$-bundle.
\end{enumerate}
\end{proposition}

\begin{proof}
Let $F:=G/H$ and $X:=P\times F$. So, $G$ acts diagonally on $X$: $g\cdot (p,zH):= (pg^{-1},gzH)$.
As $F=G/H$ is \textbf{affine}, by Proposition \ref{prop:associated_fiber_bundles}, we conclude that $G\acts X$ admits a geometric quotient
$\pi:X=P\times F \rightarrow Y:= P\times^G F$,
which is a principal $G$-bundle, and the induced map $q_H=q:Y=P\times^G F \rightarrow B$ is a fiber bundle with fiber $F$ in the \'{e}tale topology. Moreover, we obtain the following cartesian diagram
\[
\begin{tikzcd}[row sep=1.2pc, column sep=6pc]
P\times G\arrow{r}{\id_P\times\overline{\pi}}[swap]{(p,g)\mapsto (p,gH)}\arrow{d}{(p,g)\mapsto pg}[swap]{a}\arrow[dr,phantom,"\lrcorner",very near start] & X=P\times G/H\arrow[r,"{p_1}"]\arrow[d,"{\pi}"]\arrow[dr,phantom,"\lrcorner",very near start] & P\arrow[d,"{\fp}"]\\
P=P\times^G G\arrow{r}{\fp_H}[swap]{p\mapsto [p,H]} & Y=P\times^GG/H\arrow{r}{q_H}[swap]{[p,gH]\mapsto \fp(p)} & B
\end{tikzcd}
\]

\noindent{}(1). By the left cartesian square, $\fp_H:P\rightarrow Y=P\times^G G/H$ is a principal $H$-bundle.
By Proposition \ref{prop:associated_fiber_bundles} (1), $\fp_H$ is also a geometric quotient. So we can write $P/H=Y=P\times^G G/H$.

\noindent{}(2).
By the right cartesian square above, $q_H:P/H=Y\rightarrow B$ is a fiber bundle with fiber $G/H$ in the smooth topology, hence in the \'{e}tale topology by Remark \ref{rem:fiber_bundle_in_smooth_topology}. 

\noindent{}(3).
The action map $a:P\times G\rightarrow P$ induces an action
$P/H\times G/H\rightarrow P/H: (pH,gH)\mapsto pHgH=pgH$.
Then by the right cartesian square above $q_H:Y=P/H\rightarrow B$ is a principal $G/H$-bundle in the smooth topology, hence in the \'{e}tale topology by Lemma \ref{lem:principal_bundle_in_fpqc_topology}.
\end{proof}

\begin{proposition}\label{prop:principal_bundle_for_unipotent_algebraic_group_over_affine_variety_is_trivial}
Let $G$ be a \textbf{unipotent} algebraic group and $B$ be an \textbf{affine} variety over $\field$. Then any principal $G$-bundle $\fp:P\rightarrow B$ is trivial.
\end{proposition}
\noindent{}\textbf{Note.} $G$ is connected: $\pi_0(G)=G/G^0$ is a finite unipotent algebraic group, which must be trivial.

\begin{proof}
First, assume $G$ is commutative. 
Hence $G=G_a^m$ for some $m\geq 0$, as its (abelian) Lie algebra is a $\field$-vector space with trivial Lie bracket.
Now, the principal $G$-bundles over $B$ are classified by $H_{\et}^1(B,G_a^m)$ (see \cite[\href{https://stacks.math.columbia.edu/tag/03F7}{Lemma 03F7}]{stacks-project}). By \cite[\href{https://stacks.math.columbia.edu/tag/03DW}{Proposition 03DW}]{stacks-project} and the affine vanishing property, $H_{\et}^1(B,G_a^m)\isomorphic H^1(B,\cO_B^m)=0$. Thus, $\fp$ is a trivial principal $G$-bundle.

In general, we prove by induction on $\dim G$. Say, $n:=\dim G>0$. 
Let $H:=Z(G)$ be the center of $G$. Then $H$ is a commutative unipotent algebraic group of positive dimension \cite[17.5.~Cor., 17.4.~Prop.]{Hum12}, and $\overline{G}:=G/H$ is a unipotent algebraic group of dimension $<n$.
By Proposition \ref{prop:subgroup_action_on_a_principal_bundle}, there exists a $\field$-variety $\overline{P}:=P/H$ such that,
$\fp_H:P\rightarrow \overline{P}$ is a principal $H$-bundle, and $q_H:\overline{P}\rightarrow B$ is a principal $\overline{G}$-bundle.
By induction, $q_H$ is trivial, hence admits a section $s:B\rightarrow \overline{P}$. We then obtain the following commutative diagram in which the square is cartesian:
\[
\begin{tikzcd}[row sep=1pc,column sep=3pc]
s^{-1}P\arrow[r,"{\tilde{s}}"]\arrow[d,"{\tilde{\fp}_H}"]\arrow[dr,phantom,"\lrcorner",very near start] & P\arrow[d,"{\fp_H}"]\arrow[dd,bend left=60,"{\fp}"]\\
B\arrow[r,"{s}"]\arrow[dr,equal]\arrow[u,dashed,bend left=30,"{s'}"] & \overline{P}\arrow[d,"{q_H}"]\\
& B
\end{tikzcd}
\]
By base change, $\tilde{\fp}_H:s^{-1}P\rightarrow B$ is a principal $H$-bundle. As $H$ is a commutative unipotent algebraic group, $\tilde{\fp}_H$ is trivial by above, hence admits a section $s':B\rightarrow s^{-1}P$. Then $\tilde{s}\circ s':B\rightarrow P$ defines a section of $\fp:P\rightarrow B$. This implies that $\fp$ is a trivial principal $G$-bundle.
\end{proof}

\section{Stable isomorphisms and stable cancellations}
\label{sec:stable_isomorhisms_and_stable_cancellations}

Throughout this section, all varieties and algebraic groups are over $\mathbb{C}$.
Recall that stable isomorphism means an isomorphism after multiplying both varieties by the \emph{same} affine space $\mathbb{A}^m$.

\subsection{Torus quotients ans stable isomorphisms}

The proof of Theorem \ref{thm:cell_decomposition_of_very_generic_character_varieties} uses a property concerning torus quotients as follows.

\begin{proposition}\label{prop:torus_quotient_vs_stable_isomorphism}
Let
$
 X\cong(\mathbb{G}_m)^a\times\mathbb{A}^b,
$
and let $B$ be a connected solvable algebraic group with unipotent radical
$U$ and maximal torus $T$.  Suppose that $Y$ is an affine variety with a
free $T$-action and that there is a $B$-equivariant isomorphism
$X\cong U\times Y$, where
\[
 (tu)\cdot(v,y)=\bigl(tuvt^{-1},\,t\cdot y\bigr)
 \qquad(t\in T,\ u,v\in U).
\]
Then there is a noncanonical isomorphism
\[
X/T\cong U\times(Y/T).
\]
This is an isomorphism over $Y/T$ for the morphism induced by
$U\times Y\to Y$.
\end{proposition}

The main point is that the hypotheses force the $T$-action on $Y$ to
admit a global section.  We first establish the needed fact about
invertible functions.

\begin{lemma}\label{lem:unit-weights}
Let an algebraic torus $T$ act freely on
$W=(\mathbb{G}_m)^a\times\mathbb{A}^b$, and write
\[
 \mathcal{O}(W)=\mathbb{C}[z_1^{\pm1},\ldots,z_a^{\pm1},x_1,\ldots,x_b].
\]
Then each $z_i$ is a $T$-semi-invariant:
\[
 z_i(t\cdot w)=\chi_i(t)z_i(w)
 \qquad\text{for some }\chi_i\in X^*(T).
\]
The characters $\chi_1,\ldots,\chi_a$ generate $X^*(T)$ as an
abelian group.
\end{lemma}

\begin{proof}
Let
$\delta:\mathcal{O}(W)\to\mathcal{O}(T)\otimes_\mathbb{C}\mathcal{O}(W)$ be the coaction, with
$\delta(f)(t,w)=f(t\cdot w)$.
The tensor product is a Laurent polynomial ring with additional
polynomial variables, so its units are precisely
\[
 c\,\chi\otimes z_1^{m_1}\cdots z_a^{m_a},
 \qquad c\in\mathbb{C}^\times,\quad \chi\in X^*(T),\quad m_i\in\mathbb Z.
\]
Since $\delta(z_i)$ is a unit, evaluating at the identity of $T$ forces
$c=1$ and $z_1^{m_1}\cdots z_a^{m_a}=z_i$.  Consequently
$\delta(z_i)=\chi_i\otimes z_i$.

Consider the homomorphism
\[
 \lambda:T\longrightarrow(\mathbb{G}_m)^a,\qquad
 t\longmapsto\bigl(\chi_1(t),\ldots,\chi_a(t)\bigr).
\]
Suppose that $K=\ker\lambda$ is nontrivial.  As a nontrivial closed
subgroup of a complex torus, $K$ is diagonalizable and contains a
subgroup $\mu_p$ for some prime $p$.  The group $K$ preserves every
fiber of the standard projection
\[
 z:W\longrightarrow(\mathbb{G}_m)^a,\qquad
 w\longmapsto\bigl(z_1(w),\ldots,z_a(w)\bigr).
\]
Each fiber is $\mathbb{A}^b$, and the induced $\mu_p$-action on it is free.

Such a free action is impossible.  Indeed, the finite-group quotient
$Q=\mathbb{A}^b/\mu_p$ is an affine complex variety, and the quotient map is a
finite étale cover of degree $p$.  The compactly supported Euler
characteristic is integer-valued on complex algebraic varieties and is
multiplicative for finite étale covers.  It would follow that
$
 1=\chi_c(\mathbb{A}^b)=p\,\chi_c(Q),
$
a contradiction.  Here multiplicativity can also be seen by
triangulating the quotient and counting the $p$ lifts of every cell.

Thus $\ker\lambda=1$.  A homomorphism of complex tori with trivial
kernel is a closed embedding.  By the character-lattice description of
torus homomorphisms, its dual map
\[
 \lambda^*:\mathbb Z^a\longrightarrow X^*(T),\qquad
 (m_1,\ldots,m_a)\longmapsto\chi_1^{m_1}\cdots\chi_a^{m_a}
\]
is surjective, as claimed.
\end{proof}

\begin{proof}[Proof of Proposition~\ref{prop:torus_quotient_vs_stable_isomorphism}]
Identify $X$ with $U\times Y$ through the given isomorphism.
The restricted $T$-action is
$t\cdot(v,y) = (tvt^{-1}, t\cdot y)$.
It is free because the action on $Y$ is free.  Apply
Lemma~\ref{lem:unit-weights} to the given presentation of $X$ as
$(\mathbb{G}_m)^a\times\mathbb{A}^b$.
Put $r=\dim T$, choose a basis $\eta_1,\ldots,\eta_r$ of $X^*(T)$,
and choose vectors $m^{(j)}\in\mathbb Z^a$ with
$\lambda^*(m^{(j)})=\eta_j$.  The Laurent monomials
\[
 f_j=z_1^{m^{(j)}_1}\cdots z_a^{m^{(j)}_a}
 \qquad(1\leq j\leq r)
\]
are units on $X$ satisfying
$f_j(t\cdot x)=\eta_j(t)f_j(x)$.
Over $\mathbb{C}$, the variety underlying a connected unipotent group is an
affine space.  If $n=\dim U$, it follows that
\[
 \mathcal{O}(U\times Y)\cong\mathcal{O}(Y)[q_1,\ldots,q_n],
 \qquad
 \mathcal{O}(U\times Y)^\times=\mathcal{O}(Y)^\times.
\]
The equality of unit groups uses that $\mathcal{O}(Y)$ is reduced; indeed $Y$
is integral because $U\times Y\cong X$ is integral.  Thus every $f_j$
is pulled back from an invertible function on $Y$, denoted by the same
symbol.

Using the isomorphism
$T\to(\mathbb{G}_m)^r$, $t\mapsto(\eta_1(t),\ldots,\eta_r(t))$, these functions
define a regular map $h:Y\to T$ characterized by
\[
 \eta_j(h(y))=f_j(y).
\]
Semi-invariance gives
\begin{equation}\label{eq:h-equivariance}
 h(t\cdot y)=t\,h(y).
\end{equation}
Let $Z=h^{-1}(1)$.  This is a closed affine subvariety of $Y$, and
the action map
\[
 T\times Z\longrightarrow Y,\qquad (q,z)\longmapsto q\cdot z
\]
is an isomorphism with inverse
$
 y\longmapsto\bigl(h(y),\,h(y)^{-1}\cdot y\bigr).
$
In these coordinates $T$ acts by translation on the first factor and
trivially on $Z$.  Therefore $Y/T\cong Z$, both geometrically and
categorically.

Now define
\[
 F:U\times Y\longrightarrow U\times T\times Z:
 (v,y)\longmapsto
 \bigl(h(y)^{-1}v\,h(y),\,h(y),\,h(y)^{-1}\cdot y\bigr).
\]
This is a regular isomorphism; its inverse is
$
 (w,q,z)\longmapsto\bigl(qwq^{-1},\,q\cdot z\bigr).
$
By \eqref{eq:h-equivariance}, $F$ transforms the $T$-action into
\[
 t\cdot(w,q,z)=(w,tq,z).
\]
Consequently the projection $U\times T\times Z\to U\times Z$ is its
affine geometric and categorical quotient.  We obtain
\[
 X/T\cong U\times Z\cong U\times(Y/T).
\]
Explicitly, the quotient isomorphism and its inverse are
\[
 {}[(v,y)]\longmapsto
 \bigl(h(y)^{-1}v\,h(y),\,h(y)^{-1}\cdot y\bigr);\qquad
 (w,z)\longmapsto[(w,z)].
\]
The second component is the orbit of $y$ represented in $Z$, so this
isomorphism respects the morphisms to $Y/T$.
\end{proof}

\begin{corollary}\label{cor:standard-quotient}
Under the hypotheses of Proposition~\ref{prop:torus_quotient_vs_stable_isomorphism}, put $r=\dim T$.
Then $r\leq a$, and
\[
 U\times(Y/T)\cong X/T\cong(\mathbb{G}_m)^{a-r}\times\mathbb{A}^b.
\]
\end{corollary}

\begin{proof}
The surjection $\lambda^*:\mathbb Z^a\to X^*(T)$ splits, so its
kernel has rank $a-r$.  The chosen lifts $m^{(1)},\ldots,m^{(r)}$,
together with a basis of $\ker\lambda^*$, form a basis of
$\mathbb Z^a$.  If $g_1,\ldots,g_{a-r}$ are the Laurent monomials
associated with that kernel basis, then
\[
 \mathcal{O}(X)=
 \mathbb{C}[f_1^{\pm1},\ldots,f_r^{\pm1},
     g_1^{\pm1},\ldots,g_{a-r}^{\pm1},x_1,\ldots,x_b].
\]
The equivariant map $h_X=h\circ\operatorname{pr}_Y:X\to T$
has coordinates $f_1,\ldots,f_r$.  Its fiber over $1$ meets every
$T$-orbit exactly once, by the same trivialization used in the proof
of the theorem.  Hence
\[
 X/T\cong h_X^{-1}(1)
 \cong\Spec\mathbb{C}[g_1^{\pm1},\ldots,g_{a-r}^{\pm1},x_1,\ldots,x_b]
 \cong(\mathbb{G}_m)^{a-r}\times\mathbb{A}^b.
\]
No linearity assumption on the $T$-action on the $x_j$ is needed.
\end{proof}

\subsection{Stable cancellation for tori and for tori times an affine line}
\label{subsec:small-affine-dimension}

Example \ref{ex:rank_2_over_four-punctured_two-sphere_cell_decomposition} uses the stable cancellation properties of $\mathbb{G}_m^a$ and $\mathbb{G}_m^a\times \mathbb{A}^1$ as follows.

\begin{proposition}[Stable cancellation for algebraic tori]
\label{prop:torus-cancellation}
Let $W$ be an affine complex variety and let $a,m\geq0$.  If
\[
 W\times\mathbb{A}^m\cong(\mathbb{G}_m)^a\times\mathbb{A}^m,
\]
then $W\cong(\mathbb{G}_m)^a$.
\end{proposition}
\begin{proof}
Put $R=\cO(W)$ and
$S=\mathbb{C}[t_1^{\pm1},\ldots,t_a^{\pm1}]$.
The hypothesis gives an isomorphism
\[
 R[u_1,\ldots,u_m]\cong S[x_1,\ldots,x_m].
\]
In particular, $R$ is a domain.  Use this isomorphism to regard $R$
as a subring of the right-hand side.  Since every unit of
$R[\mathbf u]$ belongs to $R$, the elements $t_i^{\pm1}$ belong to
$R$.  Hence
\[
 S\subseteq R\subseteq S[x_1,\ldots,x_m].
\]
Comparing transcendence degrees over $\mathbb{C}$ yields
\[
\trdeg_{{\Frac S}}\Frac R = \trdeg_{\Frac S} \Frac S[x_1,\ldots,x_m] - \trdeg_{\Frac R} \Frac S[x_1,\ldots,x_m] = m - m = 0.
\]
Thus every element of $R$ is algebraic over $K=\Frac S$.
But every nonconstant polynomial in $K[x_1,\ldots,x_m]$ is
transcendental over $K$: its successive powers have distinct total
degrees.  It follows that
\[
 R\subseteq K\cap S[x_1,\ldots,x_m]=S.
\]
Together with $S\subseteq R$, this proves $R=S$, as required.
\end{proof}

For one affine coordinate, we recall a result of Abhyankar--Heinzer--Eakin~\cite{AHE72}. 
For any ring $R$, let $R^{[j]}$ denote a polynomial algebra in $j$ variables over $R$.
Let $\cong_R$ denote $R$-algebra isomorphism.
\begin{lemma}[{\cite[P. 329, (4.5)]{AHE72}}]
\label{lem:cancellation_for_line_over_UFD}
Let $D$ be a UFD, then for any $D$-algebra $E$ and any $m\geq 1$, we have
\begin{equation}%\label{eq:ahe-relative}
 E^{[m]}\cong_D D^{[m+1]}
 \quad\Longrightarrow\quad E\cong_D D^{[1]}.
\end{equation}
\end{lemma}

\iffalse

Statement~(4.5) on page~329 of
Abhyankar--Heinzer--Eakin~\cite{AHE72} says that if $D$ is a UFD,
then $D[z]$ is $D$-invariant.  With the definition in their
Remark~(4.4), this means: for a $D$-algebra $E$ and any $m\geq1$,
\begin{equation}%\label{eq:ahe-relative}
 E^{[m]}\cong_D D^{[m+1]}
 \quad\Longrightarrow\quad E\cong_D D^{[1]}.
\end{equation}
The original paper numbers this assertion (4.5), without labelling it
``Theorem.''

\fi

\begin{corollary}[Stable cancellation for a torus times an affine line]
\label{cor:torus-line-cancellation}
Let $W$ be an affine complex variety and let $a,m\geq0$.  If
\[
 W\times\mathbb{A}^m\cong(\mathbb{G}_m)^a\times\mathbb{A}^{m+1},
\]
then
\[
 W\cong(\mathbb{G}_m)^a\times\mathbb{A}^1.
\]
\end{corollary}

\begin{proof}
Put $R=\cO(W)$ and identify
\[
 R[u_1,\ldots,u_m]
 \cong\mathbb{C}[t_1^{\pm1},\ldots,t_a^{\pm1},x_0,\ldots,x_m]
\]
through the assumed isomorphism.  
As in the proof of Proposition \ref{prop:torus-cancellation}, the
Laurent-coordinate subring
$S=\mathbb{C}[t_1^{\pm1},\ldots,t_a^{\pm1}]$ lies in $R$, since its
generators are units.  The displayed isomorphism is therefore an
$S$-algebra isomorphism
\[
 R^{[m]}\cong_S S^{[m+1]}.
\]
The ring $S$ is a UFD.  Applying Lemma \ref{lem:cancellation_for_line_over_UFD} gives
$R\cong_S S^{[1]}$, which proves the assertion.
\end{proof}

\subsection*{Acknowledgements}
\addtocontents{toc}{\protect\setcounter{tocdepth}{1}}

The author would like to thank his PhD advisor, Vivek Shende, for his interest and for suggesting this direction of research. 
He also thanks Chenglong Yu, Penghui Li, Mingchen Xia, Yu Zhao, Mirko Mauri, Junliang Shen, Changjian Su, and Jia-Choon Lee for helpful discussions and comments.
Furthermore, he is very grateful to David Nadler, Lenhard Ng, Shing-Tung Yau, and Shaoyuan Zheng for their valuable support during his job search.
The initial version of this work was completed while the author was a postdoctoral researcher at IMS, CUHK.
During the revision, he benefited greatly from visits hosted by Changjian Su at YMSC and by Song-Yan Xie at AMSS, CAS.
Finally, he thanks the anonymous referee for valuable comments and suggestions.

\end{document}